\documentclass[11pt,reqno]{amsart}
\usepackage{amsmath,amsfonts,amssymb,amscd,amsthm,bbm}
\usepackage{braket}
\usepackage{xcolor}
\allowdisplaybreaks[4]
\numberwithin{equation}{section}
\usepackage{appendix}

\usepackage{newtxtext,newtxmath}
\usepackage{microtype}

\usepackage[numbers,sort&compress]{natbib}
\usepackage[colorlinks=true, linkcolor=blue, urlcolor=blue, citecolor=blue]{hyperref}

\newcommand\R{\mathbb R}

\newcommand\mbb\mathbb
\newcommand\mbf\mathbf
\newcommand\mcal\mathcal
\newcommand\mfrak\mathfrak

\newcommand\msf\mathsf
\renewcommand\a\alpha
\renewcommand\b\beta
\newcommand\g\gamma
\newcommand\G\Gamma
\renewcommand\d\delta
\newcommand\D\Delta
\newcommand\e\varepsilon
\newcommand\z\zeta
\renewcommand\t\theta
\newcommand\Th\Theta
\renewcommand\o\omega
\renewcommand\O\Omega
\newcommand\mr\mathring
\newcommand\ub\underbrace
\newcommand\pa\partial
\newcommand\n\nabla
\newcommand\fa\forall
\newcommand\ex\exists
\newcommand\wk\rightharpoonup
\newcommand\os\overset
\newcommand\us\underset
\newcommand\sr\stackrel
\newcommand\Ot\Leftarrow
\newcommand\To\Rightarrow
\newcommand\map\mapsto
\newcommand\ot\leftarrow
\newcommand\lot\longleftarrow
\newcommand\lto\longrightarrow
\newcommand\tot\leftrightarrow
\newcommand\ltot\longleftrightarrow
\newcommand\sm\backslash
\renewcommand\Cup\bigcup
\renewcommand\Cap\bigcap
\newcommand\sub\subset
\newcommand\Sub\Subset
\newcommand\sne\subsetneq
\newcommand\bus\supset
\newcommand\Bus\Supset

\newcommand\eq\equiv
\newcommand\ox\otimes
\newcommand\Ox\bigotimes
\newcommand\pl\oplus
\newcommand\Pl\bigoplus
\newcommand\x\times
\renewcommand\c\circ
\newcommand\q\quad
\renewcommand\l\left
\renewcommand\r\right
\newcommand\fr\frac

\newtheorem{Thm}{Theorem}[section]
\newtheorem{Lem}[Thm]{Lemma}
\newtheorem{Cor}[Thm]{Corollary}
\newtheorem{Prop}[Thm]{Proposition}

\newtheorem{Def}[Thm]{Definition}
\newtheorem{Rem}[Thm]{Remark}

\title[Segregated Solutions to Nonlinear Schr\"{o}dinger Systems]
{Segregated Solutions to Critical Elliptic Systems in High Dimensions ($N \geq 5$)}
\author{Zijuan Gao, Qing Guo, Chengxiang Zhang}
\address[]{School of Mathematical Sciences, 
	Beijing Normal University, Beijing 100875, P. R. China}
\email{gaozijuan3012@163.com}	
\address[]{College of Science, Minzu University of China, Beijing 100081, China}
\email{guoqing0117@163.com}
\address[]{Laboratory of Mathematics and Complex Systems (Ministry of Education), 
	School of Mathematical Sciences, 
	Beijing Normal University, Beijing 100875, P. R. China}
\email{zcx@bnu.edu.cn}
\date{\today}
\begin{document}
	\maketitle
	
	\begin{abstract}
		We study the existence of multiple segregated solutions to the critical
		coupled Schr\"odinger system
		\begin{equation*}
			\begin{cases}
				-\Delta u_{1} = K_1(\lvert y\rvert) \lvert u_{1}\rvert^{2^{*}-2}u_{1}+\beta \lvert u_{2}\rvert^{\frac{2^{*}}{2}}\lvert u_{1}\rvert^{\frac{2^{*}}{2}-2}u_{1},  &~y\in \mathbb R^N,\\[3mm]
				-\Delta u_{2} = K_2(\lvert y\rvert) \lvert u_{2}\rvert^{2^{*}-2}u_{2}+\beta \lvert u_{1}\rvert^{\frac{2^{*}}{2}}\lvert u_{2}\rvert^{\frac{2^{*}}{2}-2}u_{2},  &~y\in\mathbb R^N,\\[3mm]
				u_{1},u_{2}\geq0,~u_{1},u_{2}\in C_{0}(\mathbb R^{N})\cap D^{1,2}(\R^N),
			\end{cases}
		\end{equation*}
		with $N \geq 5$, $2^* = \frac{2N}{N-2}$, radial potentials $K_1, K_2 > 0$,
		and repulsive coupling $\beta < 0$.
		Under the assumption that $K_1$ and $K_2$
		attain local maxima at distinct radii $r_0 \ne \rho_0$ with
		precise asymptotic expansions near these points, we prove the existence of
		infinitely many non-radial segregated solutions $(u_{1,k}, u_{2,k})$ for all
		sufficiently large integers $k$. These solutions exhibit multiple bumps
		concentrating on two separate circles of radius $r_0$
		and $\rho_0$ respectively. Moreover, each component develops a
		``dead core'' near the concentration points of the other. The proof
		overcomes the sublinear and non-smooth nature of the coupling term
		($2^*/2 -1 < 1$) by constructing a tailored complete metric space and
		combining a finite-dimensional reduction with a novel tail minimization
		argument.
	\end{abstract}
	
	\small{
		\keywords {\noindent {\bf Keywords:} {Critical systems; Sublinear coupling; Segregated solutions;
				Reduction method; Infinitely many solutions.}}}
	\smallskip
	

\section{Introduction}

We study multiple segregated solutions to the following  nonlinear
critical Schr\"{o}dinger system
\begin{equation}\label{eq0}
	\begin{cases}
		-\Delta u_{1} = K_1(\lvert y\rvert) \lvert u_{1}\rvert^{2^{*}-2}u_{1}+\beta \lvert u_{2}\rvert^{\frac{2^{*}}{2}}\lvert u_{1}\rvert^{\frac{2^{*}}{2}-2}u_{1},  &~y\in \mathbb R^N,\\[3mm]
		-\Delta u_{2} = K_2(\lvert y\rvert) \lvert u_{2}\rvert^{2^{*}-2}u_{2}+\beta \lvert u_{1}\rvert^{\frac{2^{*}}{2}}\lvert u_{2}\rvert^{\frac{2^{*}}{2}-2}u_{2},  &~y\in\mathbb R^N,\\[3mm]
		u_{1},u_{2}\geq0,~u_{1},u_{2}\in C_{0}(\mathbb R^{N})\cap D^{1,2}(\R^N),
	\end{cases}
\end{equation}
where $N\geq 5$, $2^{*}=\frac{2N}{N-2}$, $K_i(r)>0$ ($i=1,2$) are radial
potentials, and $\beta\in\mathbb R$ is a coupling constant. System \eqref{eq0} arises when studying the solitary wave solutions of the following time-dependent coupled nonlinear Schr\"odinger equations
\begin{equation}\label{eq1}
	\begin{cases}
		-\mathrm{i} \frac{\partial\Phi_{j}}{\partial t}= \Delta \Phi_{j}+ \mu_j  \lvert\Phi_{j}\rvert^{p-2}\Phi_j + \sum_{i \neq j}  \beta_{ij} |\Phi_{i}|^{\frac{p}{2}}  \lvert\Phi_{j}\rvert^{\frac{p}{2}-2} \Phi_j,  &~y\in\mathbb R^N,~ t>0, \\[3mm]
		\Phi_{j}(y,0)= \phi_{j}(y), &~\ j=1,\dots,n,
	\end{cases}
\end{equation}
where $\mu_j (y) > 0$, $\beta_{ij}$ are coupling constants and the exponent $p \in (2, 2^*]$. The system \eqref{eq1} arises
in many physical problems, particularly in nonlinear optics. Here, $\Phi_{j}$ represents the $j$th component of a beam in Kerr-like
photorefractive media. The condition  $\mu_j>0$ shows that there is an attractive intraspecies interaction between the atoms inside each component. The interaction between two states is determined by the sign of the scattering length $\beta=\beta_{ij}$, see \cite{AA61} and references and therein. A positive $\beta$ implies an attractive interaction, causing the components of a vector solution to come together. Conversely, a negative $\beta$ indicates the repulsive interaction, causing the components to repel each other and forming the phase separations.

To obtain the solitary wave solutions of system \eqref{eq0}, one sets $\Phi_{j}(y,t) = e^{\mathrm{i}\lambda_{j}t}u_j (y)$, reducing the system to the  following $n$-coupled nonlinear Schr\"odinger equations
\begin{equation}\label{eq2}
	\begin{cases}
		\Delta u_j-\lambda_{j}u_{j} +\mu_j  \lvert u_{j}\rvert^{p-2}u + \sum_{i \neq j}  \beta_{ij} |u_{i}|^{\frac{p}{2}} \lvert u_{j}\rvert^{\frac{p}{2}-2}u_j =0,  &~y\in\mathbb R^N,  \\[3mm]
		u_{j}(y)\rightarrow 0,  ~\text{as}~ \lvert y\rvert\rightarrow\infty, & j=1,\dots,n.
	\end{cases}
\end{equation}

In recent years, significant mathematical studies have been conducted on nonlinear Schr\"odinger equations with critical exponents, for example, \cite{AC06,DWW10,LW05,LW005,LW08,S07,TV09,TW11,WY12}.
Del Pino, Musso, Pacard and Pistoia \cite{DMPP11,DMPP13} developed sequences of sign-changing solutions with high energy that concentrate along certain special sub-manifolds of $\mathbb R^{N}$. In \cite{GLW14}, Guo, Li and Wei demonstrated the existence of infinitely many positive non-radial solutions to the coupled elliptic system. In \cite{PW13}, an unbounded sequence of non-radial positive vector solutions of segregated type was constructed when $p<\frac{2N}{N-2}$.
In \cite{PV}, Pistoia and Vaira found positive non-radial solutions for $N$-coupled systems with weak interspecies forces in $N=2$ and $N=3$. 
For more information on the subcritical case of the system, we refer to \cite{BDW10,DWW10,LW05,LW08,PWW19,TW11}.
Recently, Chen, Medina and Pistoia \cite{CMP} studied the existence of segregated solutions for a critical elliptic system with a small interspecies repulsive force in $\mathbb R^{4}$.
When $n=1$ and $p=2^*$, system \eqref{eq2} reduces to the scalar equation
\begin{equation}\label{eq3}
	\begin{cases}
		-\Delta u = K(y) u^{\frac{N+2}{N-2}},~y\in \mathbb R^N,\\[3mm]
		u\in D^{1,2}(\mathbb R^{N}).
	\end{cases}
\end{equation}
This equation also arises via stereographic projection from the prescribed curvature
problem on the sphere $S^N$:
\begin{equation}\label{eq4}
	-\Delta_{S^{N}} u +\frac{N(N-2)}{2}u= \tilde{K} u^\frac{N+2}{N-2}~~\text{on}~S^N,
\end{equation}
where $\tilde{K}$ is positive and rotationally symmetric. Equations \eqref{eq3} and \eqref{eq4} are deeply
rooted in conformal geometry and has attracted extensive research; see \cite{AAJ99, AJ91, CNY02, CY91, CL01, DN85, GMPY, L1993, L96, L98, LW18, NY01}.  
For \eqref{eq3},
Wei and Yan \cite{WY10}  proved  the existence of infinitely many solutions to the prescribed scalar curvature problem \eqref{eq3} using the reduction argument. On the other hand, Y. Li proved in \cite{L93} that the system \eqref{eq3} has infinitely many solutions if $K(y)$ is periodic, while similar result was obtained in \cite{Y00} if $K(y)$ has a sequence of strict local maximum points approaching infinitely. 

\smallskip
The goal of this paper is to investigate the high-dimensional critical case with $p=\frac{2N}{N-2}$
and the coupling coefficient $\beta_{ij} <0$ (representing a repulsive interaction) in system \eqref{eq2}. 
To present our main results, we establish the following assumptions on $K_1, K_2$ are bounded:
\smallskip

\smallskip

\noindent{\bf ($\mathbb K$):} There exist constants \(r_{0}>0\), \(\rho_{0}>0\)
with \(\rho_{0}\neq r_{0}\), exponents \(m\in[2,N-2)\), and positive constants
\(c_{0,1},c_{0,2},\theta_{1},\theta_{2},\delta>0\) such that for
\(r\in(r_{0}-\delta,\,r_{0}+\delta)\) and
\(\rho\in(\rho_{0}-\delta,\,\rho_{0}+\delta)\),
\[
\begin{aligned}
	K_1(r) &= 1 - c_{0,1}\lvert r-r_{0}\rvert^{m}
	+ O\big(\lvert r-r_{0}\rvert^{m+\theta_{1}}\big),\\
	K_2(\rho) &= 1 - c_{0,2}\lvert \rho-\rho_{0}\rvert^{m}
	+ O\big(\lvert \rho-\rho_{0}\rvert^{m+\theta_{2}}\big).
\end{aligned}
\]

The primary results are as follows.
\begin{Thm}\label{th1}
	Suppose $N\geq5$, assumption $(\mathbb K)$ holds, and $\beta<0$. Then there
	exists an integer $k_0>0$ such that for every integer $k\geq k_0$, the system
	\eqref{eq0} admits a non-radial segregated solution $(u_{1,k},u_{2,k})$
	with 
	\[
	\lim_{k\to\infty}\max_{\mathbb R^N} u_{1,k} =+\infty,
	\qquad
	\lim_{k\to\infty}\max_{\mathbb R^N} u_{2,k}=+\infty.
	\]
	
\end{Thm}
Consequently, the system has infinitely many segregated   solutions.
Let us outline the main idea in the proof of Theorem \ref{th1}. We additionally establish an alternative form of segregated solutions. In other words, segregated solutions display multiple bumps near infinity, with the bumps of $u_{1}$ and $u_{2}$ occurring on different circles. To illustrate this phenomenon, we first introduce some notations.
\smallskip

Fix a positive integer $k \geq k_{0}$, where $k_{0}$ is a large integer to be determined. Set $$\mu=k^{\frac{1}{\tau_1}}$$ to be the scaling parameter, where $\tau_1= \frac{N-2-m}{N-2}$. Using the transformation $u(y)\mapsto \mu^{-\frac{N-2}{2}}u(\frac{y}{\mu})$, the system \eqref{eq0} becomes

\begin{equation}\label{eq5}
	\begin{cases}
		-\Delta u_{1} = K_1\big(\frac{\lvert y \rvert}{\mu}\big) \lvert u_{1}\rvert^{2^{*}-2}u_{1}+\beta \lvert u_{2}\rvert^{\frac{2^{*}}{2}}\lvert u_{1}\rvert^{\frac{2^{*}}{2}-2}u_{1},  &~y\in \mathbb R^N,\\[3mm]
		-\Delta u_{2} = K_2\big(\frac{\lvert y \rvert}{\mu}\big) \lvert u_{2}\rvert^{2^{*}-2}u_{2}+\beta \lvert u_{1}\rvert^{\frac{2^{*}}{2}}\lvert u_{2}\rvert^{\frac{2^{*}}{2}-2}u_{2},  &~y\in\mathbb R^N,\\[3mm]
		u_{1},u_{2}\geq0,~u_{1},u_{2}\in C_{0}(\mathbb R^{N})\cap D^{1,2}(\R^N).
	\end{cases}
\end{equation}

It is well known that all positive solutions of
\begin{equation*} 
	-\Delta u=u^{\frac{N+2}{N-2}},~u>0~\text{in}~ \mathbb R^{N} 
\end{equation*}
are given by the family
\begin{equation*} 
	\Set{(N(N-2))^{\frac{N-2}{4}}\Big(\frac{\lambda}{1+\lambda^{2}\lvert y-x \rvert^{2}}\Big)^{\frac{N-2}{2}} |~\lambda>0,~x\in \mathbb R^{N}}.
\end{equation*}

Write $y=(y',y'')$ with $y'\in\mathbb R^2$, $y''\in\mathbb R^{N-2}$. Define

\begin{equation*}
	\begin{aligned}
		H_{s} =  
		& \Big\{u~\text{is a measurable function on}~\R^N  \mid~u~
		\text{is even in} \; y_{d},~d=2,\dots,N,\\
		& u(r\cos\theta,r\sin\theta,y^{\prime\prime})= \displaystyle u\Big(r\cos\Big(\theta+\frac{2\pi j}{k}\Big),r\sin\Big(\theta+\frac{2\pi j}{k}\Big),y^{\prime\prime}\Big)\Big\}
	\end{aligned}
\end{equation*}
and set
\begin{equation*}
	 H_{s_{0}} =   H_{s} \cap  C_{0}(\mathbb R^{N}).
\end{equation*}
Define
\begin{equation*}	
	\mathbb H = H_{s_{0}}\times H_{s_{0}},
\end{equation*}
equipped with the norm
$\lVert (u_1,u_2)\rVert_\infty
= \lVert u_1\rVert_{L^\infty(\mathbb R^N)} + \lVert u_2\rVert_{L^\infty(\mathbb R^N)}$.

For $i=1,\dots,k$, denote
\begin{equation}\label{1.7}
		x^{i} = r \Big(\cos\frac{2(i-1)\pi}{k}, ~\sin\frac{2(i-1)\pi}{k},0\Big), \quad 
		y^{i} = \rho \Big(\cos\frac{2(i-1)\pi}{k}, ~\sin\frac{2(i-1)\pi}{k},0\Big)
\end{equation}
where $r\in \Big[r_{0} \mu-1,r_{0} \mu +1\Big],~\rho \in \Big[\rho_{0} \mu-1,\rho_{0} \mu+1\Big]$ and $0$ is the zero in $\mathbb R^{N-2}$.

\smallskip

We set
\begin{equation*} 
	U_{r,\lambda}(y)=\sum_{i=1}^{k}U_{x^{i},\lambda}(y),\quad V_{\rho,\nu}(y)=\sum_{i=1}^{k}V_{y^{i},\nu}(y),
\end{equation*}
where 
\begin{equation*} 
	U_{x^{i},\lambda}(y) =(N(N-2))^{\frac{N-2}{4}}\Big(\frac{\lambda}{1+\lambda^{2}\lvert y-x^{i} \rvert^{2}}\Big)^{\frac{N-2}{2}},
\end{equation*}

\begin{equation*} 
	V_{y^{i},\nu}(y) =(N(N-2))^{\frac{N-2}{4}}\Big(\frac{\nu}{1+\nu^{2}\lvert y-y^{i} \rvert^{2}}\Big)^{\frac{N-2}{2}},
\end{equation*}
for $ x^{i}$ and $ y^{i}$ are defined in \eqref{1.7}.

\smallskip
We will prove Theorem \ref{th1} by verifying the following result.

\begin{Thm}\label{th2}
	Suppose $N\geq5$, assumption $(\mathbb K)$ holds, and $\beta<0$. Then there exist some small $\bar{\theta}>0$, some constants $\lambda_{0}>0,\nu_{0}>0$ and an integer
	$k_{0}>0$ such that for any integer $k\geq k_{0}$, the system \eqref{eq5} 
	has a solution $(u_{1,k},u_{2,k})$ of the form
	\begin{equation*}
		\big(u_{1,k}(y),u_{2,k}(y)\big)=\big( U_{r_{k},\lambda_{k}}(y)+\varphi_{k}, V_{\rho_{k},\nu_{k}}(y)+\psi_{k}\big),
	\end{equation*}
	where $(\varphi_{k},\psi_{k}) \in \mathbb H $,~$\lambda_{k}\in\Big[\lambda_{0}-\frac{1}{\mu^{\frac{3}{2}\bar\theta}},\lambda_{0}+\frac{1}{\mu^{\frac{3}{2}\bar\theta}}\Big]
,~\nu_{k}\in\Big[\nu_{0}-\frac{1}{\mu^{\frac{3}{2}\bar\theta}},\nu_{0}+\frac{1}{\mu^{\frac{3}{2}\bar\theta}}\Big]$,
$r_{k} \in \Big[r_{0} \mu-\displaystyle\frac{1}{\mu^{\bar{\theta}}},r_{0} \mu+\displaystyle\frac{1}{\mu^{\bar{\theta}}}\Big],~\rho_{k} \in \Big[\rho_{0} \mu-\displaystyle\frac{1}{\mu^{\bar{\theta}}},\rho_{0} \mu+\displaystyle\frac{1}{\mu^{\bar{\theta}}}\Big]$
and $\lVert (\varphi_{k},\psi_{k}) \rVert_{\infty}\rightarrow 0 ~~\text{as}~~ k\rightarrow\infty. $
	
\end{Thm}

Moreover, for $N\geq5$, these segregated solutions exhibit a ``dead core''
phenomenon \cite{PP07}: each component vanishes in neighborhoods of the other's
concentration points.

\begin{Thm}\label{th3}
	For any $\vartheta \in(0, 1-\tau_1]\cap \big(0,
	\frac{(N-4)(N-2-\tau_1)}{(N-2)^2}\big)$, there exists $k_\vartheta$ such that
	for all $k\geq k_\vartheta$, the solution $(u_{1,k},u_{2,k})$ from Theorem
	\ref{th2} satisfies
	\[
	u_{1,k}=0\ \text{in}\ \bigcup_{j=1}^k B_{\mu^\vartheta}(y^j),\quad
	u_{2,k}=0\ \text{in}\ \bigcup_{j=1}^k B_{\mu^\vartheta}(x^j).
	\]
\end{Thm}

When $N\geq5$, the coupling exponent $\frac{2^{*}}{2}-1=\frac{2}{N-2}<1$ introduces significant challenges for classical reduction methods. These difficulties arise from the non-smooth and sublinear nature of the coupling nonlinearities. A critical obstacle emerges when applying the fixed-point theorem in the complement of the kernel space to demonstrate the existence of solutions, as this requires reducing the problem to a finite-dimensional framework. To address this, we adopt a  strategy which generalizes and simplifies the approach in \cite{WZZ23}.

The sublinearity-induced singularities occur where the solution vanishes. This motivates analyzing the behavior of functions away from concentration points first. We partition the space into inner regions (near concentration points) and an outer region (far from concentration points), based on the distribution of these points (see Section \ref{sec:2} for details).

To apply the fixed-point theorem, we introduce a set of functions with favorable decay properties. In prior works \cite{DMPP03,WY10}, weighted spaces were employed to ensure sufficient regularity for contraction mapping. However, in the sublinear regime, constructing such spaces becomes problematic due to the inherent non-smoothness. Our key innovation lies in designing a tailored functional space specifically suited for the contraction mapping theorem.

We construct a complete metric space using a tail minimization technique, inspired by that used in variational gluing methods (e.g., \cite{Coti92,BT14}), to derive sharp decay estimates for approximate solutions. In this space, each function is defined through solutions to boundary value problems in the outer region, with boundary conditions determined by prescribed inner functions. Unlike the weighted spaces or tailoring techniques in \cite{Coti92,BT14}, we rigorously establish a priori decay estimates not only for the functions themselves but also for their deviations from the limiting profiles $U_{r, \lambda}, V_{\rho,\nu}$, and  the differences between distinct functions within the space. We emphasize that the ``dead core" estimates are crucial for establishing the decay of differences between distinct functions, as they provide quantitative separation bounds in the interaction regions.

By restricting the fixed-point problem to this metric space and leveraging the derived decay estimates, we successfully reduce the analysis to a finite-dimensional setting.

\medskip

The structure of this paper is organized as follows. In Section \ref{sec:2}, we introduce some basic estimates necessary for proving Theorem \ref{th2}. In Section \ref{sec:3}, we formulate a minimization problem for the outer region and establish various a priori estimates for the solution to the exterior problem. In Section \ref{sec:4}, we demonstrate the existence of solutions using the fixed-point theorem in the orthogonal complement of the kernel space, reducing the problem to a finite-dimensional one. In Section \ref{sec:5} concludes the proof of the theorems by solving the finite-dimensional problem. In the Appendix, we present some basic estimates used in the article.

\section{Preliminaries}\label{sec:2}

Let
\begin{equation*}
	\sigma_{k}:= k^{\frac{1}{N-2}}.
\end{equation*}

Set
\begin{equation*}
	P_{l}= \bigcup_{i=1}^{k} B_{\sigma^{l}_{k}}(x^{i}),~Q_{l}=\bigcup_{i=1}^{k} B_{\sigma^{l}_{k}}(y^{i}),~l=1,2.
\end{equation*}

Let $\chi_{0}\in C^{1}_{0}(\mathbb R^{N})$ be a fixed truncation function such that $\chi_{0}=1$ in $B_{1}(0)$ and $\chi_{0}=0$ in $\mathbb R^{N}\setminus B_{2}(0)$.

Denote
\begin{equation*}
	Y_{1}=\sum_{i=1}^k \chi_{0}(\cdot-x^{i})\frac{\partial U_{x^{i},\lambda}}{\partial r},\quad
	Y_{2}=\sum_{i=1}^k \chi_{0}(\cdot-x^{i})\frac{\partial U_{x^{i},\lambda}}{\partial \lambda},~i=1,\dots,k,
\end{equation*}
\begin{equation*}
	Z_{1}=\sum_{i=1}^k\chi_{0}(\cdot-y^{i})\frac{\partial V_{y^{i}, \nu}}{\partial \rho},\quad Z_{2}=\sum_{i=1}^k\chi_{0}(\cdot-y^{i})\frac{\partial V_{y^{i}, \nu}}{\partial\nu},~i=1,\dots,k,
\end{equation*}
where $ x^{i}$ and $ y^{i}$ are defined in \eqref{1.7}.

Define
\begin{equation*}
	\mathbb E=\mathbb E_{r,\rho,\lambda,\nu}=E_{r,\lambda}\times E_{\rho,\nu},
\end{equation*}
where
\begin{equation*}
	\begin{aligned}
		E_{r,\lambda}
		=&\Big\{u\in H_{s_{0}} \mid \int_{\mathbb R^{N}} Y_{1}u=0~\text{and}~\int_{\mathbb R^{N}} Y_{2}u=0\Big\},\\
		E_{\rho,\nu}
		=&\Big\{v\in H_{s_{0}} \mid \int_{\mathbb R^{N}} Z_{1} v=0~\text{and}~\int_{\mathbb R^{N}}  Z_{2}v=0\Big\}.
	\end{aligned}
\end{equation*}

Let
$$\lambda\in [\gamma_{1},\gamma_{2}]~~\text{and}~~\nu\in [\gamma_{3},\gamma_{4}],$$
where $\gamma_{2}>\gamma_{1}>0$ and $\gamma_{4}>\gamma_{3}>0$ are some constants.

For $i=1,\dots,k$, denote
\begin{equation*}
	\Omega_{i}
	= \Set{z=(z^{\prime},z^{\prime\prime}) \in \mathbb R^{2}\times \mathbb R^{N-2} | \Big \langle \frac{z^{\prime}}{\lvert z^{\prime} \rvert}, \frac{(x^i)^{\prime}}{\lvert (x^i)^{\prime} \rvert} \Big \rangle\geq \cos\frac{\pi}{k} }
	=\Set{z    |~\Big \langle \frac{z^{\prime}}{\lvert z^{\prime} \rvert}, \frac{(y^i)^{\prime}}{\lvert (y^i)^{\prime} \rvert} \Big \rangle\geq \cos\frac{\pi}{k}}.
\end{equation*}

Let $h_1, h_2 \in L^\infty(\mathbb{R}^N) \cap H_s$, and denote by ${1}_{P_2}$ and ${1}_{Q_2}$ the characteristic functions of the sets $P_2$ and $Q_2$, respectively.  
Consider the following linear problems:
\begin{equation} \label{eq7}
	\begin{aligned}
		\begin{cases}
			-\Delta \varphi-(2^{*}-1)K_1\big(\frac{\lvert y \rvert}{\mu}\big)U^{2^{*}-2}_{r,\lambda}\varphi=1_{P_2}h_{1}+\sum_{j=1}^{2}c_{j} Y_{j}, \text{ in }\R^N,
			\\
			\varphi\in E_{r,\lambda},
		\end{cases}
	\end{aligned}
\end{equation}
and 
\begin{equation} \label{eq7.}
	\begin{aligned}
		\begin{cases}
			-\Delta \psi-(2^{*}-1)K_2\big(\frac{\lvert y \rvert}{\mu}\big)V^{2^{*}-2}_{\rho,\nu}\psi=1_{Q_2}h_{2}+\sum_{j=1}^{2}d_{j} Z_{j}, \text{ in }\R^N,\\
			\psi \in E_{\rho,\nu},
		\end{cases}
	\end{aligned}
\end{equation}
for some real numbers  $c_{j}$ and $d_{j}$. 
In what follows, we use notation $\langle u_{1},u_{2} \rangle=\int_{\mathbb R^{N}} u_{1}u_{2}$.
\begin{Lem}\label{lem1}
	
	Let $h_{1,k}, h_{2,k} \in L^\infty(\R^N)\cap H_s$.  
	Suppose $(\varphi_k, c_{1,k}, c_{2,k})$ and $(\psi_k, d_{1,k}, d_{2,k})$ are the solutions to \eqref{eq7} and \eqref{eq7.} with $h_1 = h_{1,k}$ and $h_2 = h_{2,k}$, respectively. Then the following hold:
	\begin{enumerate} 
		\item[(i)] If $\sigma_k^4 \|h_{1,k}\|_{L^\infty(P_2)} \to 0$ as $k \to \infty$, then
		\[
		\|\varphi_k\|_{L^\infty(\mathbb{R}^N)} + \sum_{j=1}^2 |c_{j,k}| \;\to\; 0 \quad \text{as } k \to \infty.
		\]
		
		\item[(ii)] If $\sigma_k^4 \|h_{2,k}\|_{L^\infty(Q_2)} \to 0$ as $k \to \infty$, then
		\[
		\|\psi_k\|_{L^\infty(\mathbb{R}^N)} + \sum_{j=1}^2 |d_{j,k}| \;\to\; 0 \quad \text{as } k \to \infty.
		\]
	\end{enumerate}
\end{Lem}
\begin{proof}
	We only prove (i).
	By contradiction, we assume that there exist a constant $c^{\prime}>0$ such that as
	$k\rightarrow\infty$, for $h_{1}=h_{1,k}$, $\lambda=\lambda_{k}\in [\gamma_{1},\gamma_{2}]$, $r=r_{k}\in \Big[r_{0} \mu-1,r_{0} \mu +1\Big]$, with $\sigma^{4}_{k}\lVert h_{1,k}\rVert_{L^{\infty}(P_{2})}\rightarrow 0$ 
	there holds 
	\begin{equation}\label{eq2.22222}
		\sum_{j=1}^2 |c_{j,k}|+\lVert \varphi_{k}\rVert_{L^{\infty}(\mathbb R^{N})}\geq 2c^{\prime}>0.
	\end{equation} 
	For simplicity, we drop the subscript $k$.
	
	We first estimate $c_{l},l=1,2$.
	Multiplying \eqref{eq7} by $\frac{\partial U_{x^{1},\lambda}}{\partial y_{1}}$ and integrating, we see that  
	\begin{equation}\label{eq11}
		 -c_{1}(\int\chi_0\left|\frac{\partial U_{0,\lambda}}{\partial y_{1}} \right|^2+o(1))=\Big\langle-\Delta \varphi-(2^{*}-1)K_1\Big(\frac{\lvert y \rvert}{\mu}\Big)U^{2^{*}-2}_{r,\lambda}\varphi,\frac{\partial U_{x^{1},\lambda}}{\partial y_{1}} \Big\rangle-\langle 1_{P_2} h_{1},\frac{\partial U_{x^{1},\lambda}}{\partial y_{1}}\rangle.
	\end{equation}
	
	From Lemma \ref{lemB1}, we have
	\begin{equation}\label{eq2.5}
		\begin{aligned}	
			&\lvert \langle 1_{P_2} h_{1},\frac{\partial U_{x^{1},\lambda}}{\partial y_{1}}\rangle\rvert 
			\leq C \lVert h_{1} \rVert_{L^{\infty}(P_{2})}\sum_{i=1}^{k} \int_{B_{\sigma^{2}_{k}}(x^{i})} \frac{1}{(1+\lvert y-x^{1}\rvert)^{N-2}}dy\\
			\leq& C \lVert h_{1} \rVert_{L^{\infty}(P_{2})} \int_{B_{\sigma^{2}_{k}}(x^{1})} \frac{1}{(1+\lvert y-x^{1}\rvert)^{N-2}}dy+C \lVert h_{1} \rVert_{L^{\infty}(P_{2})}\sum_{i=2}^{k} \int_{B_{\sigma^{2}_{k}}(x^{i})} \frac{1}{(1+\lvert y-x^{1}\rvert)^{N-2}}dy\\
			\leq& C \sigma^{4}_{k}\lVert h_{1} \rVert_{L^{\infty}(P_{2})}.
		\end{aligned}
	\end{equation}
	
	On the other hand, we also have
	\begin{equation*}
		\begin{aligned}
			&\Big\langle-\Delta \varphi-(2^{*}-1)K_1\Big(\frac{\lvert y \rvert}{\mu}\Big)U^{2^{*}-2}_{r,\lambda}\varphi,\frac{\partial U_{x^{1},\lambda}}{\partial y_{1}} \Big\rangle\\
			=&\Big\langle-\Delta \frac{\partial U_{x^{1},\lambda}}{\partial y_{1}}-(2^{*}-1)K_1\Big(\frac{\lvert y \rvert}{\mu}\Big)U^{2^{*}-2}_{r,\lambda}\frac{\partial U_{x^{1},\lambda}}{\partial y_{1}},\varphi \Big\rangle\\
			=&(2^{*}-1)\Big\langle \Big(1-K_1\Big(\frac{\lvert y \rvert}{\mu}\Big)\Big)U^{2^{*}-2}_{r,\lambda}\frac{\partial U_{x^{1},\lambda}}{\partial y_{1}},\varphi \Big\rangle-(2^{*}-1)\Big\langle (U^{2^{*}-2}_{r,\lambda}-U^{2^{*}-2}_{x^{1},\lambda})\frac{\partial U_{x^{1},\lambda}}{\partial y_{1}},\varphi \Big\rangle\\
			=&:G_{1}+G_{2}.
		\end{aligned}
	\end{equation*}

	When $\lvert |y|-\mu r_{0}\rvert\leq \sqrt{\mu}$, we get
	\begin{equation*}
		\begin{aligned}
			\int_{\lvert |y|-\mu r_{0}\rvert\leq \sqrt{\mu}} \Big\lvert K_1\Big(\frac{\lvert y \rvert}{\mu}\Big)-1\Big\rvert U^{2^{*}-2}_{r,\lambda} \frac{1}{(1+\lvert y-x^{1}\rvert)^{N-2}} dy
			\leq \frac{C}{\sqrt{\mu}} \int_{\mathbb R^{N}} U^{2^{*}-2}_{r,\lambda} \frac{1}{(1+\lvert y-x^{1}\rvert)^{N-2}} dy
			\leq \frac{C}{\sqrt{\mu}}.
		\end{aligned}
	\end{equation*}
	If $\lvert |y|-\mu r_{0}\rvert\geq \sqrt{\mu}$, then 
	\begin{equation*}
		\lvert y-  x^{j}\rvert \geq \lvert |y|-\mu r_{0}\rvert-\lvert \lvert x^{j}\rvert-\mu r_{0}\rvert\geq\sqrt{\mu}-1\geq \frac{1}{2}\sqrt{\mu},
	\end{equation*}
	which implies that 
	\begin{equation*}
		\begin{aligned}
			\int_{\lvert |y|-\mu r_{0}\rvert\geq \sqrt{\mu}} \Big\lvert K_1\Big(\frac{\lvert y \rvert}{\mu}\Big)-1\Big\rvert U^{2^{*}-2}_{r,\lambda} \frac{1}{(1+\lvert y-x^{1}\rvert)^{N-2}} dy
			=o(1).
		\end{aligned}
	\end{equation*}
 
	Then we obtain
	\begin{equation}\label{k}
		G_{1}=\lVert \varphi \rVert_{L^{\infty}(\mathbb R^{N})} O\Big(\int_{\mathbb R^{N}} \Big\lvert K_1\Big(\frac{\lvert y \rvert}{\mu}\Big)-1\Big\rvert U^{2^{*}-2}_{r,\lambda} \frac{1}{(1+\lvert y-x^{1}\rvert)^{N-2}} dy \Big)=o(\lVert \varphi \rVert_{L^{\infty}(\mathbb R^{N})}).
	\end{equation}

	For $G_{2}$, by Lemma \ref{lemB1},  \eqref{tau1a}, and \eqref{key3'}, we have
	\begin{equation}\label{bu}
		\begin{aligned}
			G_{2}=& \lVert \varphi \rVert_{L^{\infty}(\mathbb R^{N})}O\Big(\int_{\mathbb R^{N}} U^{2^*-3}_{x^{1},\lambda}\sum_{i=2}^{k}U_{x^{i},\lambda}\frac{\partial U_{x^{1},\lambda}}{\partial y_{1}}+\int_{\mathbb R^{N}}(\sum_{i=2}^{k}U_{x^{i},\lambda})^{2^*-2}\frac{\partial U_{x^{1},\lambda}}{\partial y_{1}}\Big)\\
			=&\lVert \varphi \rVert_{L^{\infty}(\mathbb R^{N})}O\Big(\sum_{i=2}^{k}\frac{1}{\lvert x^{i}-x^{1}\rvert^{1+\frac{2\tau_1}{3}}}\Big(\int_{\mathbb R^{N}}\frac{1}{(1+\lvert y-x^{1}\rvert)^{N+1-\tau_1}}+\int_{\R^N}\frac{1}{(1+\lvert y-x^{i}\rvert)^{N+1-\tau_1}}\Big)\Big)\\
			=&o(\lVert \varphi \rVert_{L^{\infty}(\mathbb R^{N})}).
		\end{aligned}
	\end{equation}
 
	Then, from \eqref{eq11}, \eqref{eq2.5}, \eqref{k}, and \eqref{bu}, we obtain	
	\begin{equation}\label{a14}
		c_{1}=o(1+\lVert \varphi \rVert_{L^{\infty}(\mathbb R^{N})}).
	\end{equation}
	Similarly, 
	multiplying \eqref{eq7} by $\frac{\partial U_{x^{1},\lambda}}{\partial\lambda}$ and integrating, we can obtain
	\begin{equation}\label{a15}
		c_{2}=o(1+\lVert \varphi \rVert_{L^{\infty}(\mathbb R^{N})}).
	\end{equation}
	From \eqref{eq2.22222}, \eqref{a14} and \eqref{a15}, we know $\lVert \varphi \rVert_{L^{\infty}(\mathbb R^{N})}\geq c'>0$.
	
	Next we write \eqref{eq7} as
	\begin{equation*}
		\begin{aligned}
			\varphi(y)=&(2^{*}-1)\int_{\mathbb R^{N}} \frac{1}{\lvert z-y\rvert^{N-2}}K_1\Big(\frac{\lvert z \rvert}{\mu}\Big)U^{2^{*}-2}_{r,\lambda}(z)\varphi(z) dz\\
			&+\int_{\mathbb R^{N}} \frac{1}{\lvert z-y\rvert^{N-2}}\Big(1_{P_2}(z)h_{1}(z)+\sum_{j=1}^{2}c_{j} Y_{j}(z)\Big) dz
			.
		\end{aligned}
	\end{equation*}
	
	For $ z\in \Omega_{1}$, we get $\lvert z-x^{i}\rvert\geq\lvert z-x^{1}\rvert,~i=2,\dots,k$. Using Lemma \ref{lemB1} and \eqref{key3}, we obtain
	\begin{equation*}
		\begin{aligned}
			\sum_{i=2}^{k}\frac{1}{(1+\lvert z-x^{i}\rvert)^{N-2}}
			\leq & \frac{1}{(1+\lvert z-x^{1}\rvert)^{N-2-\tau_{1}}}\sum_{i=2}^{k}\frac{1}{\lvert x^{i}-x^{1}\rvert^{\tau_{1}}}
			\leq C \frac{1}{(1+\lvert z-x^{1}\rvert)^{N-2-\tau_{1}}},
		\end{aligned}
	\end{equation*}
	where $\tau_{1}=\frac{N-2-m}{N-2}$.
	
	Using \eqref{tau1a} and Lemma \ref{lemB2}, there holds
	\begin{equation*}
		\begin{aligned}
			&\Big\lvert (2^{*}-1)\int_{\mathbb R^{N}} \frac{1}{\lvert z-y\rvert^{N-2}}K_1\Big(\frac{\lvert z \rvert}{\mu}\Big)U^{2^{*}-2}_{r,\lambda}(z)\varphi(z) dz\Big\rvert\\
			\leq &C \lVert \varphi\rVert_{L^{\infty}(\mathbb R^{N})}  \int_{\R^N} \frac{1}{\lvert z-y\rvert^{N-2}} \Big(\sum_{i=1}^{k}\frac{1}{(1+\lvert z-x^{i}\rvert)^{N-2}}\Big)^{\frac{4}{N-2}}dz\\
			\leq &C \lVert \varphi\rVert_{L^{\infty}(\mathbb R^{N})}  \int_{\R^N} \frac{1}{\lvert z-y\rvert^{N-2}} \sum_{i=1}^{k}\frac{1}{(1+\lvert z-x^{i}\rvert)^{4-\frac{(6-N)^+}{N-2} \tau_{1}}}dz\\
			\leq &C \lVert \varphi\rVert_{L^{\infty}(\mathbb R^{N})} \sum_{i=1}^{k}\frac{1}{(1+\lvert y-x^{i}\rvert)^{2-\frac{ \tau_{1} }{N-2}}},
		\end{aligned}
	\end{equation*}
	where $\tau_{1}=\frac{N-2-m}{N-2}$.
	
	For $ y\in \Omega_{1}\setminus B_{2\sigma^{2}_{k}}(x^{1})$, we get
	\begin{equation*}
		\begin{aligned}
			\Big\lvert \int_{B_{\sigma^{2}_{k}}(x^{1})} \frac{1}{\lvert z-y\rvert^{N-2}} h_{1}(z) dz\Big\rvert\leq C\lVert h_{1}\rVert_{L^{\infty}(P_{2})} \int_{B_{\sigma^{2}_{k}}(x^{1})} \frac{1}{\sigma^{2(N-2)}_{k}}dz\leq C \sigma^{4}_{k}\lVert h_{1}\rVert_{L^{\infty}(P_{2})}.	
		\end{aligned}
	\end{equation*}
	For $ y\in B_{2\sigma^{2}_{k}}(x^{1})$, we also obtain
	\begin{equation*}
		\begin{aligned}
			\Big\lvert \int_{B_{\sigma^{2}_{k}}(x^{1})} \frac{1}{\lvert z-y\rvert^{N-2}} h_{1}(z) dz\Big\rvert\leq \lVert h_{1}\rVert_{L^{\infty}(P_{2})}\Big\lvert \int_{B_{3\sigma^{2}_{k}}(y)} \frac{1}{\lvert z-y\rvert^{N-2}} dz\Big\rvert\leq C \sigma^{4}_{k}\lVert h_{1}\rVert_{L^{\infty}(P_{2})}.	
		\end{aligned}
	\end{equation*}
	For $ y\in \Omega_{1}$ and $z\in B_{\sigma^{2}_{k}}(x^{i})$, there holds $\lvert z-y\rvert\geq\frac{1}{2}\lvert x^{i}-x^{1}\rvert,i=2,\dots,k$.
	By \eqref{key3'}, we get
	\begin{equation*}
		\begin{aligned}
			\Big\lvert\sum_{i=2}^{k}\int_{B_{\sigma^{2}_{k}}(x^{i})} \frac{1}{\lvert z-y\rvert^{N-2}} h_{1}(z) dz\Big\rvert\leq C\sum_{i=2}^{k}\frac{1}{|x^i-x^1|^{N-2}}\sigma^{2N}_{k}\lVert h_{1}\rVert_{L^{\infty}(P_{2})}\leq C \frac{1}{\mu^{m}}\sigma^{2N}_{k}\lVert h_{1}\rVert_{L^{\infty}(P_{2})}.	
		\end{aligned}
	\end{equation*}

	Therefore,
	\begin{equation*}
		\begin{aligned}
			&\Big\lvert \int_{\mathbb R^{N}} \frac{1}{\lvert z-y\rvert^{N-2}} 1_{P_2}(z)h_{1}(z) dz\Big\rvert
			=\Big\lvert \sum_{i=1}^{k}\int_{B_{\sigma^{2}_{k}}(x^{i})} \frac{1}{\lvert z-y\rvert^{N-2}} h_{1}(z) dz\Big\rvert\\
			\leq&\Big\lvert \int_{B_{\sigma^{2}_{k}}(x^{1})} \frac{1}{\lvert z-y\rvert^{N-2}} h_{1}(z) dz\Big\rvert+\Big\lvert\sum_{i=2}^{k}\int_{B_{\sigma^{2}_{k}}(x^{i})} \frac{1}{\lvert z-y\rvert^{N-2}} h_{1}(z) dz\Big\rvert\\
			\leq& C\sigma^{4}_{k}\lVert h_{1}\rVert_{L^{\infty}(P_{2})}+C\frac{1}{\mu^{m}}\sigma^{2N}_{k}\lVert h_{1}\rVert_{L^{\infty}(P_{2})}=o(1).
		\end{aligned}
	\end{equation*}
	By Lemma \ref{lemB2}, we also obtain
	\begin{equation}\label{a13}
		\begin{aligned}
			\Big\lvert \int_{\mathbb R^{N}} \frac{1}{\lvert z-y\rvert^{N-2}} Y_{j}(z) dz\Big\rvert
			\leq& C \sum_{i=1}^{k} \int_{\mathbb R^{N}} \frac{1}{\lvert z-y\rvert^{N-2}}\frac{1}{(1+\lvert z-x^{i}\rvert)^{N-2}}dz\\
			\leq& C \sum_{i=1}^{k}\frac{1}{(1+\lvert y-x^{i}\rvert)^{N-4}}\leq C.
		\end{aligned}
	\end{equation}
	
	Combining  \eqref{a14}, \eqref{a15} and \eqref{a13}, we obtain
	\begin{equation*}
		\int_{\mathbb R^{N}} \frac{1}{\lvert z-y\rvert^{N-2}}\sum_{j=1}^{2}c_{j} Y_{j}(z) dz=o(1+\lVert \varphi \rVert_{L^{\infty}(\mathbb R^{N})}).
	\end{equation*}
	Therefore, for $y\in \Omega_1$,
	\begin{equation}\label{eq12}
		\frac{\lvert\varphi(y)\rvert}{\lVert \varphi \rVert_{L^{\infty}(\mathbb R^{N})}} \leq C
		\sum_{i=1}^{k}\frac{1}{(1+\lvert y-x^{i}\rvert)^{2-\frac{\tau_{1}}{3}}} +o(1)\leq C\frac{1}{(1+\lvert y-x^{i}\rvert)^{2-\frac{4}{3}\tau_1}} +o(1).
	\end{equation}
	We obtain from \eqref{eq12} that there is $R>0$ independent of $k$  such that
	\begin{equation}\label{eq13}
		\lVert \varphi(y) \rVert_{L^{\infty}(B_{R}(x^{1}))}/\lVert \varphi \rVert_{L^{\infty}(\mathbb R^{N})}=1.
	\end{equation}
	But $\bar \varphi(y)=\varphi( y+x^{1})/\lVert \varphi \rVert_{L^{\infty}(\mathbb R^{N})}$ converges uniformly in any compact set to a solution $u$ of
	\[ 
		-\Delta u-(2^{*}-1)U_{0,\lambda}^{2^{*}-2}u=0 ~\text{in}~\mathbb R^{N},
	\]
	for some $\lambda\in[\gamma_1,\gamma_2]$.
	Since $\varphi\in H_{s_0}$, $u$ is even in $y_j$,$j=2,\dots,N$. Therefore,
	we have 
	$u\in \text{span}\{\frac{\partial U_{0,\lambda}}{\partial y_1},\frac{\partial U_{0,\lambda}}{\partial \lambda}\}$.
	However, since $ \int_{\R^N}u\chi_0 \frac{\partial U_{0,\lambda}}{\partial y_1}= \int_{\R^N}u\chi_0 \frac{\partial U_{0,\lambda}}{\partial \lambda}= 0$,   we conclude $ u = 0 $. This contradicts the uniform lower bound in \eqref{eq13}, completing the proof.  
\end{proof}

From Lemma \ref{lem1}, using the same argument as in the proof of Proposition 4.1 in \cite{DMPP03}, we can prove the following result:

\begin{Prop}\label{Prop2.2}
	There exists $k_0 > 0$ such that for all $k \geq k_0$, the following hold for some constant $C > 0$ independent of $k$:
	\begin{enumerate} 
		\item[(i)] For every $h_1 \in L^\infty(\mathbb{R}^N)$, the linear problem \eqref{eq7} admits a unique solution 
		$\varphi = L_1(h_1) \in E_{r,\lambda}\cap D^{1,2}(\R^N)$ together with unique constants $c_j = c_j(h_1) \in \mathbb{R}$, $j=1,2$, such that
		\[
		\|L_1(h_1)\|_{L^\infty(\mathbb{R}^N)} \leq C\,\sigma_k^4\, \|h_1\|_{L^\infty(P_2)},
		\qquad
		|c_j| \leq C\,\sigma_k^4\, \|h_1\|_{L^\infty(P_2)}, \quad j=1,2.
		\]
		
		\item[(ii)] For every $h_2 \in L^\infty(\mathbb{R}^N)$, the linear problem \eqref{eq7.} admits a unique solution 
		$\psi = L_2(h_2) \in E_{\rho,\nu}\cap D^{1,2}(\R^N)$ together with unique constants $d_j = d_j(h_2) \in \mathbb{R}$, $j=1,2$, such that
		\[
		\|L_2(h_2)\|_{L^\infty(\mathbb{R}^N)} \leq C\,\sigma_k^4\, \|h_2\|_{L^\infty(Q_2)},
		\qquad
		|d_j| \leq C\,\sigma_k^4\, \|h_2\|_{L^\infty(Q_2)}, \quad j=1,2.
		\]
	\end{enumerate}

\end{Prop}

\medskip
To express the continuity of the operators with respect to the parameters,  
we denote by $\widehat{P}_i$, $\widehat{Q}_i$, $\widehat{Y}_i$, and $\widehat{Z}_i$ ($i=1,2$) the objects obtained from $P_i$, $Q_i$, $Y_i$, and $Z_i$ by replacing the parameters $(r,\rho,\lambda,\nu)$ with $(\widehat{r},\widehat{\rho},\widehat{\lambda},\widehat{\nu})$.
For $h_1 \in C_{0}(P_{2}\cap\widehat{P}_{2})$,  
let $\varphi = L_1(h_1)$ and $\widehat{\varphi} = \widehat{L}_1(h_1)$ be the unique solutions to  
the linear problem \eqref{eq7} corresponding to the parameter pairs  
$(r, \lambda)$ and $(\widehat{r}, \widehat{\lambda})$, respectively.  
Similarly, for $h_2 \in C_{0}(Q_{2}\cap\widehat{Q}_{2})$,  
we denote by $\psi = L_2(h_2)$ and $\widehat{\psi} = \widehat{L}_2(h_2)$ the unique solutions to \eqref{eq7.} corresponding to the parameter pairs $(\rho, \nu)$ and $(\widehat{\rho}, \widehat{\nu})$, respectively.

Then we have the following result.
\begin{Lem}\label{L2.2}
	The operators  $L_{1}$ and $L_{2}$  defined in Proposition~\ref{Prop2.2} are continuous with respect to the parameters $(r,\lambda)$ and $(\rho, \nu)$ in the following sense:
	\begin{align}\label{4.26..}
		&\sup\set{\| (L_{1}-\widehat L_{1})(h_{1}) \|_{L^\infty(\mathbb{R}^N)} 
		 | h_1\in M_1}
		\to 0, \quad \text{as } (r,\lambda)\to(\widehat r, \widehat\lambda),\\
		&\sup\set{\| (L_{2}-\widehat L_{2})(h_{2}) \|_{L^\infty(\mathbb{R}^N)} 
	 | h_2\in M_2}
	\to 0, \quad \text{as }(\rho,\nu)\to(\widehat\rho, \widehat\nu),\notag
	\end{align}
	where \begin{align*}
	&M_1=\set{h_1\in L^\infty(\R^N)| h_{1}=0 \text{ a.e. in } \R^N\setminus(P_2\cap \widehat P_2),\ \widehat L_{1}(h_{1})\in C_0(P_2),\ 
		\| h_{1}\|_{L^{\infty}(\R^N)}=1},\\
	&M_2=\set{h_2\in L^\infty(\R^N)| h_{2}=0 \text{ a.e. in } \R^N\setminus(Q_2\cap \widehat Q_2),\ \widehat L_{2}(h_{2})\in C_0(Q_2),\ 
		\| h_{2}\|_{L^{\infty}(\R^N)}=1}.
	\end{align*}
\end{Lem}
\begin{Rem} The sets $M_i$, $i=1,2$ are   nonempty, ensuring that the suprema under consideration are well-defined.
	 To verify this   for $M_1$, let  $\varphi_0\in C_0^\infty(B) \setminus\{0\}$ be arbitrary, where $B$ is a closed domain contained
	in $\widehat P_1\setminus \cup_{i=1}^k B_{3}(\widehat x^j)$. Define 
	\[h=-\Delta \varphi_0-(2^{*}-1) K_1\big(\frac{\lvert y \rvert}{\mu}\big)U^{2^{*}-2}_{\widehat r,\widehat\lambda}\varphi_0.\]
	Then $h\in   L^\infty(\R^N) $
	satisfies $h=0$ a.e. in $\R^N\setminus (P_2\cap \widehat P_2)$ and $\widehat L_1(h)=\varphi_0\in C_0(P_2)$. By Proposition \ref{Prop2.2}, we have $\|h\|_{L^\infty(\R^N)}\neq 0$. Then $h/\|h\|_{L^\infty(\R^N)}\in M_1$.
	An analogous construction gives the nonemptiness of $M_2$.
\end{Rem}
\begin{proof}[Proof of Lemma \ref{L2.2}]
	We only prove \eqref{4.26..}.
	For $ h_{1}\in C_0(P_2\cap \widehat{P}_2)$  satisfying  $\widehat{\varphi}=\widehat L_{1}(h_{1})\in C_0(P_2)$,  
	$\| h_{1}\|_{L^{\infty}(\R^N)}=1$
	from system \eqref{eq7}, we get for some $c_j$ and $\hat c_j$, $j=1,2$,
	\[  
	-\Delta \varphi-(2^{*}-1)K_1\big(\frac{\lvert y \rvert}{\mu}\big)U^{2^{*}-2}_{r,\lambda}\varphi= h_{1}+\sum_{j=1}^{2}c_{j} Y_{j}
	\]
	and
	\[  
	-\Delta \widehat\varphi-(2^{*}-1) K_1\big(\frac{\lvert y \rvert}{\mu}\big)U^{2^{*}-2}_{\widehat r,\widehat\lambda}\widehat\varphi= h_{1}+\sum_{j=1}^{2}\widehat c_{j} \widehat Y_{j}.
	\] 
	
	Subtracting the two equations, we obtain
	\begin{equation*}
		\varphi-\widehat\varphi
		=  L_1[(2^*-1)K_1\big(\frac{\lvert y \rvert}{\mu}\big)(U^{2^{*}-2}_{r,\lambda}-U^{2^{*}-2}_{\widehat r,\widehat\lambda})\widehat\varphi
		+\sum_{j=1}^{2}\widehat c_{j}( Y_{j}-\widehat Y_{ j})].
	\end{equation*}	
	Then the desired continuity follows from the boundedness of $\widehat\varphi$, $\widehat c_j$
	and the   continuity of $U^{2^{*}-2}_{r,\lambda}$ and $Y_{j}$ with respect 
	to $(r,\rho)$
	in  the $ L^\infty(P)$ norm on any compact set $P$.
\end{proof}

To obtain   decay estimates, we need the following  comparison principle in $\Omega\subset \mathbb R^{N}$. 
For this purpose, we first recall the classical Sobolev embedding theorem.
For every $N\geq3 $, there exists an optimal constant $\mathcal S>0$ depending only on $N$ such that 
\[ 	
	\mathcal S\lVert u\rVert^{2}_{L^{2^{*}}(\mathbb R^{N})}\leq \lVert \nabla u\rVert^{2}_{L^{2}(\mathbb R^{N})}
	,~\text{for all}~u\in D^{1,2}(\mathbb R^{N}),
\]
where $D^{1,2}(\mathbb R^{N})$ is the completion of $C_{0}^{\infty}(\mathbb R^{N})$ with respect to the norm $\lVert u \rVert=\Big(\int_{\mathbb R^{N}} \lvert\nabla u\rvert^{2}\Big)^{\frac{1}{2}}$.
\begin{Lem}\label{lem0} 
	Let $\Omega \subset \mathbb{R}^N$ ($N \geq 3$) be an open set with $C^1$ boundary. Suppose $\phi \in L^{N/2}(\Omega)$ satisfies $\|\phi\|_{L^{N/2}(\Omega)} < \mathcal{S}$.
	Let $v_1 \in D_{\text{loc}}^{1,2}(\Omega) \cap C(\overline{\Omega})$ with $v_1 \geq 0$, and $v_2 \in D^{1,2}(\Omega) \cap C(\overline{\Omega})$. If  
	$$
	-\Delta (v_2 - v_1) - \phi(v_2 - v_1) \leq 0 \quad \text{in } \Omega
	$$  
	in the weak sense, and $v_2 \leq v_1$ on $\partial \Omega$, then $v_2 \leq v_1$ in $\Omega$.  
\end{Lem}
\begin{proof}
	By assumption, $ \omega=(v_{2}-v_{1})^+\in L^{2^*}(\Omega)$.
	For any $R>1$, consider cut off function $\eta_R\in C_0^\infty(\R^N;[0,1])$ with $\eta_R=1$ in $B_R(0)$, $\eta_R=0$ in $\R^N\setminus B_{2R}(0)$, $|\nabla \eta_R|\leq 2/R$.
	Since $v_{2} \leq v_{1}$ on $\partial \Omega$, by \cite[Theorem 9.17]{HB} we have $\eta_R^2\omega \in H^1_0 (\Omega\cap B_{2R}(0))$.  Using $\eta_R^2 \omega$ as a test function in the weak formulation of the inequality $-\Delta (v_2 - v_1) - \phi(v_2 - v_1) \leq 0$, we obtain 
	\begin{align*}
		\mathcal S\|\eta_R \omega\|_{L^{2^*}(\Omega)}^2\leq& \int_{\Omega} |\nabla (\eta_R \omega) |^2
		=\int_{\Omega}\nabla (v_2-v_1) \nabla(\eta_R^2 \omega) +\int_{\Omega} |\nabla\eta_R|^2 \omega^2
		\\
		\leq & \int_{\Omega} \phi  \eta_R^2 \omega^2+ \frac{4}{R^2}\int_{\Omega \cap B_{2R}(0)\setminus B_R(0)}   \omega^2\\
		\leq & \| \phi\|_{L^{\frac{N}{2}}(\Omega)}\| \eta_R \omega\|^{2}_{L^{2^{*}}(\Omega)}
		+C \|\omega\|_{L^{2^*}(\Omega  \setminus B_R(0))}^2,
	\end{align*}
	where $C>0$ depends only on $N$.
	Taking limits as $R\to\infty$, we obtain $\mathcal S\|  \omega\|_{L^{2^*}(\Omega)}^2\leq \| \phi\rVert_{L^{\frac{N}{2}}(\Omega)}\|  \omega\|^{2}_{L^{2^{*}}(\Omega)}$.
	Since $\|\phi\|_{L^{N/2}(\Omega)} < \mathcal{S}$, this implies $\|\omega\|_{L^{2^*}(\Omega)} = 0$, so $\omega = 0$ in $\Omega$. Thus, $v_2 \leq v_1$ in $\Omega$. 
\end{proof}

\section{Minimization Problem}\label{sec:3}

Recall
\begin{equation*}
	\sigma_{k}:= k^{\frac{1}{N-2}}
\end{equation*}
and
\begin{equation*}
	P_{l}= \bigcup_{i=1}^{k} B_{\sigma^{l}_{k}}(x^{i}),~~Q_{l}=\bigcup_{i=1}^{k} B_{\sigma^{l}_{k}}(y^{i}),~~l=1,2.
\end{equation*}

Define 
\begin{equation*}
	\Lambda_{k}= \Big\{(\varphi,\psi)\in \mathbb E \mid \lVert (\varphi,\psi)\rVert_\infty\leq\frac{1}{\mu^{\frac{m}{2}}}\frac{1}{\sigma_k^{\frac{N+1}{2}}}
	\Big\},
\end{equation*}
where $\lVert (\varphi,\psi)\rVert_{\infty}=\lVert\varphi\rVert_{L^{\infty}(\mathbb R^{N})}+\lVert \psi\rVert_{L^{\infty}(\mathbb R^{N})}$.

For any $(\varphi_{0},\psi_{0})\in\Lambda_{k}$, we will solve the following problem
\begin{equation}\label{eq15}
	\begin{cases}
		-\Delta u_{1} = K_1\big(\frac{\lvert y \rvert}{\mu}\big) \lvert u_{1}\rvert^{2^{*}-2}u_{1}+\beta \lvert u_{2}\rvert^{\frac{2^{*}}{2}}\lvert u_{1}\rvert^{\frac{2^{*}}{2}-2}u_{1},  &~\text{in}~\mathbb R^N \setminus P_{1},\\[3mm]
		-\Delta u_{2} = K_2\big(\frac{\lvert y \rvert}{\mu}\big) \lvert u_{2}\rvert^{2^{*}-2}u_{2}+\beta \lvert u_{1}\rvert^{\frac{2^{*}}{2}}\lvert u_{2}\rvert^{\frac{2^{*}}{2}-2}u_{2},  &~\text{in}~\mathbb R^N \setminus Q_{1}
	\end{cases}
\end{equation}
with  \begin{equation}\label{boundary}
	u_{1}=\varphi_{0}+U_{r,\lambda} ~\text{in}~ P_{1},~~u_{2}=\psi_{0}+V_{\rho,\nu}~\text{in}~ Q_{1},
\end{equation}
satisfying
\begin{equation}\label{u}
	\begin{aligned}
		(2^{*}-1) K_\infty \lVert u_{1} \rVert^{2^*-2}_{L^{2^*}(\mathbb R^{N}\setminus P_{1})}<\frac{\mathcal S}{2}~~\text{and}~~(2^{*}-1) K_\infty \lVert u_{2} \rVert^{2^*-2}_{L^{2^*}(\mathbb R^{N}\setminus Q_{1})}<\frac{\mathcal S}{2},
	\end{aligned}
\end{equation}
where $K_\infty=\|(K_1,K_2)\|_\infty$.

\subsection{Decay Estimates for Solutions and Deviations from the Limit Profile}
Since $\beta < 0$, by Kato's inequality, any solution $(u_1, u_2)$ satisfying \eqref{eq15}  also satisfies
\begin{equation}\label{subsolution}
	\begin{cases}
		-\Delta \lvert u_{1}\rvert \leq  K_\infty \lvert u_{1}\rvert^{2^{*}-1},  &~\text{in}~\mathbb R^N \setminus P_{1},\\[3mm]
		-\Delta \lvert u_{2} \rvert\leq  K_\infty \lvert u_{2}\rvert^{2^{*}-1},  &~\text{in}~\mathbb R^N \setminus Q_{1}.
	\end{cases}
\end{equation}
In the remainder of this section, let $(\varphi_{0}, \psi_{0}) \in \Lambda_{k}$. We will use the notation $\varpi_{r,\alpha}$, $\varpi_{\rho,\alpha}$, and $\varpi_\alpha$ as defined in \eqref{varpir}, \eqref{varpirho}, and \eqref{varpi1.}, respectively.
\begin{Lem}\label{lem7}  
	Suppose $(u_1, u_2)$ is a solution to the problem \eqref{subsolution} satisfying \eqref{boundary} and \eqref{u}. Then there exists a constant $A > 0$, independent of $k$ and the choice of $(\varphi_0, \psi_0)$, such that  
	\[
	|u_1(y)| \leq A \sum_{i=1}^k \frac{1}{(1 + |y - x^i|)^{N-2}}, \quad
	|u_2(y)| \leq A \sum_{i=1}^k \frac{1}{(1 + |y - y^i|)^{N-2}}, \quad \forall y \in \mathbb{R}^N.
	\]
\end{Lem}
\begin{proof}
	From \eqref{10'}, and $1-\frac{4\tau_1}{N-2}>0$, there holds
	\begin{equation*}	
		-\Delta \varpi_{r,N-2} \geq  K_\infty\varpi_{r,N-2}^{2^{*}-1},~\text{in}~\mathbb R^N \setminus P_{1}.	
	\end{equation*}
	For  $(\varphi_{0},\psi_{0})\in\Lambda_{k}$, there exists $C>0$ independent of $k$ such that	
	\begin{equation}	\label{11}
		\lvert u_{1}(y)\rvert=\lvert \varphi_{0}+U_{r,\lambda} \rvert\leq U_{r,\lambda}+\frac{1}{\mu^{\frac{m}{2}}}\frac{1}{\sigma_k^{\frac{N+1}{2}}}\leq C\varpi_{r,N-2}, \quad y\in  \partial P_{1}.
	\end{equation}
	
	By the mean value theorem, there holds
	\begin{equation*}
		\begin{aligned}	
			-\Delta (\lvert u_{1}\rvert-C\varpi_{r,N-2}) \leq &  K_\infty(\lvert u_{1}\rvert^{2^{*}-1}-(C\varpi_{r,N-2})^{2^{*}-1})
			=\phi(\lvert u_{1}\rvert-C\varpi_{r,N-2}),~\text{in}~\mathbb R^N \setminus P_{1},
		\end{aligned}
	\end{equation*}
	where $\phi=K_\infty(\lvert u_{1}\rvert^{2^{*}-1}-(C\varpi_{r,N-2})^{2^{*}-1})/(\lvert u_{1}\rvert-C\varpi_{r,N-2})$.
	By direct calculation, it follows that \[\lVert C\varpi_{r,N-2}\rVert^{2^*-2}_{L^{2^*}(\mathbb R^{N}\setminus P_{1})}\rightarrow 0 \quad \text{as }k\rightarrow\infty. \]
Therefore,	from \eqref{u}, there holds
	\begin{equation*}
		\lVert  \phi\rVert_{L^{\frac{N}{2}}(\mathbb R^{N}\setminus P_{1})}
		\leq(2^{*}-1)K_\infty(\|  u_{1} \|_{L^{2^{*}}(\mathbb R^{N}\setminus P_{1})}+  \|C\varpi_{r,N-2}\|_{L^{2^{*}}(\mathbb R^{N}\setminus P_{1})})^{2^{*}-2}
		<\mathcal S.
	\end{equation*}
		By Lemma \ref{lem0}, we get
	\begin{equation*}	
		\lvert u_{1}(y)	\rvert\leq C\varpi_{r,N-2}(y),~\text{in}~\mathbb R^{N}\setminus P_{1}.
	\end{equation*}
	The inequality on $P_1$ is already established by \eqref{11}.  
	Similarly, we can obtain the corresponding estimate for $u_2$.
\end{proof}

\medskip

We also get the following estimates.
\begin{Lem}\label{lem8}
	For any $(\varphi_{0},\psi_{0})\in \Lambda_{k}$, let $(u_{1},u_{2})$ be a solution to the problem $\eqref{eq15}$ satisfying \eqref{boundary} and \eqref{u}.
	There exist some large $C > 0$ and $\theta>0$ small enough independent of $k$  such that
	\begin{equation*}
		\begin{aligned}
			\lvert u_{1}(y)-U_{r,\lambda}(y)\rvert
			\leq C\frac{1}{\mu^{\frac{m}{2}+\theta}}
			\sum_{i=1}^{k}\frac{1}{(1+\lvert y-x^{i}\rvert)^{\frac{N}{2}}}
			+C\frac{1}{\mu^{2-\frac{2\tau_1}{N-2}}}\sum_{j=1}^{k}\frac{1}{(1+\lvert y-y^{j}\rvert)^{N-\frac{2\tau_1}{N-2}-2}}
		\end{aligned}
	\end{equation*}
	and
	\begin{equation*}
		\begin{aligned}
			\lvert u_{2}(y)-V_{\rho,\nu}(y)\rvert
			\leq C\frac{1}{\mu^{\frac{m}{2}+\theta}}
			\sum_{i=1}^{k}\frac{1}{(1+\lvert y-y^{i}\rvert)^{\frac{N}{2}}}
			+C\frac{1}{\mu^{2-\frac{2\tau_1}{N-2}}}\sum_{j=1}^{k}\frac{1}{(1+\lvert y-x^{j}\rvert)^{N-\frac{2\tau_1}{N-2}-2}}, 
		\end{aligned}
	\end{equation*}
	for $y\in\R^N$.
	Moreover,
	\begin{equation*}
		\begin{aligned}
			&\lvert \nabla(u_{1}(y)-U_{r,\lambda}(y))\rvert\\
			\leq& C\frac{1}{\mu^{\frac{m}{2}+\theta}}
			\sum_{i=1}^{k}\frac{1}{(1+\lvert y-x^{i}\rvert)^{\frac{N}{2}}}
			+C\frac{1}{\mu^{2-\frac{2\tau_1}{N-2}}}\sum_{j=1}^{k}\frac{1}{(1+\lvert y-y^{j}\rvert)^{N-\frac{2\tau_1}{N-2}-2}},~y\in\mathbb R^{N}\setminus \cup^{k}_{i=1}B_{\sigma_{k}+1}(x^{i})  
		\end{aligned}
	\end{equation*}
	and
	\begin{equation*}
		\begin{aligned}
			&\lvert\nabla( u_{2}(y)-V_{\rho,\nu}(y))\rvert\\
			\leq& C\frac{1}{\mu^{\frac{m}{2}+\theta}}
			\sum_{i=1}^{k}\frac{1}{(1+\lvert y-y^{i}\rvert)^{\frac{N}{2}}}
			+C\frac{1}{\mu^{2-\frac{2\tau_1}{N-2}}}\sum_{j=1}^{k}\frac{1}{(1+\lvert y-x^{j}\rvert)^{N-\frac{2\tau_1}{N-2}-2}},~y\in\mathbb R^{N}\setminus \cup^{k}_{i=1}B_{\sigma_{k}+1}(y^{i}). 
		\end{aligned}
	\end{equation*}
\end{Lem}

\begin{proof}
	Let $(\varphi,\psi)=(u_{1}, u_{2})-(U_{r,\lambda},V_{\rho,\nu})$. By \eqref{eq15} and \eqref{boundary}, we have
	\begin{equation}\label{eq17}
		\begin{cases}
			-\Delta \varphi = K_1\big(\frac{\lvert y \rvert}{\mu}\big) \lvert u_{1}\rvert^{2^{*}-2}u_{1}-\sum_{i=1}^{k}U^{2^*-1}_{x^{i},\lambda}+\beta \lvert u_{2}\rvert^{\frac{2^{*}}{2}}\lvert u_{1}\rvert^{\frac{2^{*}}{2}-2}u_{1},  &~\text{in}~\mathbb R^N \setminus P_{1},\\[3mm]
			-\Delta \psi = K_2\big(\frac{\lvert y \rvert}{\mu}\big) \lvert u_{2}\rvert^{2^{*}-2}u_{2}-\sum_{i=1}^{k}V^{2^*-1}_{y^{i},\nu}+\beta \lvert u_{1}\rvert^{\frac{2^{*}}{2}}\lvert u_{2}\rvert^{\frac{2^{*}}{2}-2}u_{2},  &~\text{in}~\mathbb R^N \setminus Q_{1},\\[3mm]
			\varphi=\varphi_{0} ~\text{in}~ P_{1},~\psi=\psi_{0} ~\text{in}~ Q_{1}.
		\end{cases}
	\end{equation}
	
	(\romannumeral1) From Lemmas \ref{lem2} and \ref{lem7}, there holds
	\begin{equation}\label{3.11}
		\begin{aligned}
			K_1(\frac{y}{\mu})\lvert u_{1}\rvert^{2^{*}-2}u_{1} -\sum_{i=1}^{k}U^{2^*-1}_{x^{i},\lambda}	
			&=O\Big(\varpi_{r,N-2}^{2^*-2}\Big)\varphi+K_1(\frac{y}{\mu})\lvert U_{r,\lambda} \rvert^{2^{*}-2}U_{r,\lambda}-\sum_{i=1}^{k}U^{2^*-1}_{x^{i},\lambda} \\
			&=O\Big(\varpi_{r,N-2}^{2^*-2}\Big)\varphi+O\Big(\frac{1}{\mu^{\frac{m}{2}+\theta}}\Big)\varpi_{r, \frac{N}{2}+2}.
		\end{aligned}
	\end{equation}
	
	For 
	$y\in \R^N\setminus B_{\frac{r+\rho}{2}}(0)$,
	we have 
	$|y-x^j|\geq \frac{\rho-r}{2}=\frac{|x^j-y^j|}{2}$.
	Hence, 
	\begin{equation}\label{out}
		|y-y^j|\leq |y-x^j|+|x^j-y^j|\leq 3|y-x^j|,\quad j=1,\dots,k,\quad  y\in \R^N\setminus B_{\frac{r+\rho}{2}}(0).
	\end{equation}
	Similarly, we have 
	
	\begin{equation}\label{in}
		|y-y^j|\geq \frac{\rho-r}{2},\quad |y-y^j|\geq  \frac{1}{3}|y-x^j|,\quad  j=1,\dots,k,
		\quad \text{for } y\in  B_{\frac{r+\rho}{2}}(0).
	\end{equation}
	Note that 
	\begin{equation*}
		\begin{aligned}
			\frac{N}{2}-\frac{m}{2}-\frac{2N-2}{N-2}\tau_{1}
			>\frac{N-4-m}{2}+\frac{2(m-1)}{N-2}=\begin{cases} \frac{m-1}{6}>0,&~N=5,\\[3mm]
				\frac{(N-4)(N-2)-m(N-6)-4}{2(N-2)}>\frac{2N-8}{2(N-2)}>0,&~N \geq 6.
			\end{cases}
		\end{aligned}
	\end{equation*}
	In $\Omega_{1}\cap B_{\frac{r+\rho}{2}}(0)$, 
	by Lemmas \ref{lem7}, \ref{lemB1} and \eqref{key3.}, \eqref{tau1a}, \eqref{in}, we obtain
	\begin{equation}\label{3.9}
		\begin{aligned}
			&\lvert u_{2}\rvert^{\frac{2^{*}}{2}}\lvert u_{1}\rvert^{\frac{2^{*}}{2}-2}u_{1} \\
			\leq& C\Big(\sum_{i=1}^{k}\frac{1}{(1+\lvert y-x^{i}\rvert)^{N-2}}\Big)^{\frac{2^*}{2}-1}\Big(\sum_{j=1}^{k}\frac{1}{(1+\lvert y-y^{j}\rvert)^{N-2}}\Big)^{\frac{2^*}{2}}\\
			\leq& C\sum_{i=1}^{k}\frac{1}{(1+\lvert y-x^{i}\rvert)^{2}}\sum_{j=1}^{k}\frac{1}{(1+\lvert y-y^{j}\rvert)^{N-\frac{2}{N-2}\tau_{1}}}\\
			\leq& C\frac{1}{(1+\lvert y-x^{1}\rvert)^{2-\tau_1}}\frac{1}{(1+\lvert y-y^{1}\rvert)^{N-\frac{N}{N-2}\tau_{1}}}\\
			\leq& C\frac{1}{\mu^{\frac{N}{2}-\frac{2N-2}{N-2}\tau_{1}}} \frac{1}{(1+\lvert y-x^{1}\rvert)^{\frac{N}{2}+2}}\leq C\frac{1}{\mu^{\frac{m}{2}+\theta}}\frac{1}{(1+\lvert y-x^{1}\rvert)^{\frac{N}{2}+2}},
		\end{aligned}	
	\end{equation}
	where $\theta \in \Big(0, \frac{N}{2}-\frac{m}{2}-\frac{2N-2}{N-2}\tau_{1}\Big)$ is chosen sufficiently small.
	
	For $y\in \R^N\setminus B_{\frac{r+\rho}{2}}(0)$,
	we know that $\lvert y-x^{j}\rvert\geq\frac{\rho-r}{2}\geq C\mu$, $j=1,\dots,k$.
	By Lemma \ref{lem7} and \eqref{tau1a}, it holds
	\begin{equation}\label{3.9.}
		\begin{aligned}
			\lvert u_{2}\rvert^{\frac{2^{*}}{2}}\lvert u_{1}\rvert^{\frac{2^{*}}{2}-2}u_{1}
			\leq& C\Big(\sum_{i=1}^{k}\frac{1}{(1+\lvert y-x^{i}\rvert)^{N-2}}\Big)^{\frac{2^*}{2}-1}\Big(\sum_{j=1}^{k}\frac{1}{(1+\lvert y-y^{j}\rvert)^{N-2}}\Big)^{\frac{2^*}{2}}\\
			\leq& C\Big(\frac{k}{\mu^{N-2}}\Big)^{\frac{2^*}{2}-1}\sum_{j=1}^{k} \frac{1}{(1+\lvert y-y^{j}\rvert)^{N-\frac{2}{N-2}\tau_{1}}}\\
			\leq& C\frac{1}{\mu^{2-\frac{2}{N-2}\tau_{1}}}\sum_{j=1}^{k}\frac{1}{(1+\lvert y-y^{j}\rvert)^{N-\frac{2}{N-2}\tau_{1}}}.
		\end{aligned}	
	\end{equation}

	Combining \eqref{3.11}-\eqref{3.9.}, for the first equation in \eqref{eq17}, we find that $\lvert\varphi\rvert$ satisfies
	\begin{equation}\label{eqss}
		\begin{aligned}
			-\Delta \lvert\varphi\rvert
			\leq &C \varpi_{r,N-2}^{2^*-2}
			\lvert\varphi\rvert +C\frac{1}{\mu^{\frac{m}{2}+\theta}}\varpi_{r,\frac{N}2+2}+C\frac{1}{\mu^{2-\frac{2\tau_1}{N-2}}}\varpi_{\rho,N-\frac{2\tau_1}{N-2}} 
			,~y\in\mathbb R^{N} \setminus (P_{1}\cup Q_1).
		\end{aligned}
	\end{equation} 
	
	Let
	\begin{equation*}
		\begin{aligned}
			\omega := \frac{1}{\mu^{\frac{m}{2}+\theta}}\varpi_{r,\frac{N}2}+\frac{1}{\mu^{2-\frac{2\tau_1}{N-2}}}\varpi_{\rho,N-\frac{2\tau_1}{N-2}-2}\quad  \text{in } \R^N.
		\end{aligned}
	\end{equation*}	
	From Lemma \ref{lemB1.}, in $\R^N\setminus(P_1\cup Q_1)$, there holds
	\begin{equation}\label{e2}
		\begin{aligned}
			- \Delta \omega  \geq& C \sigma_k^{4-\frac{4\alpha}{N-2}} \left(\frac{\varpi_{r,N-2}^{2^*-2}}{\mu^{\frac{m}{2}+\theta}}\varpi_{r,\frac{N}2}+\frac{\varpi_{\rho,N-2}^{2^*-2}}{\mu^{2-\frac{2\tau_1}{N-2}}}\varpi_{\rho,N-\frac{2\tau_1}{N-2}-2}\right)\\
			&+\frac{N(N-4)}8  \frac{1}{\mu^{\frac{m}{2}+\theta}}\varpi_{r,\frac{N}2+2}+ \frac{((N-2)^2-2\tau_1)\tau_1}{(N-2)^2}  \frac{1}{\mu^{2-\frac{2\tau_1}{N-2}}}\varpi_{\rho,N-\frac{2\tau_{1}}{N-2}}.
		\end{aligned}
	\end{equation} 
	In 
	$ \R^N\setminus B_{\frac{r+\rho}{2}}(0)$, by \eqref{out}, we have
	\[
	\varpi_{\rho,N-2}^{2^*-2}\geq \frac{1}{3^4} \varpi_{r,N-2}^{2^*-2}.
	\]
	In $B_{\frac{r+\rho}{2}}(0)$, by \eqref{in}, we have
	\[
	\frac{1}{\mu^{2-\frac{2\tau_{1}}{N-2} }}\varpi_{\rho,N-\frac{2\tau_{1}}{N-2}-2}\leq 
	\frac{C}{\mu^{\frac{N}{2}- \frac{4\tau_{1}}{N-2}}}\varpi_{\rho,\frac{N}{2}}\leq 
	\frac{1}{\mu^{\frac{m}{2}+\theta}}\varpi_{r,\frac{N}2} .
	\]
	In either case, we have 
	\[
	\frac{\varpi_{r,N-2}^{2^*-2}}{\mu^{\frac{m}{2}+\theta}}\varpi_{r,\frac{N}2}+\frac{\varpi_{\rho,N-2}^{2^*-2}}{\mu^{2-\frac{2}{N-2}\tau_{1}}}\varpi_{\rho,N-\frac{2}{N-2}\tau_{1}-2} \geq \frac{1}{3^4}\varpi_{r,N-2}^{2^*-2}\omega.
	\]
	Therefore, together with  \eqref{eqss} and \eqref{e2},
	there is $M_0>0$ large such that 
	\[
	-\Delta (|\varphi|-M_0\omega) \leq C  \varpi_{r,N-2}^{2^*-2} (|\varphi|-M_0\omega), \quad \text{in }\R^N\setminus (P_1\cup Q_1).
	\]

	In $  P_{1}$, if we assume further that $\theta<\frac{N-2-m}{2(N-2)^2}$, then there holds  
	\begin{equation*}
		\lvert\varphi\rvert\leq\frac{1}{\mu^{\frac{m}{2}}}\frac{1}{\sigma_k^{\frac{N+1}{2}}}
		=\frac{1}{\mu^{\frac{m}{2}+\frac{N-2-m}{2(N-2)^2}}}\frac{1}{\sigma_{k}^{\frac{N}{2}}}
		=o\left(
		\frac{1}{\mu^{\frac{m}{2}+\theta}\sigma_{k}^{\frac{N}{2}}} \right)
		<M_0\omega.
	\end{equation*}
	In $Q_1$, 
	\[
	|\varphi|\leq |u_1|+|U_{r,\lambda}|\leq C\frac{k}{\mu^{N-2}}
	=o\left(\frac{1}{\mu^{2-\frac{2\tau_1}{N-2}}\sigma_k^{N-2-\frac{2\tau_1}{N-2}}}\right)\leq M_0\omega.
	\]
	
	By Lemma \ref{lem0}, we conclude that
	$$
	|\varphi| \leq M_0\omega \quad \text{in } \mathbb{R}^N \setminus (P_1 \cup Q_1).
	$$
	Since the inequality already holds in $P_1 \cup Q_1$, it follows that
	$$
	|\varphi| \leq M_0\omega \quad \text{in } \mathbb{R}^N.
	$$
	A similar argument applies to obtain the corresponding estimate for $|\psi|$.
	
	(\romannumeral2) By  the gradient estimates in \cite[Theorem 3.9]{GT01},  it follows that for  $  y \in \mathbb R^{N}\setminus \cup^{k}_{i=1}B_{\sigma_{k}+1}(x^{i})$,
	\begin{equation*}
		\lvert \nabla\varphi(y)\rvert\leq  C\Big(\sup_{B_{1}(y)}\lvert\varphi\rvert+\sup_{B_{1}(y)} \Big\lvert K_1\big(\frac{\lvert \cdot \rvert}{\mu}\big) \lvert u_{1}\rvert^{2^{*}-2}u_{1}-\sum_{i=1}^{k}U^{2^*-1}_{x^{i},\lambda} 
		+\beta \lvert u_{2}\rvert^{\frac{2^{*}}{2}}\lvert u_{1}\rvert^{\frac{2^{*}}{2}-2}u_{1} \Big\rvert\Big).
	\end{equation*}
	
	By \eqref{3.11}, \eqref{3.9} and \eqref{3.9.}, we obtain
	\begin{equation*}
		\begin{aligned}
			&\sup_{B_{1}(y)} \Big\lvert K_1\big(\frac{\lvert \cdot \rvert}{\mu}\big) \lvert u_{1}\rvert^{2^{*}-2}u_{1}-\sum_{i=1}^{k}U^{2^*-1}_{x^{i},\lambda}+\beta \lvert u_{2}\rvert^{\frac{2^{*}}{2}}\lvert u_{1}\rvert^{\frac{2^{*}}{2}-2}u_{1} \Big\rvert\\
			\leq&\sup_{B_{1}(y)} \Big\lvert K_1\big(\frac{\lvert \cdot \rvert}{\mu}\big) \lvert u_{1}\rvert^{2^{*}-2}u_{1}-\sum_{i=1}^{k}U^{2^*-1}_{x^{i},\lambda}\Big\rvert+\sup_{B_{1}(y)} \Big\lvert\beta \lvert u_{2}\rvert^{\frac{2^{*}}{2}}\lvert u_{1}\rvert^{\frac{2^{*}}{2}-2}u_{1} \Big\rvert\\
			\leq &C \sup_{B_{1}(y)}\varpi_{r,N-2}^{2^*-2}\lvert \varphi\rvert+C\sup_{z\in B_{1}(y)}\frac{1}{\mu^{\frac{m}{2}+\theta}}\sum_{i=1}^{k}\frac{1}{(1+\lvert z-x^{i}\rvert)^{\frac{N}{2}+2}}\\
			&+C\sup_{z\in B_{1}(y)}\frac{1}{\mu^{2-\frac{2\tau_1}{N-2}}}\sum_{j=1}^{k}\frac{1}{(1+\lvert z -y^{j}\rvert)^{N-\frac{2\tau_1}{N-2}}}.
		\end{aligned}
	\end{equation*}
	Therefore, we get the estimate of $\lvert \nabla(u_{1}(y)-U_{r,\lambda}(y))\rvert$.
	A similar argument can be applied to obtain the estimate of $\lvert\nabla( u_{2}(y)-V_{\rho,\nu}(y))\rvert$.
\end{proof}
\begin{Cor}\label{cor3.3}
	Under the  
	assumption of Lemma \ref{lem8}, there hold
	
	\[
	\lvert u_{1}(y)-U_{r,\lambda}(y)\rvert
	\leq C\frac{1}{\mu^{\frac{m}{2}+\theta}}
	\sum_{i=1}^{k}\frac{1}{(1+\lvert y-x^{i}\rvert)^{\frac{N}{2}}}
	\quad \text{for } y\in \cup_{i=1}^k B_{\frac{\rho-r}{2}}(x^i)
	\]
	
	and
	
	\[
	\lvert u_{2}(y)-V_{\rho,\nu}(y)\rvert
	\leq C\frac{1}{\mu^{\frac{m}{2}+\theta}}
	\sum_{i=1}^{k}\frac{1}{(1+\lvert y-y^{i}\rvert)^{\frac{N}{2}}}
	\quad \text{for } y\in \cup_{i=1}^k B_{\frac{\rho-r}{2}}(y^i).
	\]
	
\end{Cor}
\begin{proof}
	For $y\in \cup_{i=1}^k B_{\frac{\rho-r}{2}}(x^i)$,
	we have 
	$|y-y^j|\geq |y-x^j|$ and $|y-y^j|\geq C\mu$, $j=1,\dots,k$.
	Therefore,
	\[
	\frac{1}{\mu^{2-\frac{2\tau_1}{N-2}}} \frac{1}{(1+\lvert y-y^{j}\rvert)^{N-\frac{2\tau_1}{N-2}-2}}\leq \frac{1}{\mu^{\frac{N}2-\frac{4\tau_1}{N-2}}} \frac{1}{(1+\lvert y-y^{j}\rvert)^{\frac{N}2}}=o(\frac1{\mu^{\frac{m}2+\theta}}) \frac{1}{(1+\lvert y-x^{j}\rvert)^{\frac{N}2}}.
	\]
	Then we conclude the proof.
\end{proof}

\subsection{Estimates for Dead Core Formation}
We estimate the size of dead cores for  a solution $(u_1, u_2)$   to the problem \eqref{subsolution} with \eqref{boundary} and \eqref{u}. To achieve this result, we rely on the following lemma, which can be derived from \cite[Theorems 8.4.2–8.4.7]{PP07}. A direct proof can also be found in our previous work \cite{G-Z}.

\begin{Lem}\label{lem5.11}
	Suppose $N\geq5$. For each $M>0$, there exists $\delta>0$   such that
	\begin{equation*}
		\begin{cases}
			-\Delta \omega+M\omega^{\frac{2}{N-2}}=0,&~\text{in}~ B_{3/2}(0),\\[3mm]
			\omega=\delta,&~\text{on}~ \partial B_{3/2}(0)
		\end{cases}
	\end{equation*}
	has a unique distribution non-negative radial solution $\omega(y) = \omega(\lvert y\rvert)$, satisfying $\omega\in C^{1}(B_{1}(0))$, $\omega=0$ in $B_{1}(0)$ and $\omega^{\prime}(|y|)>0$ if $\omega>0$.
\end{Lem}

By Corollary \ref{cor3.3}, we obtain in $ B_{ \mu^{1-\tau_1} }(x^{1})\cap \Omega_1$
\begin{equation*} 
	\begin{aligned}
		\frac{\lvert\varphi\rvert}{U_{r,\lambda}}=O(\frac{1}{\mu^{\frac{m}{2}+\theta}})(1+|y-x^1|)^{\frac{N}{2}-2+\tau_1}
		=O(\frac{1}{\mu^\theta}).
	\end{aligned}	
\end{equation*}
Hence, 
\[
u_1>\frac{1}{2} U_{x^1,\lambda}   \quad {in}\quad  B_{ \mu^{1-\tau_1} }(x^{1})\cap \Omega_1.
\]

\begin{Lem}\label{deadcore}
	For any $\vartheta \in (0, 1-\tau_1]\cap (0,   \frac{(N-4)(N-2-\tau_1)}{(N-2)^2})$, there exists a large $k_{\vartheta}$ such that  
	for any $k\geq k_{\vartheta}$, 
	\[
	u_1=0 \text{ in } \cup_{j=1}^k B_{\mu^\vartheta} (y^j),\quad u_2=0\text{ in } \cup_{j=1}^k B_{\mu^\vartheta} (x^j).
	\]
\end{Lem}
\begin{proof} 
	Fixing any $\vartheta\in(0, 1-\tau_1]$,
	we know  
	that 
	\[
	u_1>\frac{1}{2} U_{x^1,\lambda} > C\mu^{-\vartheta (N-2)} \gg |u_2|  =O(\frac{k}{\mu^{N-2}})=O(\frac{1}{\mu^{N-2-\tau_1}}) \quad \text{in } B_{\frac{3}{2}\mu^\vartheta}(x^1).
	\]

	From system \eqref{eq5}, we have
	\begin{equation*} 
		-\Delta |u_{2}| +C\mu^{-\vartheta N}|u_2|^{\frac{2}{N-2}} \leq
		-\Delta |u_{2}| - \Big(K_2\big(\frac{\lvert y \rvert}{\mu}\big) |u_{2}|^{\frac{2^{*}}{2}} +\beta  |u_{1}|^{\frac{2^{*}}{2}}\Big)|u_{2}|^{\frac{2^{*}}{2}-1} \leq0,~\text{in}~ B_{\frac{3 }{2}\mu^\vartheta }(x^{1}).
	\end{equation*}
	
	Setting $M=C $ in Lemma \ref{lem5.11}, consider 
	\begin{equation*} 
		\bar \omega (y)=  \mu^{-\vartheta\frac{(N-2)^2}{N-4}}\omega\Big(\frac{ y-x^1 }{\mu^\vartheta}\Big).
	\end{equation*}
	Then we can get 
	\begin{equation*}
		\begin{cases}
			-\Delta \bar \omega +C \mu^{-\vartheta N} \bar \omega^{\frac{2}{N-2}}=0,&~\text{in}~ B_{\frac{3 }{2}\mu^\vartheta }(x^{1}),\\
			\bar\omega= \delta \mu^{-\vartheta\frac{(N-2)^2}{N-4}},&~\text{on}~ \partial B_{\frac{3 }{2}\mu^\vartheta }(x^{1}),\\ 
			\bar\omega=0,&~\text{in}~  B_{\frac{\pi r}{k} }(x^1).
		\end{cases}
	\end{equation*}
	If $\vartheta <\frac{(N-4)(N-2-\tau_1)}{(N-2)^2}$,
	then for large $k$ there holds, on $\partial B_{\frac{3 }{2}\mu^\vartheta }(x^{1})$, 
	\[
	\bar\omega \geq C \frac{1}{\mu^{\vartheta\frac{(N-2)^2}{N-4}}}> |u_2| = O(\frac{1}{\mu^{N-2-\tau_1}}).
	\]
	
	Therefore, we have
	\begin{equation*}
		u_{2}=0~\text{in}~\cup_{i=1}^{k} B_{\mu^\vartheta}(x^{i}).
	\end{equation*}
	The proof is concluded by applying similar arguments to $u_{1}$.
\end{proof}
\begin{Rem}\label{dcrem}
	1. If $1-\tau_1< \frac{(N-4)(N-2-\tau_1)}{(N-2)^2}$, we can prove 
	\[
	u_1=0 \text{ in } \cup_{j=1}^k B_{|y^1-y^2| } (y^j),\quad u_2=0\text{ in } \cup_{j=1}^k B_{|x^1-x^2|} (x^j).
	\]
	
	2. We can check that $\frac{2\tau_1}{N-2}<\min\{1-\tau_1, \frac{(N-4)(N-2-\tau_1)}{(N-2)^2}\}$. Therefore,
	we can choose a proper $\vartheta$ such that $\mu^\vartheta>\sigma^{2}_k=\mu^\frac{2\tau_1}{N-2}$. Hence,
	for large $k$, $u_1=0$ in $Q_2$ and $u_2=0$ in $P_2$.
\end{Rem}

\subsection{Decay of Solution Differences with Different Boundary Conditions}
For any $(\varphi_{0},\psi_{0})\in \Lambda_{k}$, we assume that $(u_{1},u_{2})$ is a solution to the problem $\eqref{eq15}$ satisfying \eqref{boundary} and \eqref{u}. Let $(\bar\varphi_{0},\bar\psi_{0})\in \Lambda_{k}$ be another point and $(\bar u_{1},\bar u_{2})$ be a solution to the corresponding problem $\eqref{eq15}$ satisfying the corresponding \eqref{boundary} and \eqref{u}. Denote $$(w,v) := (u_{1},u_{2})-(\bar u_{1},\bar u_{2}),~(w_{0},v_{0}) := (\varphi_{0},\psi_{0})-(\bar\varphi_{0},\bar\psi_{0}).$$ 

We give some estimates on $(w,v)$.
\begin{Lem}\label{lem9}
	There are some $C > 0$ independent of $k$ such that
	\begin{equation}\label{3.5.1}
		\begin{aligned}
			\lvert w(y)\rvert
			\leq&C \sigma^{N-2}_{k}(\| w_{0}\|_{L^{\infty}(P_{1})}+\|v_0\|_{L^{\infty}(Q_{1})})\sum_{i=1}^{k}\left(\frac{1}{(1+\lvert y-x^{i}\rvert)^{N-2}}+\frac{1}{(1+\lvert y-y^{i}\rvert)^{N-2}}\right),
		\end{aligned}
	\end{equation} 	
	and
	\begin{equation}\label{3.5.2}
		\begin{aligned}
			\lvert v(y)\rvert
			\leq&C \sigma^{N-2}_{k}(\| w_{0}\|_{L^{\infty}(P_{1})}+\|v_0\|_{L^{\infty}(Q_{1})})\sum_{i=1}^{k}\left(\frac{1}{(1+\lvert y-x^{i}\rvert)^{N-2}}+\frac{1}{(1+\lvert y-y^{i}\rvert)^{N-2}}\right).
		\end{aligned}
	\end{equation} 
	Moreover,
	\begin{equation*}
		\begin{aligned}
			\lvert \nabla w(y)\rvert 
			\leq & C\sigma^{N-2}_{k}(\| w_{0}\|_{L^{\infty}(P_{1})}+\|v_0\|_{L^{\infty}(Q_{1})})
			\sum_{i=1}^{k} \frac{1}{(1+\lvert y-x^{i}\rvert)^{N-2}} 
			,~y\in P_{2}\setminus \cup^{k}_{i=1}B_{\sigma_{k}+1}(x^{i}) 
		\end{aligned}
	\end{equation*}
	and
	\begin{equation*}
		\begin{aligned}
			\lvert\nabla v(y)\rvert 
			\leq& C\sigma^{N-2}_{k}(\| w_{0}\|_{L^{\infty}(P_{1})}+\|v_0\|_{L^{\infty}(Q_{1})}) \sum_{i=1}^{k} \frac{1}{(1+\lvert y-x^{i}\rvert)^{N-2}} 
			,~y\in Q_{2}\setminus \cup^{k}_{i=1}B_{\sigma_{k}+1}(y^{i}). 
		\end{aligned}
	\end{equation*}
\end{Lem}

\begin{proof}    
	By \eqref{eq15} and \eqref{boundary}, we have
	\begin{equation*}
		\begin{cases}
			-\Delta w = K_1\big(\frac{\lvert y \rvert}{\mu}\big) \big(\lvert u_{1}\rvert^{2^{*}-2}u_{1}-\lvert \bar u_{1}\rvert^{2^{*}-2}\bar u_{1}\big)+\beta\big(\lvert u_{2}\rvert^{\frac{2^{*}}{2}}\lvert u_{1}\rvert^{\frac{2^{*}}{2}-2}u_{1}-\lvert \bar u_{2}\rvert^{\frac{2^{*}}{2}}\lvert \bar u_{1}\rvert^{\frac{2^{*}}{2}-2}\bar u_{1}\big),  &~\text{in}~\mathbb R^N \setminus P_{1},\\[3mm]
			-\Delta v=  K_2\big(\frac{\lvert y \rvert}{\mu}\big) \big(\lvert u_{2}\rvert^{2^{*}-2}u_{2}-\lvert \bar u_{2}\rvert^{2^{*}-2}\bar u_{2}\big)+\beta\big(\lvert u_{1}\rvert^{\frac{2^{*}}{2}}\lvert u_{2}\rvert^{\frac{2^{*}}{2}-2}u_{2}-\lvert \bar u_{1}\rvert^{\frac{2^{*}}{2}}\lvert \bar u_{2}\rvert^{\frac{2^{*}}{2}-2}\bar u_{2}\big),  &~\text{in}~\mathbb R^N \setminus Q_{1},\\[3mm]
			w=w_{0} ~\text{in}~ P_{1},~v=v_{0} ~\text{in}~ Q_{1}.
		\end{cases}
	\end{equation*}
	
	(\romannumeral1) For $y\in \mathbb R^{N} \setminus (P_{1}\cup Q_1)$, there holds
	\begin{equation}\label{3.16}
		\lvert u_{1}\rvert^{2^{*}-2}u_{1}-\lvert \bar u_{1}\rvert^{2^{*}-2}\bar u_{1}=O\big(\varpi_{N-2}^{2^{*}-2} \big)w
	\end{equation}
	and
	\begin{equation*}
		\begin{aligned}
			&\lvert u_{2}\rvert^{\frac{2^{*}}{2}}\lvert u_{1}\rvert^{\frac{2^{*}}{2}-2}u_{1}-\lvert \bar u_{2}\rvert^{\frac{2^{*}}{2}}\lvert \bar u_{1}\rvert^{\frac{2^{*}}{2}-2}\bar u_{1}\\
			=&\lvert u_{2}\rvert^{\frac{2^{*}}{2}}\big(\lvert u_{1}\rvert^{\frac{2^{*}}{2}-2}u_{1}-\lvert \bar u_{1}\rvert^{\frac{2^{*}}{2}-2}\bar u_{1}\big) 
			+(\lvert u_{2}\rvert^{\frac{2^{*}}{2}}-\lvert \bar u_{2}\rvert^{\frac{2^{*}}{2}})\lvert \bar u_{1}\rvert^{\frac{2^{*}}{2}-2}\bar u_{1}\\
			=&\lvert u_{2}\rvert^{\frac{2^{*}}{2}}\big(\lvert u_{1}\rvert^{\frac{2^{*}}{2}-2}u_{1}-\lvert \bar u_{1}\rvert^{\frac{2^{*}}{2}-2}\bar u_{1}\big)+O\Big(\varpi_{N-2}^{2^*-2}\Big) v,
		\end{aligned}
	\end{equation*}
	where $\varpi_{N-2}$ is defined in \eqref{varpi1.}.

	Then by Kato’s inequality, and noting that 
	\begin{equation*}
		\begin{aligned}
			\beta\lvert u_{2}\rvert^{\frac{2^{*}}{2}}\big(\lvert u_{1}\rvert^{\frac{2^{*}}{2}-2}u_{1}-\lvert \bar u_{1}\rvert^{\frac{2^{*}}{2}-2}\bar u_{1}\big) \text{sgn}(w)\leq0,
		\end{aligned}
	\end{equation*} 
	with sgn denoting the signum function,
	we find that $\lvert w\rvert$ satisfies
	\begin{equation*}
		\begin{aligned}
			-\Delta \lvert w\rvert \leq& C\varpi_{N-2}^{2^*-2} (|w|+|v|)
			, \quad \text{in }\mathbb R^{N} \setminus (P_{1}\cup Q_1).
		\end{aligned}
	\end{equation*} 
	Similarly,
	\begin{equation*}
		\begin{aligned}
			-\Delta \lvert v\rvert \leq& C\varpi_{N-2}^{2^*-2} (|w|+|v|)
			, \quad \text{in }\mathbb R^{N} \setminus (P_{1}\cup Q_1).
		\end{aligned}
	\end{equation*} 
	Set 
	\[
	\tilde \omega =\sigma_k^{N-2}(\|w\|_{L^\infty(P_1\cup Q_1)}+\|v\|_{L^\infty(P_1\cup Q_1)})\varpi_{N-2}.
	\]
	We have 
	\[
	\tilde  \omega > |w|+|v| \text{ in } P_{1}\cup Q_1,
	\]
	and by \eqref{10'},
	\[
	-\Delta \tilde \omega \geq   C\varpi_{N-2}^{2^*-2} \tilde \omega \quad \text{in }\mathbb R^{N} \setminus (P_{1}\cup Q_1).
	\]
	
	Therefore, 
	\[
	|w| + |v| \leq \sigma_k^{N-2}(\|w\|_{L^\infty(P_1\cup Q_1)}+\|v\|_{L^\infty(P_1\cup Q_1)})\varpi_{N-2}, \quad \text{in }\mathbb R^{N} \setminus (P_{1}\cup Q_1).
	\]
	Since $w=0$ in $Q_1$, $v=0$ in $P_1$, we get the conclusion.

	(\romannumeral2) By the gradient estimate in \cite[Theorem 3.9]{GT01}, for $ y\in P_{2}\setminus \cup^{k}_{i=1}B_{\sigma_{k}+1}(x^{i})$,  we know that
	\begin{equation*}
		\begin{aligned}
			\lvert \nabla w(y)\rvert\leq& C\Big(\sup_{B_{1}(y)}\lvert w(x)\rvert+\sup_{B_{1}(y)} \Big\lvert K_1\big(\frac{\lvert \cdot \rvert}{\mu}\big) \big(\lvert u_{1}\rvert^{2^{*}-2}u_{1}-\lvert \bar u_{1}\rvert^{2^{*}-2}\bar u_{1}\big)\\
			&+\beta\big(\lvert u_{2}\rvert^{\frac{2^{*}}{2}}\lvert u_{1}\rvert^{\frac{2^{*}}{2}-2}u_{1}-\lvert \bar u_{2}\rvert^{\frac{2^{*}}{2}}\lvert \bar u_{1}\rvert^{\frac{2^{*}}{2}-2}\bar u_{1}\big) \Big\rvert\Big).
		\end{aligned}
	\end{equation*}
	From Remark \ref{dcrem}, we know that $u_{2}=0$ and $\bar u_{2}=0$ in $P_{2}$.
	By \eqref{3.16}, we obtain
	\begin{equation*}
		\begin{aligned}
			&\sup_{B_{1}(y)} \Big\lvert K_1\big(\frac{\lvert \cdot \rvert}{\mu}\big) \big(\lvert u_{1}\rvert^{2^{*}-2}u_{1}-\lvert \bar u_{1}\rvert^{2^{*}-2}\bar u_{1}\big) \Big\rvert
			\leq &C\sup_{B_{1}(y)}\big(\lvert u_{1}\rvert^{2^{*}-2}+\lvert\bar u_{1}\rvert^{2^{*}-2}\big)\lvert w\rvert.
		\end{aligned}
	\end{equation*}
	Thus,
	\begin{equation*}
		\begin{aligned}
			\lvert \nabla w(y)\rvert\leq C\sup_{B_{1}(y)}\lvert w(x)\rvert.
		\end{aligned}
	\end{equation*}
	We can obtain the estimate for $\nabla v$ similarly.	
\end{proof}
\begin{Rem}
	From \eqref{3.5.1} and \eqref{3.5.2}, we know that there is at most one solution $(u_{1},u_{2})$ to the problem $\eqref{eq15}$ satisfying \eqref{boundary} and \eqref{u}.
\end{Rem}

\medskip

Next, we will proof the existence of a solution $(u_{1},u_{2})$ to problem $\eqref{eq15}$ satisfying \eqref{boundary} and \eqref{u}. 
\subsection{The Existence of Minimizer} 

Let $A>0$ be the constant given in Lemma \ref{lem7}. We consider solutions $(u_{1},u_{2})$ to problem \eqref{eq15} satisfying  \eqref{boundary} and 
\begin{equation}\label{2A}
	|u_1(y)| \leq2A \sum_{i=1}^k \frac{1}{(1 + |y - x^i|)^{N-2}}, \quad
	|u_2(y)| \leq 2A \sum_{i=1}^k \frac{1}{(1 + |y - y^i|)^{N-2}}, \quad \forall y \in \mathbb{R}^N.
\end{equation}
We note that condition \eqref{2A} implies \eqref{u}.

\medskip
For any  $(\varphi_{0},\psi_{0})\in \Lambda_{k}\cap (D^{1,2}(\R^N))^{2}$,
denote 
\begin{equation*}
	\mathbb S(\varphi_{0},\psi_{0}): =\Big\{ (u_{1},u_{2})
	\in (D^{1,2}(\mathbb R^{N})\cap H_s)^{2}  \mid   (u_{1}, u_{2}) \text{ satisfies }\eqref{boundary} \text{ and }\eqref{2A} 
	\Big\}.
\end{equation*}
This set is a closed subset of $(D^{1,2}(\mathbb{R}^N))^2$ under the induced product norm
\[
\|(u_1,u_2)\|^2_{(D^{1,2})^2}=\|\nabla u_1\|_2^2+\|\nabla u_2\|_2^2.
\]

Define a functional $\Gamma\colon\mathbb S(\varphi_{0},\psi_{0})\rightarrow \mathbb{R}$ as
\[ 
	\begin{aligned}
		\Gamma(u_{1},u_{2})=&\frac{1}{2}\int_{\mathbb R^{N}} \big(\lvert\nabla u_{1}\rvert^{2}+\lvert\nabla u_{2}\rvert^{2}\big)-\frac{1}{2^{*}}\int_{\mathbb R^{N}}\Big(K_1\Big(\frac{\lvert y\rvert}{\mu}\Big)\lvert u_1\rvert^{2^{*}}+K_2\Big(\frac{\lvert y\rvert}{\mu}\Big)\lvert u_2\rvert^{2^{*}}\Big)\\
		&-\frac{2\beta}{2^{*}}\int_{\mathbb R^{N}}\lvert u_1\rvert^{\frac{2^{*}}{2}}\lvert u_2\rvert^{\frac{2^{*}}{2}},~~(u_{1},u_{2})\in \mathbb S(\varphi_{0},\psi_{0}).
	\end{aligned}
\]
The solution to \eqref{eq15} can be obtained by solving the following minimizing problem
\begin{equation}\label{eq16}
	\min\big\{\Gamma(u_{1},u_{2}) \mid (u_{1},u_{2}) \in \mathbb S(\varphi_{0},\psi_{0})\big\}.
\end{equation}

\begin{Lem}\label{lem3}
	The minimization problem \eqref{eq16} is attained by a solution to \eqref{eq15} in   $\mathbb S(\varphi_{0},\psi_{0})$.
\end{Lem}
\begin{proof}
	First note that 
	$\mathbb S(\varphi_{0},\psi_{0})$ is closed and convex in $\mathbb E$. Hence,
	it is weakly closed.
	To show the minimization problem \eqref{eq16} is attained, we show that 
	the functional $\Gamma$ is coercive and lower semicontinuous on $\mathbb S(\varphi_{0},\psi_{0})$.
	
	(\romannumeral1) Coercivity.
	By \eqref{2A}, there exists $C>0$  such that
	\begin{equation*}
		\begin{aligned}
			-\frac{1}{2^{*}}\int_{\mathbb R^{N}\setminus P_{1}}K_1\Big(\frac{\lvert y\rvert}{\mu}\Big)\lvert u_1\rvert^{2^{*}}
			-\frac{1}{2^{*}}\int_{\mathbb R^{N}\setminus Q_{1}}K_2\Big(\frac{\lvert y\rvert}{\mu}\Big)\lvert u_2\rvert^{2^{*}}
			>-C.	
		\end{aligned}
	\end{equation*}

	By $\beta<0$, there holds
	\begin{equation*}
		\begin{aligned}
			\Gamma(u_{1},u_{2}) 
			>&\frac{1}{2}\int_{\mathbb R^{N}} \big(\lvert\nabla u_{1}\rvert^{2}+\lvert\nabla u_{2}\rvert^{2}\big)-C.		
		\end{aligned}
	\end{equation*}

	Then, $\Gamma(u_{1},u_{2})\rightarrow +\infty$ as $\|(u_1,u_2)\|_{(D^{1,2})^2}\rightarrow\infty$.
	
	(\romannumeral2) $\Gamma$ is weakly lower semi-continuous on $\mathbb S(\varphi_{0},\psi_{0})$. 
	
	Since $\beta < 0$, Fatou’s Lemma implies that   
	$-\frac{2\beta}{2^{*}}\int_{\mathbb R^{N}}\lvert u_1\rvert^{\frac{2^{*}}{2}}\lvert u_2\rvert^{\frac{2^{*}}{2}}$ is weakly lower semi-continuous on $\mathbb S(\varphi_{0},\psi_{0})$.
	We now consider the functional   $I(u_1,u_2)=\Gamma(u_1,u_2)+\frac{2\beta}{2^{*}}\int_{\mathbb R^{N}}\lvert u_1\rvert^{\frac{2^{*}}{2}}\lvert u_2\rvert^{\frac{2^{*}}{2}}$, which is of class $C^2$.
	For $(\xi,\eta)\in D_0^{1,2}(\R^N\setminus P_1) \times D_0^{1,2}(\R^N\setminus Q_1) \cap (H_s)^2$,
	we have 
	\begin{align*}
		I''(u_1,u_2)[(\xi,\eta),(\xi,\eta) ] \geq& 
		\|\nabla \xi\|^2_{2}+\|\nabla \eta\|^2_{2}-C\int_{\R^N\setminus P_1} |u_1|^{2^*-2}\xi^2-C\int_{\R^N\setminus Q_1} |u_2|^{2^*-2}\eta^2\\
		\geq& (1-C\|u_1\|_{L^{2^*}(\R^N\setminus P_1)}^{2^*-2}-C \|u_2\|_{L^{2^*}(\R^N\setminus Q_1)}^{2^*-2}) \|(\xi,\eta)\|^2_{(D^{1,2})^2}
	\end{align*}
	By \eqref{2A}, we know that $I(u_1,u_2)$ is convex for large $k$. Therefore, it is weakly lower semicontinuous.
	
	(iii)
	Suppose the minimization problem \eqref{eq16} is achieved by some $(u_1, u_2) \in \mathbb{S}(\varphi_0, \psi_0)$.  
	We now show that this minimizer is in fact a solution to \eqref{eq15}.
	
	Since $(|u_1|, |u_2|) \in \mathbb{S}(\varphi_0, \psi_0)$ and $\Gamma(u_1, u_2) = \Gamma(|u_1|, |u_2|)$, it follows that $(|u_1|, |u_2|)$ is also a minimizer of $\Gamma$ over $\mathbb{S}(\varphi_0, \psi_0)$.
	
	Let $\xi \in C_0^\infty(\mathbb{R}^N \setminus P_1) \cap H_s$ with $\xi \geq 0$. Then for sufficiently small $t > 0$, we have $(|u_1| - t\xi, |u_2|) \in \mathbb{S}(\varphi_0, \psi_0)$. Therefore,
	$$
	\Gamma'(|u_1|, |u_2|)(\xi, 0)
	= \lim_{t \to 0^+} \frac{ \Gamma(|u_1|, |u_2|)-\Gamma(|u_1| - t\xi, |u_2|) }{t}
	\leq 0.
	$$
	
	This implies that
	$$
	-\Delta |u_1| \leq K_1\left(\frac{|y|}{\mu}\right) |u_1|^{2^* - 1} + \beta |u_2|^{\frac{2^*}{2}} |u_1|^{\frac{2^*}{2} - 1}, \quad \text{in } \mathbb{R}^N \setminus P_1.
	$$
	
	Since $\beta < 0$, the second term on the right-hand side is non-positive, so the inequality above implies that $|u_1|$ satisfies the first inequality in \eqref{subsolution}. Similarly, $|u_2|$ satisfies the second inequality.
	
	Now, applying Lemma \ref{lem7}, we deduce that the bounds in \eqref{2A} are strictly satisfied. 
	Therefore, $(u_1,u_2)$ is a solution to \eqref{eq15}.
\end{proof}

Since $\Lambda_k \cap (D^{1,2}(\R^N))^2$ is dense in $\Lambda_k$ with respect to the norm $\|\cdot\|_\infty$, and the a priori estimates from Lemma \ref{lem9} hold uniformly, we may extend the solvability of \eqref{eq15} from the dense subspace to all of  $\Lambda_k$.  
Consequently, for every $(\varphi_0, \psi_0)\in \Lambda_k$, there exists a unique solution $(u_1, u_2)$ to problem \eqref{eq15} satisfying   \eqref{boundary} and   \eqref{u}.

This allows us to introduce   the operator $S$ as follows:
\begin{Def}\label{def3.3}
	Let $(u_{1},u_{2})$  denote the unique solution to \eqref{eq15} corresponding to
	$(\varphi_{0},\psi_{0})\in \Lambda_{k}$, satisfying \eqref{boundary} and \eqref{u}. We define the operator
	$S:\Lambda_k\to \mathbb E$ by 
	\[ S(\varphi_{0},\psi_{0}):=(u_{1},u_{2})-(U_{r,\lambda},V_{\rho,\nu}).\]
\end{Def}
\begin{Rem}\label{rem3.11}
	When $(\varphi_{0},\psi_{0})\in \Lambda_{k}\cap (D^{1,2}(\R^N))^2$, 
	by Lemma \ref{lem3},  $(u_1, u_2)\in (D^{1,2}(\R^N))^2$,
	and hence,
	$S(\varphi_{0},\psi_{0})\in (D^{1,2}(\R^N))^2$.
\end{Rem}

\section{Reduction}\label{sec:4}

In this section, we are now ready to carry out the reduction for system \eqref{eq5}. We consider the following problem 
\begin{equation}\label{eq23}
	\begin{cases}
		-\Delta \varphi-(2^{*}-1)K_1\big(\frac{\lvert y \rvert}{\mu}\big)U^{2^{*}-2}_{r,\lambda}\varphi
		=l_{1}+R_{1}(\varphi)+ D_1(\varphi,\psi)
		+\sum_{j=1}^{2}c_{j}Y_j,\\[3mm]
		-\Delta \psi-(2^{*}-1)K_2\big(\frac{\lvert y \rvert}{\mu}\big)V^{2^{*}-2}_{\rho,\nu}\psi=l_{2}+R_{2}(\psi)+D_2(\varphi,\psi)+\sum_{j=1}^{2}d_{j}Z_j,\\[3mm]
		(\varphi,\psi)\in \mathbb E,
	\end{cases}
\end{equation}
where 
\begin{gather*} 
	l_1=	K_1\big(\frac{\lvert y \rvert}{\mu}\big)U^{2^{*}-1}_{r,\lambda}-\sum_{i=1}^{k}U^{2^{*}-1}_{x^{i},\lambda},\quad
	l_2=	K_2\big(\frac{\lvert y \rvert}{\mu}\big)V^{2^{*}-1}_{\rho,\nu}-\sum_{i=1}^{k}V^{2^{*}-1}_{y^{i},\nu}\\
	R_{1}(\varphi)=	K_1\big(\frac{\lvert y \rvert}{\mu}\big)\lvert U_{r,\lambda}+\varphi \rvert^{2^{*}-2}(U_{r,\lambda}+\varphi)-K_1\big(\frac{\lvert y \rvert}{\mu}\big)U^{2^{*}-1}_{r,\lambda}-(2^{*}-1)K_1\big(\frac{\lvert y \rvert}{\mu}\big)U^{2^{*}-2}_{r,\lambda}\varphi,\\
	R_{2}(\psi)=	K_2\big(\frac{\lvert y \rvert}{\mu}\big)\lvert V_{\rho,\nu}+\psi \rvert^{2^{*}-2}(V_{\rho,\nu}+\psi)-K_2\big(\frac{\lvert y \rvert}{\mu}\big)V^{2^{*}-1}_{\rho,\lambda}-(2^{*}-1)K_2\big(\frac{\lvert y \rvert}{\mu}\big)V^{2^{*}-2}_{\rho,\nu}\psi,\\
	D_{1}(\varphi,\psi) =\beta \lvert \psi+V_{\rho,\nu}\rvert^{\frac{2^{*}}{2}}\lvert \varphi+U_{r,\lambda}\rvert^{\frac{2^{*}}{2}-2}(\varphi+U_{r,\lambda}),\\
	D_2(\varphi,\psi)=\beta \lvert \varphi+U_{r,\lambda}\rvert^{\frac{2^{*}}{2}}\lvert \psi+V_{\rho,\nu}\rvert^{\frac{2^{*}}{2}-2}(\psi+V_{\rho,\nu}).
\end{gather*}
By the definition of $S$
(Definition \ref{def3.3}), if $(\varphi_{0},\psi_{0})$ in $\Lambda_{k}$ solves system \eqref{eq23}, then $(\varphi_{0},\psi_{0})\in S(\Lambda_{k})$. Next we are sufficed to solve system \eqref{eq23} in $S(\Lambda_{k})$.

Let $(\chi_{1},\chi_{2})\in C^{1}_{0}(\mathbb R^{N})\times C^{1}_{0}(\mathbb R^{N})$ be the fixed truncation functions such that 
\begin{equation}\label{jd}
	\begin{cases}
		\chi_{1}=1,~\text{in}~\mathbb R^N \setminus \cup^{k}_{i=1}B_{\frac{3}{4}\sigma^{2}_{k}}(x^{i}),\\[3mm]
		\chi_{1}=0,~\text{in}~\cup^{k}_{i=1}B_{\frac{1}{4}\sigma^{2}_{k}}(x^{i}),\\[3mm]
		\lvert\nabla\chi_{1}\rvert\leq \frac{4}{\sigma^{2}_{k}},~
		\lvert\Delta\chi_{1}\rvert\leq \frac{8N}{\sigma^{4}_{k}},
	\end{cases}
	~\text{and}~
	\begin{cases}
		\chi_{2}=1,~\text{in}~\mathbb R^N \setminus \cup^{k}_{i=1}B_{\frac{3}{4}\sigma^{2}_{k}}(y^{i}),\\[3mm]
		\chi_{2}=0,~\text{in}~\cup^{k}_{i=1}B_{\frac{1}{4}\sigma^{2}_{k}}(y^{i}),\\[3mm]
		\lvert\nabla\chi_{2}\rvert\leq \frac{4}{\sigma^{2}_{k}},~
		\lvert\Delta\chi_{2}\rvert\leq \frac{8N}{\sigma^{4}_{k}}.
	\end{cases}
\end{equation} 
For any
$(\varphi_0, \psi_0)\in S(\Lambda_k)$, 
$(u_{1},u_{2})=(\varphi_{0},\psi_{0})+(U_{r,\lambda},V_{\rho,\nu})$ is the solution of system \eqref{eq15}. Therefore,  
\begin{equation}\label{4.5}
	\begin{aligned}
		\begin{cases}
			-\Delta \varphi_{0}-(2^{*}-1)K_1\big(\frac{\lvert y \rvert}{\mu}\big)U^{2^{*}-2}_{r,\lambda}\varphi_{0}=
			l_{1}+R_{1}(\varphi_{0})+D_1(\varphi_0,\psi_0),~\text{in}~\mathbb R^N \setminus P_{1},\\[3mm]
			-\Delta \psi_{0}-(2^{*}-1)K_2\big(\frac{\lvert y \rvert}{\mu}\big)V^{2^{*}-2}_{\rho,\nu}\psi_{0}= l_{2}+R_{2}(\psi_{0})+D_2(\varphi_0,\psi_0),~\text{in}~\mathbb R^N \setminus Q_{1}.
		\end{cases}
	\end{aligned}
\end{equation}
Then we have, in $\R^N$,
\begin{equation}\label{xin3}
	\begin{aligned}
		-\Delta (\chi_{1}\varphi_{0})-(2^{*}-1)K_1\big(\frac{\lvert y \rvert}{\mu}\big)U^{2^{*}-2}_{r,\lambda} (\chi_{1}\varphi_{0})  
		= \chi_{1}\Big(l_{1}+R_{1}(\varphi_{0})+D_1(\varphi_0,\psi_0)\Big)
		-\varphi_{0}\Delta\chi_{1}-2\nabla\chi_{1}\nabla \varphi_{0}
	\end{aligned}
\end{equation}
and
\begin{equation}\label{xin4}
	\begin{aligned}
		-\Delta (\chi_{2}\psi_{0})-(2^{*}-1)K_2\big(\frac{\lvert y \rvert}{\mu}\big)V^{2^{*}-2}_{\rho,\nu}(\chi_{2}\psi_{0})
		=\chi_{2}\Big(l_{2}+R_{2}(\psi_{0})+D_2(\varphi_0,\psi_0)\Big)
		-\psi_{0}\Delta\chi_{2}-2\nabla\chi_{2}\nabla \psi_{0}
		.
	\end{aligned}
\end{equation}
If $(\varphi_0, \psi_0)$ further satisfies the systems \eqref{eq23}, by Remark \ref{dcrem}, \eqref{xin3} and \eqref{xin4}, we get
\begin{equation}\label{4.}
	\begin{aligned}
		&-\Delta (\varphi_{0}-\chi_{1}\varphi_{0})-(2^{*}-1)K_1\big(\frac{\lvert y \rvert}{\mu}\big)U^{2^{*}-2}_{r,\lambda} (\varphi_{0}-\chi_{1}\varphi_{0})\\
		=&(1-\chi_{1})\Big(l_{1}+R_{1}(\varphi_{0})+D_1(\varphi_0,\psi_0)\Big)
		+\varphi_{0}\Delta\chi_{1}+2\nabla\chi_{1}\nabla \varphi_{0} 
		+\sum_{j=1}^{2}c_{j} Y_{j}\\
		=&(1-\chi_{1})\big(l_{1}+R_{1}(\varphi_{0})\big)
		+\varphi_{0}\Delta\chi_{1}+2\nabla\chi_{1}\nabla \varphi_{0}
		+\sum_{j=1}^{2}c_{j} Y_{j}
	\end{aligned}
\end{equation}
and
\begin{equation}\label{4..}
	\begin{aligned}
		&-\Delta (\psi_{0}-\chi_{2}\psi_{0})-(2^{*}-1)K_2\big(\frac{\lvert y \rvert}{\mu}\big)V^{2^{*}-2}_{\rho,\nu}(\psi_{0}-\chi_{2}\psi_{0})\\
		=&(1-\chi_{2})\Big(l_{2}+R_{2}(\psi_{0})+D_2(\varphi_0,\psi_0)\Big)
		+\psi_{0}\Delta\chi_{2}+2\nabla\chi_{2}\nabla \psi_{0} 
		+\sum_{j=1}^{2}d_{j} Z_{j}\\
		=&(1-\chi_{2})\big(l_{2}+R_{2}(\psi_{0})\big)
		+\psi_{0}\Delta\chi_{2}+2\nabla\chi_{2}\nabla \psi_{0}+\sum_{j=1}^{2}d_{j} Z_{j}.
	\end{aligned}
\end{equation}
\smallskip 
This motivates us to solve the 
fixed point for the operator $T$ defined as follows:
\begin{Def}
	Let $S(\Lambda_{k})$ be the complete metric space   equipped with the distance
	\begin{equation*}
		d\big((\varphi_{0}, \psi_{0}),(\bar\varphi_{0}, \bar\psi_{0})\big)
		= \|\varphi_{0}-\bar\varphi_{0}\|_{L^{\infty}(P_{1})}
		+ \|\psi_{0}-\bar\psi_{0}\|_{L^{\infty}(Q_{1})}.
	\end{equation*}
	For any $(\varphi_{0},\psi_{0})\in S(\Lambda_{k})$, define the map $T: S(\Lambda_k) \to \mathbb E$ by
	\begin{equation*}
		T(\varphi_{0},\psi_{0}) :=S
		\begin{pmatrix}
			L_{1}\!\Big( \varphi_{0}\Delta\chi_{1} + 2\nabla\chi_{1}  \nabla \varphi_{0} + (1-\chi_{1})\big(l_{1}+R_{1}(\varphi_{0})\big) \Big) \\
			L_{2}\!\Big( \psi_{0}\Delta\chi_{2} + 2\nabla\chi_{2}  \nabla \psi_{0} + (1-\chi_{2})\big(l_{2}+R_{2}(\psi_{0})\big) \Big)
		\end{pmatrix}^{\top}.
	\end{equation*}
\end{Def}
The following Lemmas~\ref{lem10}, \ref{lem11} and Corollary \ref{cor4.4} establish the well-definedness of $T$.
\begin{Lem}\label{lem10}
	Assume that $\lvert\lvert x^{1}\rvert-\mu r_{0}\rvert\leq1$ and $\lvert\lvert y^{1}\rvert-\mu \rho_{0}\rvert\leq1$. If $N\geq5$, then 
	\begin{equation*}
		\begin{aligned}
			\lVert l_{1}\rVert_{L^{\infty}(P_{2})}\leq C\Big(\frac{1}{\mu}\Big)^{m},~\lVert l_{2}\rVert_{L^{\infty}(Q_{2})}\leq C\Big(\frac{1}{\mu}\Big)^{m}.
		\end{aligned}
	\end{equation*}	
	
\end{Lem}
\begin{proof}
	We are sufficed to deal with $\lVert l_{1}\rVert_{L^{\infty}(P_{2})}$, since for $\lVert l_{2}\rVert_{L^{\infty}(Q_{2})}$, one can obtain similar estimates just in the same way.
	Recalling the definition of $l_{1}$, we have
	\[
	l_{1}
	=K_1\big(\frac{\lvert y \rvert}{\mu}\big)U^{2^{*}-1}_{r,\lambda}-\sum_{i=1}^{k}U^{2^{*}-1}_{x^{i},\lambda}=J_{1}+J_{2}\]
	where 
	\[J_1:= K_1\big(\frac{\lvert y \rvert}{\mu}\big)\Big(U^{2^{*}-1}_{r,\lambda}-\sum_{i=1}^{k}U^{2^{*}-1}_{x^{i},\lambda}\Big),\quad  J_2:=\Big(K_1\big(\frac{\lvert y \rvert}{\mu}\big)-1\Big)\sum_{i=1}^{k}U^{2^{*}-1}_{x^{i},\lambda}.
	\]
	By symmetry, we might as well assume that 
	$y\in B_{\sigma^{2}_{k}}(x^{1})$, then $\lvert y-x^{i}\rvert\geq\lvert y-x^{1}\rvert,~~i=2,\dots,k$. Therefore,
	\begin{equation*}
		\lvert  J_{1} \rvert \leq C\frac{1}{(1+\lvert y-x^{1}\rvert)^{4}}\sum_{i=2}^{k}\frac{1}{(1+\lvert y-x^{i}\rvert)^{N-2}}+C\Big(\sum_{i=2}^{k}\frac{1}{(1+\lvert y-x^{i}\rvert)^{N-2}}\Big)^{2^*-1}.
	\end{equation*}
	By \eqref{tau1a}, we obtain
	\[	
	\Big(\sum_{i=2}^{k}\frac{1}{(1+\lvert y-x^{i}\rvert)^{N-2}}\Big)^{2^*-1}
	\leq C\sum_{i=2}^{k}\frac{1}{(1+\lvert y-x^{i}\rvert)^{N+2-\frac{4}{N-2}\tau_{1}}}.
	\]

	Thus, by \eqref{key3'},
	\begin{equation*}
		\lVert  J_{1} \rVert_{L^{\infty}(P_{2})} \leq C\sum_{i=2}^{k}\frac{1}{\lvert x^{i}-x^{1}\rvert^{N-2}}+C\sum_{i=2}^{k}\frac{1}{\lvert x^{i}-x^{1}\rvert^{N+2-\frac{4}{N-2}\tau_{1}}}\leq C\Big(\frac{1}{\mu}\Big)^{m}.
	\end{equation*}
	
	Next we estimate
	$J_2$.
	For $y\in B_{\sigma^{2}_{k}}(x^{1})$, we have
	\begin{equation}\label{eq27}
		\begin{aligned}
			\Big \lvert \Big(K_1\big(\frac{\lvert y \rvert}{\mu}\big)-1\Big)\sum_{i=2}^{k}U^{2^{*}-1}_{x^{i},\lambda} \Big \rvert\leq C\sum_{i=2}^{k}\frac{1}{(1+\lvert y-x^{i}\rvert)^{N+2}}\leq C\sum_{i=2}^{k}\frac{1}{\lvert x^{i}-x^{1}\rvert^{N+2}}\leq C\Big(\frac{1}{\mu}\Big)^{m}.
		\end{aligned}
	\end{equation}
	If $y\in B_{\sigma^{2}_{k}}(x^{1})$, there holds
	\begin{equation*}
		\begin{aligned}
			\Big \lvert K_1\big(\frac{\lvert y \rvert}{\mu}\big) -1 \Big \rvert
			\leq& C\Big \lvert \frac{\lvert y \rvert}{\mu} -r_{0}\Big \rvert^{m}
			\leq \frac{C}{\mu^{m}}\Big(\lvert \lvert y \rvert-\lvert x^{1} \rvert\rvert^{m}+\lvert \lvert x^{1} \rvert-r_{0}\mu\rvert^{m}\Big)
			\leq \frac{C}{\mu^{m}}\lvert  y - x^{1} \rvert^{m}.
		\end{aligned}
	\end{equation*}
	Thus,
	\begin{equation}\label{eq28}
		\begin{aligned}
			\Big \lvert K_1\big(\frac{\lvert y \rvert}{\mu}\big) -1 \Big \rvert U^{2^{*}-1}_{x^{1},\lambda} 
			\leq\frac{C}{\mu^{m}} \frac{\lvert  y - x^{1} \rvert^{m}}{(1+\lvert y-x^{1}\rvert)^{N+2}}\leq\frac{C}{\mu^{m}} \frac{1}{(1+\lvert y-x^{1}\rvert)^{N+2-m}}\leq C\Big(\frac{1}{\mu}\Big)^{m}.
		\end{aligned}
	\end{equation}  
	
	Hence, \eqref{eq27} and \eqref{eq28} imply that
	\begin{equation*}
		\lVert J_{2}\rVert_{L^{\infty}( P_{2})}\leq C\Big(\frac{1}{\mu}\Big)^{m}.
	\end{equation*}
\end{proof}

\begin{Lem}\label{lem11} 
	Suppose $N\geq5$.
	For $(\varphi_{0},\psi_{0})\in S(\Lambda_{k})$, we obtain
	\begin{equation*}
		\begin{aligned}
			\lVert R_{1}(\varphi_{0})\rVert_{L^{\infty}( P_{2})}\leq C\Big(\frac{1}{\mu}\Big)^{m},~~\text{and}~~\lVert R_{2}(\psi_{0})\rVert_{L^{\infty}(Q_{2})}\leq C\Big(\frac{1}{\mu}\Big)^{m}.
		\end{aligned}
	\end{equation*}	
\end{Lem}

\begin{proof} 
	If $N=5$, we get $\tau_{1}=\frac{N-2-m}{N-2}=\frac{3-m}{3}\leq\frac{1}{3}$. By Lemma \eqref{key3}, there holds
	\begin{equation*}
		U^{\frac{1}{3}}_{r,\lambda}\leq C\sum_{i=1}^{k}\frac{1}{(1+\lvert y-x^{i}\rvert)}<C\sum_{i=1}^{k}\frac{1}{(1+\lvert y-x^{i}\rvert)^{\tau_{1}}}\leq C.				
	\end{equation*}
	
	Then we have
	\begin{equation*}
		\lvert R_{1}(\varphi_{0})\rvert\leq
		C U^{\frac{1}{3}}_{r,\lambda}\lvert \varphi_{0}\rvert^{2}+C\lvert \varphi_{0}\rvert^{\frac{7}{3}}\leq C \lvert \varphi_{0}\rvert^{2},~N=5.
	\end{equation*}	
	
	For $N\geq 6$ and $\delta_2>0$ small enough, we obtain $\min_{P_{2}}\{ U_{r,\lambda}\}>\frac{1}{\sigma^{2N-4+\delta_2}_{k}}$. Then there exists $C>0$ independent of $k$ such that 
	$\frac{\varphi_{0}}{U_{r,\lambda}}\leq C$ in $P_{2}$. Thus,
	\begin{equation*}
		\lvert R_{1}(\varphi_{0})\rvert\leq C U^{2^{*}-3}_{r,\lambda}\lvert \varphi_{0}\rvert^{2},~N\geq 6.
	\end{equation*}	
	
	By Corollary \ref{cor3.3}, we obtain
	\begin{equation*} 
		\begin{aligned}
			\frac{\lvert \varphi_{0}\rvert^{2}}{U^{3-2^{*}}_{r,\lambda}}=O\Big(\frac{1}{\mu^{m+2\theta}}\Big)(1+|y-x^1|)^{-6+\tau_1},\quad  y\in B_{\sigma_k^2}(x^1).
		\end{aligned}	
	\end{equation*}
	Therefore,
	\begin{equation*}
		\lVert R_{1}(\varphi_{0})\rVert_{L^{\infty}(P_{2})}\leq C\Big(\frac{1}{\mu}\Big)^{m}.
	\end{equation*}
	Similarly, we can get the estimate of
	$\lVert R_{2}(\psi_{0})\rVert_{L^{\infty}(Q_{2})}$.
\end{proof}
We can show the map $T: S(\Lambda_k)\to S(\Lambda_k)$ is well-defined; that is,
\begin{Cor} \label{cor4.4}
	For  large $k$, there holds
	\begin{equation*}
		\begin{aligned}	
			\left(\begin{matrix}L_{1}\big(\varphi_{0}\Delta\chi_{1}+2\nabla\chi_{1} \nabla \varphi_{0}+(1-\chi_{1})\big(l_{1}+R_{1}(\varphi_{0})
				\big)\big)\\L_{2}\big(\psi_{0}\Delta\chi_{2}+2\nabla\chi_{2} \nabla \psi_{0}+(1-\chi_{2})\big(l_{2}+R_{2}(\psi_{0})
				\big)\big)\end{matrix}\right)^{\top} \in\Lambda_{k}.
		\end{aligned}	
	\end{equation*}
\end{Cor}
\begin{proof}

	Note that supp$(1-\chi_{1})\subset P_{2}$ and supp$\nabla\chi_{1}\subset \big(\cup^{k}_{i=1}B_{\frac{3}{4}\sigma^{2}_{k}}(x^{i})\setminus\cup^{k}_{i=1}B_{\frac{1}{4}\sigma^{2}_{k}}(x^{i})\big)$. By Lemma \ref{lem8}, we get
	\begin{equation*}
		\begin{aligned}	
			\lVert \varphi_{0}\Delta\chi_{1}+2\nabla\chi_{1} \nabla \varphi_{0}\rVert_{L^{\infty}(P_{2})} 
			\leq C\frac{1}{\mu^{\frac{m}{2}+\theta}}\frac{1}{\sigma^{N+2}_{k}}.
		\end{aligned}	
	\end{equation*}
	From Lemmas \ref{lem10} and \ref{lem11}, using Proposition \ref{Prop2.2}, we obtain
	\begin{equation*}
		\begin{aligned}	
			&\lVert L_{1}\big(\varphi_{0}\Delta\chi_{1}+2\nabla\chi_{1} \nabla \varphi_{0}+(1-\chi_{1})\big(l_{1}+R_{1}(\varphi_{0})
			\big)\big)\rVert_{L^{\infty}(\mathbb R^N)}\\
			\leq&\sigma^{4}_{k}\lVert \varphi_{0}\Delta\chi_{1}+2\nabla\chi_{1} \nabla \varphi_{0}\rVert_{L^{\infty}(P_{2})}
			+\sigma^{4}_{k}\lVert l_{1}+R_{1}(\varphi_{0})\rVert_{L^{\infty}(P_{2})}
			\leq \frac{1}{2}\frac{1}{\mu^{\frac{m}{2}}}\frac{1}{\sigma_k^{\frac{N+1}{2}}}.
		\end{aligned}	
	\end{equation*}
	
	Similarly, 
	\begin{equation*}
		\begin{aligned}	
			\lVert L_{2}\big(\psi_{0}\Delta\chi_{2}+2\nabla\chi_{2} \nabla \psi_{0}+(1-\chi_{2})\big(l_{2}+R_{2}(\psi_{0})
			\big)\big)\rVert_{L^{\infty}(\mathbb R^N)}
			\leq \frac{1}{2}\frac{1}{\mu^{\frac{m}{2}}}\frac{1}{\sigma_k^{\frac{N+1}{2}}}.
		\end{aligned}	
	\end{equation*}
	
	The proof is complete.
\end{proof}

\begin{Lem}\label{lem13}
	The operator $T:S(\Lambda_{k})\rightarrow S(\Lambda_{k})$ is a contraction:
	\begin{equation*}
		d(T(\varphi_{0}, \psi_{0}),T(\bar\varphi_{0}, \bar\psi_{0}))\leq \frac{1}{2}d((\varphi_{0}, \psi_{0}),(\bar\varphi_{0}, \bar\psi_{0})).
	\end{equation*}
\end{Lem}

\begin{proof}
	
	Take $(\varphi_{0}, \psi_{0})\in S(\Lambda_{k})$, $(\bar\varphi_{0}, \bar\psi_{0})\in S(\Lambda_{k})$.
	Set
	\begin{equation*}
		(u_{1},u_{2})=(\varphi_{0},\psi_{0})+(U_{r,\lambda},V_{\rho,\nu}),~~ (\bar u_{1},\bar u_{2})=(\bar\varphi_{0},\bar\psi_{0})+(U_{r,\lambda},V_{\rho,\nu})
	\end{equation*}
	\begin{equation*}
		(w_{0},v_{0})=(\varphi_{0}, \psi_{0})-(\bar\varphi_{0}, \bar\psi_{0})=(u_{1},u_{2})-(\bar u_{1},\bar u_{2}),
	\end{equation*}
	\begin{equation*}
		\begin{aligned}
			W
			=
			L_{1}\Big((1-\chi_{1})\big(R_{1}(\varphi_{0})-R_{1}(\bar \varphi_{0})
			\big)
			+w_{0}\Delta\chi_{1}+2\nabla\chi_{1} \nabla w_{0}\Big)
		\end{aligned}
	\end{equation*}
	and
	\begin{equation*}
		\begin{aligned}
			V
			=
			L_{2}\Big((1-\chi_{2})\big(R_{2}(\psi_{0})-R_{2}(\bar \psi_{0})
			\big)
			+v_{0}\Delta\chi_{2}+2\nabla\chi_{2} \nabla v_{0}\Big).
		\end{aligned}
	\end{equation*}

	Recalling
	the definition of $S$, we only need to show
	\begin{equation*}
		\begin{aligned}	
			\lVert W \rVert_{L^{\infty}(P_{1})}+\lVert V \rVert_{L^{\infty}(Q_{1})}
			\leq\frac{1}{2}(\lVert w_{0}\rVert_{L^{\infty}(P_{1})}+\lVert v_{0}\rVert_{L^{\infty}(Q_{1})}).
		\end{aligned}	
	\end{equation*}
	
	By Proposition \ref{Prop2.2},
	we know 
	\begin{equation}\label{W}
		\|W\|_{L^\infty(P_1)}\leq C\sigma_k^4\left(
		\| R_{1}(\varphi_{0})-R_{1}(\bar \varphi_{0})
		\|_{L^\infty(P_2)}
		+\|w_{0}\Delta\chi_{1}+2\nabla\chi_{1} \nabla w_{0}\|_{L^\infty(P_2)}\right).
	\end{equation}	
	
	By Lemma \ref{lem9}, when $N=5$, it follows that
	\begin{equation*}
		\begin{aligned}
			\lVert R_{1}(\varphi_{0})-R_{1}(\bar \varphi_{0}) \rVert_{L^{\infty}(P_{2})}
			\leq & C \lVert(\lvert \varphi_{0}\rvert+\lvert\bar \varphi_{0}\rvert)w_{0}\rVert_{L^{\infty}(P_{2})}\\
			\leq & \frac{C}{\sigma_k^{\frac{N+1}{2}}\mu^{\frac{m}{2}}}
			(\lVert w_{0}\rVert_{L^{\infty}(P_{1})}+\|v_0\|_{L^\infty(Q_1)}).
		\end{aligned}
	\end{equation*}
	When $N\geq6$, we have 
	\[ 
		\begin{aligned}
			\lvert R_{1}(\varphi_{0})-R_{1}(\bar \varphi_{0})\rvert 
			=&\Big\lvert K_1\big(\frac{\lvert y \rvert}{\mu}\big)\Big(\lvert u_{1}\rvert^{2^*-2}u_{1}-\lvert\bar u_{1}\rvert^{2^*-2}\bar u_{1}-(2^*-1)U^{2^*-2}_{r,\lambda}w_{0}\Big)\Big\rvert\\
			\leq& C\Big\lvert  \int_0^1\Big(
			\lvert t \varphi_{0}+(1-t)\bar \varphi_{0}+U_{r,\lambda}\rvert^{2^*-2}-U^{2^*-2}_{r,\lambda}\Big)w_{0}\mathrm{d} t \Big\rvert\\
			\leq&  C |(\lvert \varphi_{0}\rvert^{2^*-2}+\lvert\bar \varphi_{0}\rvert^{2^*-2})w_{0}|.
		\end{aligned}
	\]
	Therefore,
	\begin{equation}\label{w1}
		\begin{aligned}
			\lVert R_{1}(\varphi_{0})-R_{1}(\bar \varphi_{0}) \rVert_{L^{\infty}(P_{2})}
			\leq C \frac{1}{\mu^\frac{2m}{N-2}\sigma^{\frac{2(N+1)}{N-2}}_{k}}(\lVert w_{0}\rVert_{L^{\infty}(P_{1})}+\|v_0\|_{L^\infty(Q_1)}).
		\end{aligned}
	\end{equation}

	Since $\text{supp}\,\Delta\chi_{1}\cup \text{supp}\,\nabla\chi_{1}\subset  \cup^{k}_{i=1}B_{\frac{3}{4}\sigma^{2}_{k}}(x^{i})\setminus\cup^{k}_{i=1}B_{\frac{1}{4}\sigma^{2}_{k}}(x^{i})$, by Lemma \ref{lem9}, and \eqref{jd}, we obtain   
	\begin{equation}\label{w3}
		\begin{aligned}	
			\lVert w_{0} \Delta\chi_{1}+2\nabla\chi_{1} \nabla w_{0}\rVert_{L^{\infty}(P_{2})}
			=O\Big(\frac{1}{\sigma^{N}_{k}}\Big)(\lVert w_{0}\rVert_{L^{\infty}(P_{1})}+\lVert v_{0}\rVert_{L^{\infty}(Q_{1})}).
		\end{aligned}	
	\end{equation}

	Combining \eqref{W}, \eqref{w1}, and \eqref{w3}, we obtain	
	\begin{equation}\label{W1}
		\begin{aligned}
			\lVert W \rVert_{L^{\infty}(Q_{1})}
			\leq \frac{1}{4}(\lVert w_{0}\rVert_{L^{\infty}(P_{1})}+\lVert v_{0}\rVert_{L^{\infty}(Q_{1})}).
		\end{aligned}
	\end{equation}
	A similar argument establishes the estimate for $V$, which completes the proof.
\end{proof}
\smallskip
Lemma \ref{lem13} implies that $T$ has a unique fixed point in $S(\Lambda_k)$. To see it is a solution, we show

\begin{Lem}\label{lem12}   $(\varphi_{0},\psi_{0})$ is a fixed point of $T$ in $S(\Lambda_{k})$ if and only if it is a solution of \eqref{eq23} in $\Lambda_k$ satisfying
	\begin{equation*}
		\begin{aligned}
			(2^{*}-1) K_\infty \lVert \varphi_{0}+U_{r,\lambda} \rVert^{2^*-2}_{L^{2^*}(\mathbb R^{N}\setminus P_{1})}<\frac{\mathcal S}{2}~~\text{and}~~(2^{*}-1) K_\infty \lVert \psi_{0}+V_{\rho,\nu} \rVert^{2^*-2}_{L^{2^*}(\mathbb R^{N}\setminus Q_{1})}<\frac{\mathcal S}{2}.
		\end{aligned}
	\end{equation*}
	
	Furthermore, whenever either of the two equivalent conditions is satisfied, it necessarily follows that  \[(\varphi_{0},\psi_{0})\in  S(\Lambda_k) \cap \Lambda_k\cap (D^{1,2}(\R^N))^2.\]
\end{Lem}

\begin{proof}
	(i)	Assume $(\varphi_{0},\psi_{0})$ is a fixed point of $T$, then
	\begin{equation*}
		\begin{aligned}
			(\varphi_{0},\psi_{0})=
			S\left(\begin{matrix}
				L_{1}\big(\varphi_{0}\Delta\chi_{1}+2\nabla\chi_{1} \nabla \varphi_{0}+(1-\chi_{1})\big(l_{1}+R_{1}(\varphi_{0})
				\big)\big)\\
				L_{2}\big(\psi_{0}\Delta\chi_{2}+2\nabla\chi_{2} \nabla \psi_{0}+(1-\chi_{2})\big(l_{2}+R_{2}(\psi_{0})
				\big)\big)\end{matrix}\right)^{\top}.
		\end{aligned}
	\end{equation*}
	We denote
	\begin{gather*} 
		\zeta_{1}:=L_{1}\Big(\varphi_{0}\Delta\chi_{1}+2\nabla\chi_{1} \nabla \varphi_{0}+(1-\chi_{1})\big(l_{1}+R_{1}(\varphi_{0})
		\big)\Big)+\chi_{1}\varphi_{0}
		\\
		\zeta_{2}:=L_{2}\Big(\psi_{0}\Delta\chi_{2}+2\nabla\chi_{2} \nabla \psi_{0}+(1-\chi_{2})\big(l_{2}+R_{2}(\psi_{0})
		\big)\Big)+\chi_{2}\psi_{0}
		. 
	\end{gather*}
	By Proposition \ref{Prop2.2} and Remark \ref{rem3.11},
	we know $(\varphi_{0},\psi_{0}) \in (D^{1,2}(\R^N))^2$ and $(\zeta_{1}, \zeta_2)\in (D^{1,2}(\R^N))^2$.
	Recalling the cut-off functions $(\chi_{1},\chi_{2})$ defined in \eqref{jd}, and the definition of $S$, we obtain
	\begin{gather*}
		\zeta_{1}=L_{1}\Big(\varphi_{0}\Delta\chi_{1}+2\nabla\chi_{1} \nabla \varphi_{0}+(1-\chi_{1})\big(l_{1}+R_{1}(\varphi_{0})
		\big)\Big) =\varphi_0 ~\text{in}~P_{1},\\
		\zeta_{2}=L_{2}\Big(\psi_{0}\Delta\chi_{2}+2\nabla\chi_{2} \nabla \psi_{0}+(1-\chi_{2})\big(l_{2}+R_{2}(\psi_{0})
		\big)\Big) =\psi_0~\text{in}~Q_{1}
		.
	\end{gather*}
	By the definition of $L_1$ and   \eqref{xin3}, we know that
	\begin{equation}\label{xin1}
		-\Delta \zeta_{1}-(2^{*}-1)K_1\big(\frac{\lvert y \rvert}{\mu}\big)U^{2^{*}-2}_{r,\lambda} \zeta_{1}\\
		=l_{1}+R_{1}(\varphi_{0})+D_1(\varphi_0,\psi_0)
		+\sum_{j=1}^{2}c_{j} Y_{j}.
	\end{equation}

	Subtracting the first equation of system \eqref{4.5} and \eqref{xin1}, 
	we obtain
	\begin{equation*}
		-\Delta (\zeta_{1} -\varphi_{0})-(2^{*}-1)K_1\big(\frac{\lvert y \rvert}{\mu}\big)U^{2^{*}-2}_{r,\lambda}(\zeta_{1} -\varphi_{0})=0,~\text{in}~\mathbb R^N \setminus P_{1},
	\end{equation*}
	and $\zeta_{1} -\varphi_{0}\in D_0^{1,2}(\R^N\setminus P_1)$.
	Therefore,
	\[
	\|\nabla (\zeta_{1} -\varphi_{0})\|_{L^2}^2\leq C \lVert U_{r,\lambda}^{2^*-2} \rVert_{L^{N/2}(\mathbb R^{N}\setminus P_{1})} \| \zeta_{1} -\varphi_{0}\|_{L^{2^*}}^2
	=o(1)\|\nabla (\zeta_{1} -\varphi_{0})\|_{L^2}^2.
	\]
	Hence, $\zeta_{1} =\varphi_{0}$. Similarly, $\zeta_2=\psi_0$.
	
	(ii) If $(\varphi_0, \psi_0)\in \Lambda_k$ solves 
	\eqref{eq23}, then    it follows that  $(\varphi_0, \psi_0)\in S(\Lambda_k)$ and the conditions \eqref{4.} and \eqref{4..} hold.
	Therefore, \begin{gather*} 
		\varphi_0=L_{1}\Big(\varphi_{0}\Delta\chi_{1}+2\nabla\chi_{1} \nabla \varphi_{0}+(1-\chi_{1})\big(l_{1}+R_{1}(\varphi_{0})
		\big)\Big)+\chi_{1}\varphi_{0},
		\\
		\psi_0=L_{2}\Big(\psi_{0}\Delta\chi_{2}+2\nabla\chi_{2} \nabla \psi_{0}+(1-\chi_{2})\big(l_{2}+R_{2}(\psi_{0})
		\big)\Big)+\chi_{2}\psi_{0}
		. 
	\end{gather*}
	Then \[(\varphi_0,\psi_0)=S(\varphi_0, \psi_0)= S\left(\begin{matrix}
		L_{1}\big(\varphi_{0}\Delta\chi_{1}+2\nabla\chi_{1} \nabla \varphi_{0}+(1-\chi_{1})\big(l_{1}+R_{1}(\varphi_{0})
		\big)\big)\\
		L_{2}\big(\psi_{0}\Delta\chi_{2}+2\nabla\chi_{2} \nabla \psi_{0}+(1-\chi_{2})\big(l_{2}+R_{2}(\psi_{0})
		\big)\big)\end{matrix}\right)^{\top}.\]
\end{proof}

\smallskip

\begin{Prop}
	For every $(r,\rho,\lambda,\nu) \in 
	\big[r_{0} \mu - 1,\, r_{0} \mu + 1\big] 
	\times 
	\big[\rho_{0} \mu - 1,\, \rho_{0} \mu + 1\big] 
	\times 
	[\gamma_{1},\gamma_{2}] 
	\times 
	[\gamma_{3}, \gamma_{4}]$,
	there exists a unique pair 
	$(\varphi_{0}(r,\rho,\lambda,\nu),\, \psi_{0}(r,\rho,\lambda,\nu)) \in \Lambda_k$ 
	that solves \eqref{eq23}. Moreover, the mapping 
	\[
	(r,\rho,\lambda,\nu) \mapsto (\varphi_{0}(r,\rho,\lambda,\nu),\, \psi_{0}(r,\rho,\lambda,\nu))
	\]
	is continuous.
\end{Prop}

\begin{proof}
	The existence and uniqueness has been verified in Lemmas \ref{lem13} and \ref{lem12}.
	We 
	consider 
	\[(u_{1},u_{2})=(\varphi_{0}+U_{r,\lambda},\psi_{0}+V_{\rho,\nu}), \quad 
	(\varphi_{0},\psi_{0})=(\varphi_{0}(r,\rho,\lambda,\nu ),\psi_{0}(r,\rho,\lambda,\nu ))\]
	and 
	\[ (\widehat u_{1},\widehat u_{2})=(\widehat\varphi_{0}+U_{\widehat r,\widehat\lambda},\widehat\psi_{0}+V_{\widehat\rho,\widehat\nu}), \quad (\widehat\varphi_{0},\widehat\psi_{0})=(\varphi_{0}(\widehat r,\widehat \rho,\widehat\lambda, \widehat \nu),\psi_{0}(\widehat r,\widehat\rho,\widehat\lambda, \widehat \nu)) \]
	corresponding to parameters $(r,\rho,\lambda,\nu )$, and  $(\widehat r,\widehat\rho,\widehat\lambda,\widehat\nu )$, respectively.
	
	From systems \eqref{4.} and \eqref{4..}, using Proposition \ref{Prop2.2}, we have
	\begin{equation*}
		\begin{aligned}
			(1-\chi_{1})\varphi_{0}=&L_{1}\Big(\varphi_{0}\Delta\chi_{1}+2\nabla\chi_{1} \nabla \varphi_{0}+(1-\chi_{1})\big(l_{1}+R_{1}(\varphi_{0})
			\big)\Big)
		\end{aligned}
	\end{equation*}
	and
	\begin{equation*}
		\begin{aligned}
			(1-\chi_{1})\widehat\varphi_{0}=&
			\widehat L_{1}\Big(\widehat\varphi_{0}\Delta\chi_{1}+2\nabla\chi_{1} \nabla \widehat\varphi_{0}+(1-\chi_{1})\big(\widehat l_{1}+\widehat R_{1}(\widehat\varphi_{0})
			\big)\Big)
			,
		\end{aligned}
	\end{equation*}
	where
	\begin{equation*} 
		\begin{aligned}
			\widehat l_{1}=
			K_1\big(\frac{\lvert y \rvert}{\mu}\big)U^{2^{*}-1}_{\widehat r,\widehat\lambda}-\sum_{i=1}^{k}U^{2^{*}-1}_{\widehat x^{i},\widehat \lambda}
		\end{aligned}
	\end{equation*}
	and
	\begin{equation*} 
		\begin{aligned}
			\widehat R_{1}(\widehat\varphi_{0})
			=
			K_1\big(\frac{\lvert y \rvert}{\mu}\big)\lvert U_{\widehat r,\widehat\lambda}+\widehat\varphi_{0} \rvert^{2^{*}-2}(U_{\widehat r,\widehat\lambda}+\widehat\varphi_{0})-K_1\big(\frac{\lvert y \rvert}{\mu}\big)U^{2^{*}-1}_{\widehat r,\widehat\lambda}-(2^{*}-1)K_1\big(\frac{\lvert y \rvert}{\mu}\big)U^{2^{*}-2}_{\widehat r,\widehat \lambda}\widehat\varphi_{0}
			.
		\end{aligned}
	\end{equation*}
	Note that 
	\[
	\|l_{1}-\widehat l_{1}\|_{L^\infty(P_2)}+
	\| R_{1}(\widehat\varphi_{0})-\widehat R_{1}(\widehat\varphi_{0})\|_{L^\infty(P_2)}\to 0\quad  \text{as }(r,\rho,\lambda,\nu ) \to (\widehat r,\widehat\rho,\widehat\lambda,\widehat\nu).
\]

	By Lemma \ref{L2.2},   as $(r,\rho,\lambda,\nu ) \to (\widehat r,\widehat\rho,\widehat\lambda,\widehat\nu)$, we obtain
	\begin{equation*}
		\begin{aligned}
			&(1-\chi_{1})(\varphi_{0}-\widehat\varphi_{0})\\
			=&(L_{1}-\widehat L_{1})\Big(\widehat\varphi_{0}\Delta\chi_{1}+2\nabla\chi_{1} \nabla \widehat\varphi_{0}+(1-\chi_{1})\big(\widehat l_{1}+\widehat R_{1}(\widehat\varphi_{0})
			\big)\Big)
			\\
			&+L_{1}\Big((\varphi_{0}-\widehat\varphi_{0})\Delta\chi_{1}+2\nabla\chi_{1} \nabla (\varphi_{0}-\widehat\varphi_{0})+(1-\chi_{1})\big(R_{1}(\varphi_{0})- \widehat R_{1}(\widehat\varphi_{0})
			\big)\Big)\\
			=&L_{1}\Big((\varphi_{0}-\widehat\varphi_{0})\Delta\chi_{1}+2\nabla\chi_{1} \nabla (\varphi_{0}-\widehat\varphi_{0})+(1-\chi_{1})\big(R_{1}(\varphi_{0})- R_{1}(\widehat\varphi_{0})
			\big)\Big)+o(1).
		\end{aligned}
	\end{equation*}

	Therefore, by \eqref{W1},   we get 
	\[ 
	\lVert(\varphi_{0}-\widehat\varphi_{0})\rVert_{L^{\infty}( P_{1})}\leq 
	\frac14 ( \lVert(\varphi_{0}-\widehat\varphi_{0})\rVert_{L^{\infty}( P_{1})}+\lVert \psi_{0}-\widehat\psi_{0}\rVert_{L^{\infty}( Q_{1})}) +o(1).
	\]
	Similarly, we have 
	\[ 
	\lVert \psi_{0}-\widehat\psi_{0}\rVert_{L^{\infty}( Q_{1})}\leq 
	\frac14 ( \lVert(\varphi_{0}-\widehat\varphi_{0})\rVert_{L^{\infty}( P_{1})}+\lVert \psi_{0}-\widehat\psi_{0}\rVert_{L^{\infty}( Q_{1})}) +o(1).
	\]
	Therefore, as $(r,\rho,\lambda,\nu ) \to (\widehat r,\widehat\rho,\widehat\lambda,\widehat\nu )$,
	\[ \lVert(\varphi_{0}-\widehat\varphi_{0})\rVert_{L^{\infty}( P_{1})}+\lVert \psi_{0}-\widehat\psi_{0}\rVert_{L^{\infty}( Q_{1})}\to 0.
	\]
	Using Lemma \ref{lem9} again we obtain
	\begin{equation*}
		\begin{aligned}
			&\|(\varphi_0,\psi_0)-(\widehat\varphi_0,\widehat\psi_0)\|_\infty\\
			\leq& C_k\Big(\lVert \varphi_{0}-\widehat\varphi_{0}\rVert_{L^{\infty}( P_{1})}+\lVert \psi_{0}-\widehat\psi_{0}\rVert_{L^{\infty}( Q_{1})}+\lVert (U_{r,\lambda}-U_{\widehat r,\widehat\lambda})\rVert_{L^{\infty}( P_{1})}+\lVert V_{\rho,\nu}-V_{\widehat\rho,\widehat \nu}\rVert_{L^{\infty}( Q_{1})}\Big)
			\\
			\to&0,\quad  \mbox{as $(r,\rho,\lambda,\nu ) \to (\widehat r,\widehat\rho,\widehat\lambda,\widehat\nu )$}.
		\end{aligned}
	\end{equation*} 	
\end{proof}

\section{Proof of Main Theorems}\label{sec:5}

Section \ref{sec:4} implies that 
\begin{equation}\label{5.1}
	\begin{cases}
		-\Delta u_{1,k} - K_1\big(\frac{\lvert y \rvert}{\mu}\big) \lvert u_{1,k}\rvert^{2^{*}-2}u_{1,k}-\beta \lvert u_{2,k}\rvert^{\frac{2^{*}}{2}}\lvert u_{1,k}\rvert^{\frac{2^{*}}{2}-2}u_{1,k}=\sum_{j=1}^{2}c_{j} Y_{j},\\[3mm]
		-\Delta u_{2,k} - K_2\big(\frac{\lvert y \rvert}{\mu}\big) \lvert u_{2,k}\rvert^{2^{*}-2}u_{2,k}-\beta \lvert u_{1,k}\rvert^{\frac{2^{*}}{2}}\lvert u_{2,k}\rvert^{\frac{2^{*}}{2}-2}u_{2,k}
		=\sum_{j=1}^{2}d_{j} Z_{j},
	\end{cases}
\end{equation}
for some constants $c_{j}$ and $d_{j}$.
It remains to find suitable $(r,\rho,\lambda,\nu )$ such that $c_{j},d_{j}= 0,~j=1,2$.

Therefore, if the left-hand side of \eqref{5.1} belongs to $\mathbb E$, then the function on the right-hand side of \eqref{5.1} must be zero.
Recall that
\begin{equation*}
	\big(u_{1,k}(y),u_{2,k}(y)\big)=\big( U_{r_{k},\lambda_{k}}(y)+\varphi_{k}, V_{\rho_{k},\nu_{k}}(y)+\psi_{k} \big),
\end{equation*}
where 
\begin{equation*} 
	U_{r_{k},\lambda_{k}}(y)=\sum_{i=1}^{k}U_{x^{i},\lambda_{k}}(y),~ V_{\rho_{k},\nu_{k}}(y)=\sum_{i=1}^{k}V_{y^{i},\nu_{k}}(y).
\end{equation*}
For simplicity, we drop the subscript $k$. Denote
\begin{equation*}
	\big(u_{1},u_{2}\big):=\big(u_{1,k}(y),u_{2,k}(y)\big).
\end{equation*}

Let 
$\zeta$ be a smooth cut-off function such that
$\zeta(y)=1$ if $|y-x^1|\leq \mu^\frac{m}{N-3/2}$, $\zeta(y)=0$ if $|y-x^1|\geq 2\mu^\frac{m}{N-3/2}$,
$0\leq \zeta(y) \leq 1$, $|\nabla \zeta(y)|\leq 2/\mu^\frac{m}{N-3/2}$ for $y\in\R^N$.
For $\bar\alpha >1$  and $ |y-x^1|\leq 2\mu^\frac{m}{N-3/2}$, there holds
\begin{equation*}
	\sum_{j=2}^k \frac{1}{(1+|y-x^j|)^{\bar\alpha}} =O(\mu^{\frac{m\bar\alpha}{N-3/2}-\frac{m\bar\alpha}{N-2}})\frac{1}{(1+|y-x^1|)^{\bar\alpha}}.
\end{equation*}

Let 
$\chi$ be a smooth cut-off function such that
$\chi(y)=1$ if $|y-y^1|\leq \mu^\frac{m}{N-3/2}$, $\chi(y)=0$ if $|y-y^1|\geq 2\mu^\frac{m}{N-3/2}$,
$0\leq \chi(y) \leq 1$, $|\nabla \chi(y)|\leq 2/\mu^\frac{m}{N-3/2}$ for $y\in\R^N$.

Define
\begin{align*}
	\widetilde c_{1}=	\int_{\mathbb R^{N}}\Big(-\Delta u_{1} - K_1\big(\frac{\lvert y \rvert}{\mu}\big) \lvert u_{1}\rvert^{2^{*}-2}u_{1}-\beta \lvert u_{2}\rvert^{\frac{2^{*}}{2}}\lvert u_{1}\rvert^{\frac{2^{*}}{2}-2}u_{1}\Big) \zeta\frac{\partial U_{x^{1},\lambda}}{\partial y_{1}},
	\\
	\widetilde c_{2} =	\int_{\mathbb R^{N}}\Big(-\Delta u_{1} - K_1\big(\frac{\lvert y \rvert}{\mu}\big) \lvert u_{1}\rvert^{2^{*}-2}u_{1}-\beta \lvert u_{2}\rvert^{\frac{2^{*}}{2}}\lvert u_{1}\rvert^{\frac{2^{*}}{2}-2}u_{1}\Big) \zeta\frac{\partial U_{x^{1},\lambda}}{\partial \lambda},
	\\
	\widetilde{d}_1=	\int_{\mathbb R^{N}}\Big(-\Delta u_{2} - K_2\big(\frac{\lvert y \rvert}{\mu}\big) \lvert u_{2}\rvert^{2^{*}-2}u_{2}-\beta \lvert u_{1}\rvert^{\frac{2^{*}}{2}}\lvert u_{2}\rvert^{\frac{2^{*}}{2}-2}u_{2}\Big) \chi\frac{\partial V_{y^{1},\nu}}{\partial y_{1}},
	\\
	\widetilde{d}_2=	\int_{\mathbb R^{N}}\Big(-\Delta u_{2} - K_2\big(\frac{\lvert y \rvert}{\mu}\big) \lvert u_{2}\rvert^{2^{*}-2}u_{2}-\beta \lvert u_{1}\rvert^{\frac{2^{*}}{2}}\lvert u_{2}\rvert^{\frac{2^{*}}{2}-2}u_{2}\Big) \chi\frac{\partial V_{y^{1},\nu}}{\partial \nu}.
\end{align*}

Since 
\[
\int_{\mathbb R^{N}} Y_{1}  \zeta\frac{\partial U_{x^{1},\lambda}}{\partial y_{1}} = -\int\chi_0\left|\frac{\partial U_{0,\lambda}}{\partial y_{1}} \right|^2, \quad 
\int_{\mathbb R^{N}}  Y_{2}\zeta\frac{\partial U_{x^{1},\lambda}}{\partial \lambda}=
\int \chi_0\left|\frac{\partial U_{0,\lambda}}{\partial \lambda} \right|^2,
\]
we know $c_i=0$ if and only if $\widetilde{c}_i=0$, $i=1,2$.
Similarly,
$d_i=0$ if and only if $\widetilde{d}_i=0$, $i=1,2$.

To find $(r,\rho,\lambda,\nu)$ such that $c_i = 0$ and $d_i = 0$ for $i=1,2$, we first estimate the quantities.

\begin{Lem}\label{lem5.3} There exists $\theta>0$ such that
	\begin{align*}
		\widetilde{c}_2&
		=-\frac{B_{1}}{\lambda^{m+1}\mu^{m}}
		+\sum_{i=2}^{k}\frac{B_{2}}{\lambda^{N-1}\lvert x^{i}-x^{1}\rvert^{N-2}}
		+O\Big(\frac{1}{\mu^{m+\theta}}+\frac{1}{\mu^{m}}\lvert\mu r_{0}-\lvert x^{1}\rvert\rvert^{2}\Big),
		\\
		\widetilde{d}_2&
		=-\frac{B_{3}}{\nu^{m+1}\mu^{m}}
		+\sum_{i=2}^{k}\frac{B_{4}}{\nu^{N-1}\lvert y^{i}-y^{1}\rvert^{N-2}}
		+O\Big(\frac{1}{\mu^{m+\theta}}+\frac{1}{\mu^{m}}\lvert\mu \rho_{0}-\lvert y^{1}\rvert\rvert^{2}\Big),
	\end{align*}
	where $B_{i}>0,i = 1, 2, 3, 4$ are some constants.
\end{Lem}
\begin{proof} 
	From Lemmas \ref{lem7} and \ref{lemB1}, there holds
	\begin{equation}\label{3lambda}
		\begin{aligned}
			&\Big\lvert\int_{\mathbb R^{N}}  \lvert u_{2}\rvert^{\frac{2^{*}}{2}}\lvert u_{1}\rvert^{\frac{2^{*}}{2}-2}u_{1} \zeta\frac{\partial U_{x^{1},\lambda}}{\partial \lambda}\Big\rvert\leq
			\int_{B_{2\mu^\frac{m}{N-3/2}}(x^1)}  \lvert u_{2}\rvert^{\frac{2^{*}}{2}}\lvert u_{1}\rvert^{\frac{2^{*}}{2}-1} \Big\lvert\frac{\partial U_{x^{1},\lambda}}{\partial \lambda}\Big\rvert
			\\
			\leq& \int_{|y-x^1|\leq 2\mu^\frac{m}{N-3/2}} \sum_{j=1}^k \frac{1}{(1+\lvert y-y^{j}\rvert)^{2^*(N-2)/2}} \frac{1}{(1+\lvert y-x^{1}\rvert)^{2^*(N-2)/2}}\\
			\leq&\sum_{j=1}^k \frac{1}{|y^j-x^1|^{N-1}}
			\int_{\R^N} \frac{1}{(1+\lvert y-y^{j}\rvert)^{N+1}} +\frac{1}{(1+\lvert y-x^{1}\rvert)^{N+1}}
			=O\Big(\frac{1}{\mu^{N-2}}\Big).
		\end{aligned}
	\end{equation}
	On the other hand, we have
	\begin{equation*}
		\begin{aligned}
			&\int_{\mathbb R^{N}}\Big(-\Delta u_{1}-K_1\big(\frac{\lvert y \rvert}{\mu}\big)\lvert u_{1}\rvert^{2^{*}-2}u_{1}\Big)\zeta \frac{\partial U_{x^{1},\lambda}}{\partial \lambda}\\
			=&\int_{\mathbb R^{N}}\Big(-\Delta U_{r,\lambda}-K_1\big(\frac{\lvert y \rvert}{\mu}\big) U_{r,\lambda}^{2^*-1}\Big)  \frac{\partial U_{x^{1},\lambda}}{\partial \lambda} +\int_{\mathbb R^{N}}\Big(\Delta U_{r,\lambda}+K_1\big(\frac{\lvert y \rvert}{\mu}\big) U_{r,\lambda}^{2^*-1}\Big)(1-\zeta) \frac{\partial U_{x^{1},\lambda}}{\partial \lambda}\\
			& -\int_{\mathbb R^{N}}  K_1\big(\frac{\lvert y \rvert}{\mu}\big)\left(\lvert u_{1}\rvert^{2^{*}-2}u_{1}-U_{r,\lambda}^{2^*-1}\right) \zeta \frac{\partial U_{x^{1},\lambda}}{\partial \lambda}-\int_{\mathbb R^N}\Delta\left(\zeta\frac{\partial U_{x^{1},\lambda}}{\partial \lambda}\right) \varphi \\
			=& \frac{1}{k}\frac{\partial}{\partial \lambda }  F(U_{r,\lambda}) 
			+ M_1
			+ M_2 +M_3+M_4,
		\end{aligned}
	\end{equation*}
	where 
	\[\begin{aligned}
		M_1 &= \int_{\mathbb{R}^{N}}\left(\Delta U_{r,\lambda} + K_1\left(\frac{\lvert y \rvert}{\mu}\right) U_{r,\lambda}^{2^*-1} \right)(1 - \zeta) \frac{\partial U_{x^{1},\lambda}}{\partial \lambda}, \\
		M_2 &= -\int_{\mathbb{R}^{N}} K_1\left(\frac{\lvert y \rvert}{\mu}\right)\left(\lvert u_{1}\rvert^{2^{*}-2}u_{1} - U_{r,\lambda}^{2^*-1} - (2^*-1) U_{x^1,\lambda}^{2^*-2}\varphi \right) \zeta \frac{\partial U_{x^{1},\lambda}}{\partial \lambda}, \\
		M_3 &= -\int_{\mathbb{R}^N} (\Delta\zeta) \frac{\partial U_{x^{1},\lambda}}{\partial \lambda} \varphi - 2\int_{\mathbb{R}^N} \nabla\zeta  \nabla\frac{\partial U_{x^{1},\lambda}}{\partial \lambda} \varphi, \\
		M_4 &= -(2^*-1)\int_{\mathbb{R}^{N}} \left(K_1\left(\frac{\lvert y \rvert}{\mu}\right) - 1\right) U_{x^1,\lambda}^{2^*-2} \varphi \zeta \frac{\partial U_{x^{1},\lambda}}{\partial \lambda},
	\end{aligned}\]
	and the functional $F: D^{1,2}(\mathbb R^N)\to \mathbb R$ is defined for any $ u\in D^{1,2}(\mathbb R^N)$
	\[
	F(u)=\frac{1}{2}\int_{\mathbb R^N}|\nabla u|^2-\frac{1}{2^*}\int_{\mathbb R^N} K_1\big(\frac{y}{\mu}\big) |u|^{2^*}.
	\]
	
	By Lemmas \ref{lemB1} and \ref{lem2}, we have 
	\begin{equation*}
		\begin{aligned}
			|M_1|  
			\leq & \int_{|y-x^1|\geq \mu^{\frac{m}{N-1}}} \left|\left( K_1\big(\frac{\lvert y \rvert}{\mu}\big) U_{r,\lambda}^{2^*-1}-\sum_{j=1}^k U_{x^j,\lambda}^{2^*-1}\right)   \frac{\partial U_{x^{1},\lambda}}{\partial \lambda}\right|  \\
			\leq &  
			\frac{C}{\mu^{\frac{m}{2}+\theta}}\int_{|y-x^1|\geq \mu^{\frac{m}{N-1}}} \sum_{j=1}^k\frac{1}{(1+|y-x^j|)^{\frac{N}{2}+2}}    \frac{1}{(1+|y-x^1|)^{N-2}} \\
			\leq &\frac{C}{\mu^{\frac{m}{2}+\theta}}\sum_{j=2}^k\frac{1}{|x^1-x^j|^\frac{N-2}{2}}\int_{\R^N}\left( \frac{1}{(1+|y-x^j|)^{N+1}} +   \frac{1}{(1+|y-x^1|)^{N+1}} \right)
			\\
			& + \frac{C}{\mu^{\frac{m}{2}+\theta}}\int_{|y-x^1|\geq \mu^{\frac{m}{N-1}}}\frac{1}{(1+|y-x^1|)^{\frac{3N}{2}}}
			=O  \left(\frac{1}{\mu^{m+\theta}}\right).
		\end{aligned}
	\end{equation*}

	To estimate $M_2$,
	we first note that for $y\in B_{2\mu^{\frac{m}{N-3/2}}}(x^1)$, 
	\[\frac{\varphi}{U_{x^1,\lambda}}=O(\frac{1}{\mu^\theta}), \quad
	\frac{\sum_{j=2}^k U_{x^j,\lambda}}{U_{x^1,\lambda}}=O(\frac1{\mu^{\frac{m}{2N-3}}}).
	\]
	Then for $|y-x^1|\leq 2\mu^{\frac{m}{N-3/2}}$,
	\[ \begin{aligned}
		&\lvert u_{1}\rvert^{2^{*}-2}u_{1}-U_{r,\lambda}^{2^*-1}  -(2^*-1) U_{x^1,\lambda}^{2^*-2}\varphi\\
		=& (2^*-1)\left((U_{r,\lambda}+t_y \varphi)^{2^*-2}- U_{x^1,\lambda}^{2^*-2}\right) \varphi 
		=O(U_{x^1,\lambda}^{2^*-3}\sum_{j=2}^k U_{x^j,\lambda}|\varphi| + U_{x^1,\lambda}^{2^*-3}\varphi^2),
	\end{aligned}
	\]
	where $t_y\in(0,1)$.
	
	From Corollary  \ref{cor3.3}, we have
	\begin{equation}\label{Q1100}
	\begin{aligned}
		&\int_{\R^N}U_{x^1,\lambda}^{2^*-3}\sum_{j=2}^k U_{x^j,\lambda}\zeta \left|\varphi \frac{\partial U_{x^1,\lambda}}{\partial\lambda}\right|= \frac{1}{2^*-2}\int_{\R^N}\sum_{j=2}^k U_{x^j,\lambda}\zeta \left|\varphi \frac{\partial U_{x^1,\lambda}^{2^*-2}}{\partial\lambda}\right|\\
		\leq&C\frac{1}{\mu^{\frac{m}2+\theta}}
		\int_{B_{2\mu^{\frac{m}{N-3/2}}}(x^1)}    \frac{1}{(1+\lvert y-x^{1}\rvert)^{\frac{N}{2}+4}} \sum_{j=2}^k\frac{1}{(1+|y-x^j|)^{N-2}}\\
		\leq&C\frac{1}{\mu^{\frac{m}2+\theta}}
		\sum_{j=2}^k \frac{1}{ |x^1-x^j|^{\frac{N}{2}}}\int_{B_{2\mu^{\frac{m}{N-3/2}}}(x^1)} \frac{1}{(1+\lvert y-x^{1}\rvert)^{N+2}} \\
		=&O\Big(\frac{1}{\mu^{m+\theta}}\Big)
	\end{aligned}
\end{equation}
	and
	\begin{equation}\label{Q2}
		\begin{aligned}	
			&\int_{\R^N} \zeta\Big\lvert U^{2^*-3}_{x^{1},\lambda}\frac{\partial U_{x^{1},\lambda}}{\partial \lambda}\varphi^{2}\Big\rvert=\frac{1}{2^*-2}\int_{\R^N} \zeta\Big\lvert \frac{\partial U^{2^*-2}_{x^{1},\lambda}}{\partial \lambda}\varphi^{2}\Big\rvert\\
			\leq&C\frac{1}{\mu^{m+2\theta}}
			\int_{B_{2\mu^{\frac{m}{N-3/2}}}(x^1)}    \frac{1}{(1+\lvert y-x^{1}\rvert)^{N+4}}
			=O\Big(\frac{1}{\mu^{m+\theta}}\Big).
		\end{aligned}
	\end{equation}
	From \eqref{Q1100} and \eqref{Q2}, we get
	\[
	M_2=O\Big(\frac{1}{\mu^{m+\theta}}\Big).
	\]
	
	We now estimate $M_3$. 
	By Corollary \ref{cor3.3}, there holds
	\begin{align*}
		&\int_{\mathbb R^N}|\left(\Delta\zeta\right)\frac{\partial U_{x^{1},\lambda}}{\partial \lambda} \varphi |
		\leq  
		\frac{C}{\mu^\frac{4m}{2N-3} } \int_{B_{2\mu^\frac{2m}{2N-3}}(x^1)\setminus B_{ \mu^\frac{2m}{2N-3}}(x^1) } |\frac{\partial U_{x^{1},\lambda}}{\partial \lambda} \varphi|\\
		\leq &\frac{C}{\mu^{\frac{4m}{2N-3}+\frac{m}{2}+\theta} } 
		\int_{B_{2\mu^\frac{2m}{2N-3}}(x^1)\setminus B_{ \mu^\frac{2m}{2N-3}}(x^1) } \frac{1}{(1+|y-x^1|)^{\frac{3N}{2}-2}}
		\leq \frac{C}{\mu^{\frac{Nm}{2N-3}+\frac{m}{2}+\theta}}
		=O\Big(\frac{1}{\mu^{m+\theta}}\Big),
	\end{align*}
	and \[
	\begin{aligned}
		&\int_{\mathbb R^N}|\nabla\zeta\nabla\frac{\partial U_{x^{1},\lambda}}{\partial \lambda} \varphi|
		\leq  
		\frac{C}{\mu^\frac{2m}{2N-3} } \int_{B_{2\mu^\frac{2m}{2N-3}}(x^1)\setminus B_{ \mu^\frac{2m}{2N-3}}(x^1) } |\nabla\frac{\partial U_{x^{1},\lambda}}{\partial \lambda} \varphi|\\
		\leq &\frac{C}{\mu^{\frac{2m}{2N-3}+\frac{m}{2}+\theta} } 
		\int_{B_{2\mu^\frac{2m}{2N-3}}(x^1)\setminus B_{ \mu^\frac{2m}{2N-3}}(x^1) } \frac{1}{(1+|y-x^1|)^{\frac{3N}{2}-1}} 
		=O\Big(\frac{1}{\mu^{m+\theta}}\Big).
	\end{aligned}
	\]
	Then
	\begin{align*}
		|M_3| =O\Big(\frac{1}{\mu^{m+\theta}}\Big).
	\end{align*}	 
	
	At last 
	\begin{align*}
		|M_4|  \leq& \frac{C}{\mu^{m+\frac{m}{2}+\theta}}   
		\int_{\mathbb{R}^{N}}  \zeta (|y-x^1|^{m}+1) \frac{1}{(1+|y-x^1|)^{\frac{3N}2+2}}  \\ 
		\leq& \frac{C}{\mu^{m+\frac{m}{2}+\theta -\frac{m(2m-N-4)^+}{2N-3}}}   +
		\frac{Ck}{\mu^{m+1}}=O(\frac{1}{\mu^{m+\theta}}).
	\end{align*}
	
	Therefore, by the estimates for $M_i$, $i=1,2,3,4$, and \cite[Proposition A.2]{WY10},  we obtain the first estimate.
	We can also prove the other   result by similar arguments.	
\end{proof}

\begin{Lem}\label{lem5.2} There exists $\theta>0$ such that
	\[ 
		\begin{aligned}
			\widetilde{c}_1
			=&-A_{1}\frac{\lvert x^{1}\rvert-\mu r_{0}}{\lambda^{m-2}\mu^{m}}
			+O\Big(\frac{1}{\mu^{m+\theta}}+\frac{(m-2)^+}{\mu^{m}}\left(\lvert\lvert x^{1}\rvert-\mu r_{0}\rvert\rvert^{m-1}+ \lvert\lvert x^{1}\rvert-\mu r_{0}\rvert\rvert^{2}\right)\Big),
			\\
			\widetilde{d}_1
			=&-A_{2}\frac{\lvert y^{1}\rvert-\mu \rho_{0}}{\nu^{m-1}\mu^{m}}
			+O\Big(\frac{1}{\mu^{m+\theta}}+\frac{(m-2)^+}{\mu^{m}}\left(\lvert\lvert y^{1}\rvert-\mu r_{0}\rvert\rvert^{m-1}+\lvert\lvert y^{1}\rvert-\mu \rho_{0}\rvert\rvert^{2}\right)\Big),
		\end{aligned}
	\]
	where $ A_{i}>0,i = 1,2$ are some constants.
\end{Lem}
\begin{proof}
	Similar to \eqref{3lambda}, we obtain
	\begin{equation}\label{3lambda.}
		\begin{aligned}
			\Big\lvert\int_{\mathbb R^{N}}  \lvert u_{2}\rvert^{\frac{2^{*}}{2}}\lvert u_{1}\rvert^{\frac{2^{*}}{2}-2}u_{1} \zeta\frac{\partial U_{x^{1},\lambda}}{\partial y_{1}}\Big\rvert
			=O\Big(\frac{1}{\mu^{N-2}}\Big).
		\end{aligned}
	\end{equation}
	It follows that
	\begin{equation}\label{1lambda.}
		\begin{aligned}
			&\int_{\mathbb R^{N}}\Big(-\Delta u_{1}-K_1\big(\frac{\lvert y \rvert}{\mu}\big)\lvert u_{1}\rvert^{2^{*}-2}u_{1}\Big)\zeta \frac{\partial U_{x^{1},\lambda}}{\partial y_{1}}\\
			=&-\int_{\mathbb R^{N}}K_1\big(\frac{\lvert y \rvert}{\mu}\big)\lvert u_{1}\rvert^{2^{*}-2}u_{1}\zeta \frac{\partial U_{x^{1},\lambda}}{\partial y_{1}}+(2^*-1)\int_{\mathbb R^{N}}\varphi U_{x^{1},\lambda}^{2^*-2}\zeta \frac{\partial U_{x^{1},\lambda}}{\partial y_{1}}+\sum_{i=1}^{k}U_{x^{i},\lambda}^{2^*-1}\zeta \frac{\partial U_{x^{1},\lambda}}{\partial y_{1}}\\
			&-\int_{\mathbb R^{N}} \varphi\left(\frac{\partial U_{x^{1},\lambda}}{\partial y_{1}}\Delta\zeta+2\nabla\zeta\nabla\frac{\partial U_{x^{1},\lambda}}{\partial y_{1}}\right)\\
			=&\int_{\mathbb R^{N}}\left(1-K_1\big(\frac{\lvert y \rvert}{\mu}\big) \right)U_{r,\lambda}^{2^*-1}\zeta \frac{\partial U_{x^{1},\lambda}}{\partial y_{1}}+N_1+ N_2+N_3+N_4,
		\end{aligned}
	\end{equation}
	where 
	\begin{equation*}
		\begin{gathered}
			N_1=-\int_{\mathbb R^{N}}\left(U_{r,\lambda}^{2^*-1}-\sum_{i=1}^{k}U_{x^{i},\lambda}^{2^*-1}\right)\zeta \frac{\partial U_{x^{1},\lambda}}{\partial y_{1}},
			\\
			N_2=-\int_{\mathbb R^{N}}  K_1\big(\frac{\lvert y \rvert}{\mu}\big)\left(\lvert u_{1}\rvert^{2^{*}-2}u_{1}-U_{r,\lambda}^{2^*-1} -(2^*-1) U_{x^1,\lambda}^{2^*-2}\varphi\right) \zeta \frac{\partial U_{x^{1},\lambda}}{\partial y_{1}},
			\\ 
			N_3=-\int_{\mathbb R^{N}} \varphi\left(\frac{\partial U_{x^{1},\lambda}}{\partial y_{1}}\Delta\zeta+2\nabla\zeta\nabla\frac{\partial U_{x^{1},\lambda}}{\partial y_{1}}\right),\\
			N_4=-(2^*-1)\int_{\mathbb{R}^{N}} \left(K_1\left(\frac{\lvert y \rvert}{\mu}\right) - 1\right) U_{x^1,\lambda}^{2^*-2} \varphi \zeta \frac{\partial U_{x^{1},\lambda}}{\partial y_1}.
		\end{gathered}
	\end{equation*}
	Similar to $M_{i}$, $i=2,3,4$ in the proof of Lemma \ref{lem5.3}, we also get $\lvert N_{i}\rvert=O\Big(\frac{1}{\mu^{m+\theta}}\Big)$, $i=2,3,4$.
	
	Next, we estimate $N_{1}$.
	For $y\in B_{2\mu^{\frac{m}{N-3/2}}}(x^1)$, there holds 
	\[ \quad
	\frac{\sum_{j=2}^k U_{x^j,\lambda}}{U_{x^1,\lambda}}=O\Big(\frac1{\mu^{\frac{m}{2N-3}}}\Big).
	\]
	Then
	\[U_{r,\lambda}^{2^*-1}- U_{x^{1},\lambda}^{2^*-1}=(2^*-1)U_{x^{1},\lambda}^{2^*-2}\sum_{j=2}^k U_{x^{j},\lambda}+O\left(U_{x^{1},\lambda}^{2^*-3}\left(\sum_{j=2}^k U_{x^{j},\lambda}\right)^2\right).\]
	By \eqref{tau1a}, we have 
	\begin{align*}
		\int_{\R^N}U_{x^{1},\lambda}^{2^*-3}\left(\sum_{j=2}^k U_{x^{j},\lambda}\right)^2\zeta\left|\frac{\partial U_{x^{1},\lambda}}{\partial y_1}\right|=  & \frac{1}{2^*-2}\int_{\R^N} \left(\sum_{j=2}^k U_{x^{j},\lambda}\right)^2\zeta\left|\frac{\partial U_{x^{1},\lambda}^{2^*-2}}{\partial y_1}\right|\\
		\leq &C\int_{\Omega_1}  \sum_{j=2}^k \frac{1}{(1+|y-x^j|)^{2N-4-\tau_1}} \frac{1}{(1+|y-x^1|)^{5}}\\
		\leq &C\sum_{j=2}^k \frac{1}{|x^1-x^j|^{N}}\int_{\Omega_1}    \frac{1}{(1+|y-x^1|)^{N+1-\tau_1}}
		=O(\frac{1}{\mu^{m+\theta}}).
	\end{align*}
	We also have 
	\[
	\int_{\R^N}\sum_{j=2}^k U_{x^j,\lambda}^{2^*-1}\zeta\left|\frac{\partial U_{x^{1},\lambda}}{\partial y_1}\right|
	\leq C \sum_{j=2}^k \int_{\Omega_1}  \frac{1}{(1+|y-x^j|)^{N+2}} \frac{1}{(1+|y-x^1|)^{N-1}}=O(\frac{1}{\mu^{m+\theta}}).
	\]
	Thus,
	\begin{equation*}
		\begin{aligned}
			N_{1}
			= -\sum_{j=2}^k \int_{\mathbb R^{N}}  U_{x^{j},\lambda}\zeta \frac{\partial U_{x^{1},\lambda}^{2^*-1}}{\partial y_{1}}+O(\frac{1}{\mu^{m+\theta}}).
		\end{aligned}
	\end{equation*}
	Integration by parts, there holds for $j\neq 1$,
	\begin{align*}
		\left|\int_{\mathbb R^{N}}U_{x^{j},\lambda}\zeta \frac{\partial U_{x^{1},\lambda}^{2^*-1}}{\partial y_{1}}\right|\leq
		\int_{\mathbb R^{N}}|\nabla U_{x^{j},\lambda}|\zeta  U_{x^{1},\lambda}^{2^*-1} 
		+\int_{\mathbb R^{N}}U_{x^{j},\lambda}|\nabla\zeta|  U_{x^{1},\lambda}^{2^*-1} .
	\end{align*}
	On one hand 
	we have 
	\[
	\int_{\mathbb R^{N}}|\nabla U_{x^{j},\lambda}|\zeta  U_{x^{1},\lambda}^{2^*-1} \leq C
	\int_{\Omega_1}\frac{1}{(1+|y-x^j|)^{N-1}} \frac{1}{(1+|y-x^1|)^{N+2}}\leq \frac{C}{|x^1-x^j|^{N-1}}.
	\]
	On the other hand we have 
	\[
	\int_{\mathbb R^{N}}U_{x^{j},\lambda}|\nabla\zeta|  U_{x^{1},\lambda}^{2^*-1} \leq 
	\frac{C}{\mu^{\frac{2m}{2N-3}}|x^1-x^j|^{N-2}}\int_{\R^N\setminus B_{\mu^{\frac{2m}{2N-3}}}(x^1)}  U_{x^{1},\lambda}^{2^*-1}
	\leq \frac{C}{\mu^{\frac{6m}{2N-3}}|x^1-x^j|^{N-2}}.
	\]

	Therefore,
	\begin{equation*}
		\begin{aligned}
			\lvert N_{1}\rvert=O(\sum_{j=2}^k\frac{1}{|x^1-x^j|^{N-1}}+\sum_{j=2}^k\frac{1}{\mu^{\frac{6m}{2N-3}}|x^1-x^j|^{N-2}}) +O(\frac{1}{\mu^{m+\theta}})=O(\frac{1}{\mu^{m+\theta}})
			.
		\end{aligned}
	\end{equation*}
	
	Then, we estimate  the main term.
	In  $ \text{supp}\,\zeta$, we have 
	\[
	U_{r,\lambda}^{2^*-1}= \left(1+O( \mu^{-\frac{m}{2N-3}})\right)U_{x^1,\lambda}^{2^*-1},
	\]
	Therefore,
	\begin{align*}
		&\int_{\mathbb R^{N}}\left(1-K_1\big(\frac{\lvert y \rvert}{\mu}\big) \right)U_{r,\lambda}^{2^*-1}\zeta \frac{\partial U_{x^{1},\lambda}}{\partial y_{1}} \\
		=&\frac{1}{2^*}\frac{c_{0,1}}{\mu^{m}}\int_{\R^N}\lvert\lvert y\rvert-\mu r_{0}\rvert^{m}\zeta\frac{\partial U^{2^{*}}_{x^{1},\lambda}}{\partial y_{1}}
		+O\left(\frac{1}{\mu^{m+\theta} } \int_{\R^N} \lvert\lvert y\rvert-\mu r_{0}\rvert^{m+\theta_{1}}\zeta\left|\frac{\partial U^{2^{*}}_{x^{1},\lambda}}{\partial y_{1}}\right| \right),
	\end{align*}
	for some $\theta\in(0,\min\{\theta_1, \frac{m}{2N-3}\})$.
	We can check that 
	\begin{equation}\label{eq5.18}
		\int_{\R^N} \lvert\lvert y\rvert-\mu r_{0}\rvert^{m+\theta_{1}}\zeta\left|\frac{\partial U^{2^{*}}_{x^{1},\lambda}}{\partial y_{1}}\right|\leq C
		\int_{\R^N}\frac{ 1+ \lvert y-x^1\rvert^{m+\theta_{1}}}{(1+|y-x^1|)^{2N+1}}\zeta \leq C.
	\end{equation}
	
	Integrating by parts, we have
	\begin{align*}
		\int_{\R^N}\lvert\lvert y\rvert-\mu r_{0}\rvert^{m}\zeta\frac{\partial U^{2^{*}}_{x^{1},\lambda}}{\partial y_{1}}
		=-m\int_{\R^N}\lvert\lvert y\rvert-\mu r_{0}\rvert^{m-2}(\lvert y\rvert-\mu r_{0})\frac{y_{1}}{\lvert y\rvert} \zeta U^{2^{*}}_{x^{1},\lambda} -\int_{\R^N}\lvert\lvert y\rvert-\mu r_{0}\rvert^{m} U^{2^{*}}_{x^{1},\lambda}\frac{\partial \zeta}{\partial y_{1}}.
	\end{align*}
	Similar to \eqref{eq5.18}, there holds
	\[
	\left|\int_{\R^N}\lvert\lvert y\rvert-\mu r_{0}\rvert^{m} U^{2^{*}}_{x^{1},\lambda}\frac{\partial \zeta}{\partial y_{1}}\right|
	\leq \frac{C}{\mu^{\frac{2m}{2m-3}}} \int_{\R^N}\frac{\lvert  y-x^1\rvert^{m}}{(1+|y-x^1|)^{2N}} =O(\frac{1}{\mu^\theta}).
	\]
	We then estimate
	\begin{align*}
		\mathcal I:=&\int_{\R^N}\lvert\lvert y\rvert-\mu r_{0}\rvert^{m-2}(\lvert y\rvert-\mu r_{0})\frac{y_{1}}{\lvert y\rvert} \zeta U^{2^{*}}_{x^{1},\lambda} \\
		= &\int_{\R^N}\lvert\lvert z+x^{1}\rvert-\mu r_{0}\rvert^{m-2}(\lvert z+x^{1}\rvert-\mu r_{0})\frac{z_{1}+r}{\lvert z+x^{1}\rvert} \tilde\zeta  U^{2^{*}}_{0,\lambda},
	\end{align*}
	where $z=(z_{1},z_{*})\in \R\times \R^{N-1}$ and $\tilde\zeta(z)=\zeta(z+x^{1}) $.
	For $z\in \text{supp}\, \tilde\zeta =B_{2\mu^{ \frac{2m}{2N-3}}}(0)$, we have 
	\[
	\lvert z+x^{1}\rvert=\sqrt{r^2+2z_1r+|z|^2}= r+z_1+O(\frac{|z|^2}{r}),
	\]
	where $r=\lvert x^1\rvert$.
		Therefore,
	\begin{align*}
		\mathcal I   
		=&\int_{\R^N}\lvert\lvert z+x^{1}\rvert-\mu r_{0}\rvert^{m-2}(\lvert z+x^{1}\rvert-\mu r_{0})(1+O(\frac{|z|^2}{r^2})) \tilde\zeta U^{2^{*}}_{0,\lambda}\\
		=& \int_{\R^N}\lvert\lvert z+x^{1}\rvert-\mu r_{0}\rvert^{m-2}(\lvert z+x^{1}\rvert-\mu r_{0})  \tilde\zeta U^{2^{*}}_{0,\lambda} +O(\frac{1}{r^2}\int_{\R^N}\frac{1}{(1+|z|)^{2N-m-1}})\\
		=& \int_{\R^N}\lvert\lvert z+x^{1}\rvert-\mu r_{0}\rvert^{m-2}(\lvert z+x^{1}\rvert-\mu r_{0})  \tilde\zeta U^{2^{*}}_{0,\lambda} +O(\frac{1}{\mu^\theta} ).
	\end{align*}
	Without loss of generality, we can assume that $\tilde \zeta$ is radially symmetric.
	Then 
	\[
	\int_{\R^N}|z_1|^{m-2}z_1\tilde\zeta U^{2^{*}}_{0,\lambda}=0.
	\]
	Note also that 
	\[
	\int_{\R^N}|z_1|^{m-2} (1-\tilde\zeta) U_{0,\lambda}^{2^*}=O(\int_{|z|\geq \mu^{\frac{2m}{2N-3}}} \frac{1}{(1+|z|)^{2N-m+2}})= O(\frac{1}{\mu^\theta}).
	\]
	
	If $m=2$, then 
	\begin{equation}\label{eq5.19}
		\mathcal I=\int_{\R^N} (r-\mu r_{0})  \tilde\zeta U^{2^{*}}_{0,\lambda} +O(\frac{1}{\mu^\theta} )=
		(r-\mu r_{0}) \int_{\R^N}   U^{2^{*}}_{0,1}  +O(\frac{1}{\mu^\theta} ).
	\end{equation}
	If $m>2$, by Lemma \ref{lema6}, we obtain 
	\begin{align*}
		&\lvert\lvert z+x^{1}\rvert-\mu r_{0}\rvert^{m-2}(\lvert z+x^{1}\rvert-\mu r_{0}) = 
		|z_1|^{m-2}z_1+(m-1) |z_1|^{m-2}(r-\mu r_0+O(\frac{|z|^2}{r})) \\
		&+O(|r-\mu r_0|^{m-1})+O(\frac{|z|^{2m-2}}{r^{m-1}})+O( (m-3)^+|z_1|^{m-3}(\frac{|z|^4}{r^2}+  |r-\mu r_0|^2)) \\
		=&|z_1|^{m-2}z_1+(m-1) |z_1|^{m-2}(r-\mu r_0)\\
		&+O(\frac{|z|^{m}}{r}) +O( (1+|z|)^{m-3}(|r-\mu r_0|^{m-1}+|r-\mu r_0|^{2})).
	\end{align*}

	We can also check that 
	\[
	\int_{\R^N} \frac{|z|^{m}}{r} \tilde\zeta U^{2^{*}}_{0,\lambda}=
	O(\frac{1}{\mu^\theta}),\quad 
	\int_{\R^N}(1+|z|)^{m-3} \tilde\zeta U^{2^{*}}_{0,\lambda}=O(1).
	\]
	
	Therefore, for the case $m>2$, there holds
	\begin{equation}\label{eq5.20}
		\begin{aligned}
			\mathcal I=& (m-1)(r-\mu r_0)\int_{\R^N} |z_1|^{m-2} U_{0,\lambda}^{2^*}
			+O(\frac{1}{\mu^\theta}+|r-\mu r_0|^{m-1}+|r-\mu r_0|^{2})\\
			=&\frac{(m-1)(r-\mu r_0)}{\lambda^{m-2}}\int_{\R^N} |z_1|^{m-2} U^{2^{*}}_{0,1}
			+O(\frac{1}{\mu^\theta}+|r-\mu r_0|^{m-1}+|r-\mu r_0|^{2}).
		\end{aligned}
	\end{equation}
	Combining \eqref{3lambda.}, \eqref{1lambda.}, \eqref{eq5.19} and \eqref{eq5.20}, the conclusion follows.
\end{proof}

\begin{proof}[\textbf{Proof of Theorem \ref{th2} and Theorem \ref{th1}}]
	From \cite{WY10}, 
	there is a constant $B_{5}>0$ such that
	\begin{equation*}
		\sum_{i=2}^{k}\frac{1}{\lvert x^{i}-x^{1}\rvert^{N-2}}=\frac{ B_{5}k^{N-2}}{\lvert x^{1}\rvert^{N-2}}+O\Big(\frac{k}{\lvert x^{1}\rvert^{N-2}}\Big).
	\end{equation*}
	Similarly,
	\begin{equation*}
		\sum_{i=2}^{k}\frac{1}{\lvert y^{i}-y^{1}\rvert^{N-2}}=\frac{ B_{6}k^{N-2}}{\lvert y^{1}\rvert^{N-2}}+O\Big(\frac{k}{\lvert y^{1}\rvert^{N-2}}\Big),
	\end{equation*}
	where $B_{6}>0$ is a constant.
	Therefore, by Lemmas \ref{lem5.3} and \ref{lem5.2},    
	\begin{align*}
		\widetilde{c}_2 
		=&
		-\frac{B_{1}}{\lambda^{m+1}\mu^{m}}
		+\frac{ B_{5}k^{N-2}}{\lambda^{N-1}r^{N-2}}
		+O\Big(\frac{1}{\mu^{m+\theta}}+\frac{1}{\mu^{m}}\lvert\mu r_{0}-\lvert x^{1}\rvert\rvert^{2}+\frac{k}{r^{N-2}}\Big),
		\\
		\widetilde{d}_2 
		=&
		-\frac{B_{3}}{\nu^{m+1}\mu^{m}}
		+\frac{ B_{6}k^{N-2}}{\nu^{N-1}\rho^{N-2}}
		+O\Big(\frac{1}{\mu^{m+\theta}}+\frac{1}{\mu^{m}}\lvert\mu \rho_{0}-\lvert y^{1}\rvert\rvert^{2}+\frac{k}{\rho^{N-2}}\Big),
		\\
		\widetilde{c}_1
		=&-A_{1}\frac{\lvert x^{1}\rvert-\mu r_{0}}{\lambda^{m-2}\mu^{m}}
		+O\Big(\frac{1}{\mu^{m+\theta}}+\frac{(m-2)^+}{\mu^{m}}\left(\lvert\lvert x^{1}\rvert-\mu r_{0}\rvert\rvert^{m-1}+ \lvert\lvert x^{1}\rvert-\mu r_{0}\rvert\rvert^{2}\right)\Big),
		\\
		\widetilde{d}_1
		=&-A_{2}\frac{\lvert y^{1}\rvert-\mu \rho_{0}}{\nu^{m-1}\mu^{m}}
		+O\Big(\frac{1}{\mu^{m+\theta}}+\frac{(m-2)^+}{\mu^{m}}\left(\lvert\lvert y^{1}\rvert-\mu r_{0}\rvert\rvert^{m-1}+\lvert\lvert y^{1}\rvert-\mu \rho_{0}\rvert\rvert^{2}\right)\Big).
	\end{align*}

	Let   
	\begin{equation*} 
		\lambda_{0}=\Big(\frac{B_{5}}{B_{1}r^{N-2}_{0}}\Big)^{\frac{1}{N-2-m}}, \quad \nu_{0}=\Big(\frac{B_{6}}{B_{3}\rho^{N-2}_{0}}\Big)^{\frac{1}{N-2-m}}
	\end{equation*}
	be   the solution of
	\begin{equation*} 
		-\frac{B_{1}}{\lambda^{m+1}}
		+\frac{B_{5}}{\lambda^{N-1}r^{N-2}_{0}}
		=0, \quad 
		-\frac{B_{3}}{\nu^{m+1}}
		+\frac{B_{6}}{\nu^{N-1}\rho^{N-2}_{0}}
		=0.
	\end{equation*}

	Denote  
	\begin{equation*}
		\mathbb M:=\Big[r_{0} \mu-\displaystyle\frac{1}{\mu^{\bar{\theta}}},r_{0} \mu+\displaystyle\frac{1}{\mu^{\bar{\theta}}}\Big]\times \Big[\rho_{0} \mu-\displaystyle\frac{1}{\mu^{\bar{\theta}}},\rho_{0} \mu+\displaystyle\frac{1}{\mu^{\bar{\theta}}}\Big]\times 
		\Big[\lambda_{0}-\frac{1}{\mu^{\frac{3}{2}\bar\theta}},\lambda_{0}+\frac{1}{\mu^{\frac{3}{2}\bar\theta}}\Big]
		\times\Big[\nu_{0}-\frac{1}{\mu^{\frac{3}{2}\bar\theta}},\nu_{0}+\frac{1}{\mu^{\frac{3}{2}\bar\theta}}\Big],
	\end{equation*}
	where $0<\bar\theta <\min\{\theta,N-2-m-\tau_1\}/\max\{4, 2m-2\}$.
	We find a zero of the map $ (r,\lambda,\rho,\nu)\mapsto (\widetilde{c}_1, \widetilde{c}_2, \widetilde{d}_1, \widetilde{d}_2)$ on $\mathbb M$.
	
	For $(r,\lambda,\rho,\nu)\in \mathbb M$, there holds
	\begin{equation*} 
		\begin{aligned}
			\widetilde	c_{2} 
			=&-\frac{B_{1}}{\lambda^{m+1}\mu^{m}}
			+\frac{B_{5}}{\lambda^{N-1}r^{N-2}_{0}}\frac{1}{\mu^{m}}
			+O\Big(\frac{1}{\mu^{m+2\bar\theta}}\Big),\\
			\widetilde c_{1} 
			=&-A_{1}\frac{r-\mu r_{0}}{\mu^{m}}
			+O\Big(\frac{1}{\mu^{m+2\bar\theta}}\Big).
		\end{aligned}
	\end{equation*}
	If $\lambda=\lambda_{0}-\frac{1}{\mu^{\frac{3}{2}\bar\theta}}$, then
	\begin{equation*} 
		\begin{aligned}
			\widetilde	c_{2} 	=&-\frac{B_{1}}{(\lambda_{0}-\frac{1}{\mu^{\frac{3}{2}\bar\theta}})^{m+1}\mu^{m}}
			+\frac{B_{5}}{(\lambda_{0}-\frac{1}{\mu^{\frac{3}{2}\bar\theta}})^{N-1}r^{N-2}_{0}}\frac{1}{\mu^{m}} +O\Big(\frac{1}{\mu^{m+2\bar\theta}}\Big)\\
			=&-\frac{B_{1}}{\lambda^{m+1}_{0}\mu^{m}}
			+\frac{B_{5}}{\lambda^{N-1}_{0}r^{N-2}_{0}}\frac{1}{\mu^{m}}+(N-m-2)\frac{B_{1}}{\lambda^{m+2}_{0}\mu^{m+\frac{3}{2}\bar\theta}}
			+O\Big(\frac{1}{\mu^{m+2\bar\theta}}\Big)\\
			=&(N-2-m)\frac{B_{1}}{\lambda^{m+2}_{0}\mu^{m+\frac{3}{2}\bar\theta}}
			+O\Big(\frac{1}{\mu^{m+2\bar\theta}}\Big)>0.
		\end{aligned}
	\end{equation*}	
	If $r=r_{0}\mu- \displaystyle\frac{1}{\mu^{\bar{\theta}}}$, then
	\begin{equation*} 
		\begin{aligned}
			\widetilde{c}_1=-A_{1}\frac{1}{\mu^{m}}\Big(r_{0}\mu- \displaystyle\frac{1}{\mu^{\bar{\theta}}}-\mu r_{0}\Big)+O\Big(\frac{1}{\mu^{m+2\bar\theta}}\Big)=A_{1}\frac{1}{\mu^{m+\bar{\theta}}}+O\Big(\frac{1}{\mu^{m+2\bar\theta}}\Big)>0,
		\end{aligned}
	\end{equation*}
	Similarly, we can get  $\widetilde	c_{2} <0$ if $\lambda=\lambda_{0}+\frac{1}{\mu^{\frac{3}{2}\bar\theta}}$;
	$\widetilde{d}_2>0$ if 
	$\nu=\nu_{0}-\frac{1}{\mu^{\frac{3}{2}\bar\theta}}$;  $\widetilde{d}_2<0$
	if $\nu=\nu_{0}+\frac{1}{\mu^{\frac{3}{2}\bar\theta}}$.
	$\widetilde{c}_1<0$
	if $r=r_{0}\mu+ \displaystyle\frac{1}{\mu^{\bar{\theta}}}$;
	$\widetilde d_{1}>0$ 
	if $\rho=\rho_{0}\mu- \displaystyle\frac{1}{\mu^{\bar{\theta}}}$ 
	$\widetilde d_{1}<0$; if $\rho=\rho_{0}\mu+ \displaystyle\frac{1}{\mu^{\bar{\theta}}}$.
	
	Then by Miranda's theorem \cite[Theorem 3.1]{MJ13},  the map $ (r,\lambda,\rho,\nu)\mapsto (\widetilde{c}_1, \widetilde{c}_2, \widetilde{d}_1, \widetilde{d}_2)$   has at least one zero point in $\mathbb M$.

	Finally, to see $u_{1,k},u_{2,k}$ are non-negative,
	by $(u_{1,k},u_{2,k})\in \Lambda_{k}$, we know $\mathrm{supp}\, u^{-}_{1,k} \subset \R^N\setminus P_2$ and $\mathrm{supp}\, u^{-}_{2,k} \subset \R^N\setminus Q_2$.
	Testing  the first equation in \eqref{subsolution} by $u^{-}_{1,k}$,   we have 
	\begin{equation*}
		\lVert \nabla u^{-}_{1,k}\rVert^{2}_{L^{2}}  \leq C\lVert u^{-}_{1,k}\rVert^{2^*}_{L^{2^*}}\leq\lVert u_{1,k}\rVert^{2^*-2}_{L^{2^*}(\R^N\setminus P_2)}\lVert \nabla u^{-}_{1,k}\rVert^{2}_{L^{2}}.
	\end{equation*}
	By Lemma \ref{lem7}, $\lVert u_{1,k}\rVert^{2^*-2}_{L^{2^*}(\R^N\setminus P_2)}\to 0$ as $k\to+\infty$.
	So $u^{-}_{1,k}=0$. Similarly, we also get $u^{-}_{2,k}=0$, concluding the proof. 
\end{proof}

\begin{proof}[\textbf{Proof of Theorem \ref{th3}}]
	In the outer region, 
	from Lemma \ref{deadcore}, we can obtain Theorem \ref{th3}. Using the operator $S$, 
	in the inner region, we also obtain the solution $(u_{1,k},u_{2,k})$
	defined in Theorem \ref{th2} through reduction method. Thus, we complete the proof.
\end{proof}

\appendix
\section{Basic Estimates}
In this section, we collect several fundamental and useful estimates that will be used repeatedly throughout the paper. The Lemmas \ref{lemB1}--\ref{lem2} stated below are adapted from \cite{WY10}.
\begin{Lem}\label{lemB1}
	For any constant $0<\tau\leq\min\{\alpha_{1},\alpha_{2}\}$, there is a constant $C>0$ such that
	\begin{equation*}	
		\frac{1}{(1+\lvert y-x^{i}\rvert)^{\alpha_{1}}}\frac{1}{(1+\lvert y-y^{i}\rvert)^{\alpha_{2}}}\leq \frac{C}{\lvert x^{i}-y^{i}\rvert^{\tau}}\Big(\frac{1}{(1+\lvert y-x^{i}\rvert)^{\alpha_{1}+\alpha_{2}-\tau}}+\frac{1}{(1+\lvert y-y^{i}\rvert)^{\alpha_{1}+\alpha_{2}-\tau}}\Big).
	\end{equation*}
\end{Lem}

\begin{Lem} \label{lemB2} 
	For any constant $\sigma>0$ with $\sigma\neq N-2$, there is a constant $C>0$ such that
	\begin{equation*}	
		\int_{\mathbb R^{N}} \frac{1}{\lvert y-z\rvert^{N-2}} \frac{1}{(1+\lvert z\rvert)^{2+\sigma}} dz\leq \frac{C}{(1+\lvert y\rvert)^{\min\{\sigma,N-2\}}}.
	\end{equation*}
\end{Lem}

\begin{Lem}\label{lem2}
	Let $r\in [r_{0} \mu-1,r_{0}\mu +1],~\rho \in [\rho_{0} \mu-1,\rho_{0} \mu+1]$. Then  there exists $\theta>0$ small enough such that as $k\to+\infty$,
	\begin{equation*}
		K_1(\frac{y}{\mu})U^{2^*-1}_{r,\lambda}-\sum_{i=1}^{k}U^{2^*-1}_{x^{i},\lambda}=O\Big(\frac{1}{\mu^{\frac{m}{2}+\theta}}\Big)\sum_{i=1}^{k}\frac{1}{(1+\lvert y-x^{i}\rvert)^{\frac{N}{2}+2}}.
	\end{equation*}
\end{Lem}
From \cite{WY10}, we also obtain that
	for $\tau_{1}=\frac{N-2-m}{N-2}$ and $\tau>1$, there is some $C > 0$ independent of $k$ such that
	\begin{equation} \label{key3'}
		\sum_{i=2}^{k}\frac{1}{\lvert x^{i}-x^{1}\rvert^{\tau}} +\sum_{i=2}^{k}\frac{1}{\lvert y^{i}-y^{1}\rvert^{\tau}}\leq C\left(\frac{k}{\mu}\right)^\tau\leq    C\mu^{-\frac{m\tau}{N-2}}.
	\end{equation}
	\begin{equation} \label{key3}
		\sum_{i=2}^{k}\frac{1}{\lvert x^{i}-x^{1}\rvert^{\tau_{1}}}\leq C~\text{and}~\sum_{i=1}^{k}\frac{1}{(1+\lvert y-x^{i}\rvert)^{\tau_{1}}}\leq C.
	\end{equation}
	Moreover,
	\begin{equation} \label{key3.}
		\sum_{i=2}^{k}\frac{1}{\lvert y^{i}-y^{1}\rvert^{\tau_{1}}}\leq C~\text{and}~\sum_{i=1}^{k}\frac{1}{(1+\lvert y-y^{i}\rvert)^{\tau_{1}}}\leq C.
	\end{equation}
	
\smallskip

We also use the following important basic estimates in this article.
By H\"older's inequality,  for $\alpha\geq \tau_1$, $\beta_1>1$, there holds
\begin{equation}\label{tau1a}
	\begin{aligned}
		\left(\sum_{i=i_0}^{k}\frac{1}{(1+\lvert y-x^{i}\rvert)^{\alpha}}\right)^{\beta_1}  =& \left(\sum_{i=i_0}^{k}\frac{1}{(1+\lvert y-x^{i}\rvert)^{\alpha-\frac{\beta_1-1}{\beta_1}\tau_1}}\frac{1}{(1+\lvert y-x^{i}\rvert)^{ \frac{\beta_1-1}{\beta_1}\tau_1}}\right)^{\beta_1}\\
		\leq & \sum_{i=i_0}^{k}\frac{1}{(1+\lvert y-x^{i}\rvert)^{\alpha\beta_1- (\beta_1-1) \tau_1}}\left(\sum_{i=i_0}^{k}\frac{1}{(1+\lvert y-x^{i}\rvert)^{ \tau_1}}\right)^{\beta_1-1}\\
		\leq &C\sum_{i=i_0}^{k}\frac{1}{(1+\lvert y-x^{i}\rvert)^{\alpha\beta_1- (\beta_1-1) \tau_1}},
	\end{aligned}
\end{equation}
where $i_0=1$ or $2$.
For $\alpha >0$,
define 
\begin{equation}\label{varpir}
	\varpi_{r,\alpha}(y):= \sum_{j=1}^k\frac{1}{ (1 + \lvert y -x^j \rvert)^{\alpha}}.
\end{equation}

With this notation, the inequality \eqref{tau1a} implies  
\begin{equation}\label{varpi2}
	\varpi_{r,\alpha}^{\beta_1} \leq \varpi_{r,\alpha\beta_1-(\beta_1-1)\tau_1} \quad \text{for } \alpha\geq \tau_1,~ \beta_1 >1.
\end{equation}
Note that $\varpi_{r,\alpha}$ is decreasing with respect to $\alpha$.
Moreover,
since $y\in \R^N\setminus P_1$, there holds   
\begin{equation}\label{varpi1}
	\varpi_{r,\alpha}\leq \sigma_k^{\alpha'-\alpha} \varpi_{r,\alpha'}, \quad\text{for } 0<\alpha'<\alpha.
\end{equation}

We define  
\begin{equation}\label{varpirho}
	\varpi_{\rho,\alpha}(y) := \sum_{j=1}^k \frac{1}{(1 + |y - y^j|)^\alpha},  
\end{equation}
and  
\begin{equation}\label{varpi1.}
	\varpi_{\alpha}(y) := \varpi_{r,\alpha}(y) + \varpi_{\rho,\alpha}(y).
\end{equation}
The functions $\varpi_{\rho,\alpha}$ and $\varpi_{\alpha}$ inherit the same properties as $\varpi_{r,\alpha}$—namely, \eqref{varpi2} and \eqref{varpi1}—in $\mathbb{R}^N \setminus Q_1$ and $\mathbb{R}^N \setminus (P_1 \cup Q_1)$, respectively.

\begin{Lem}\label{lemB1.}
	There holds
	\begin{equation}\label{10'}
		\begin{aligned}
			&-\Delta \varpi_{r,N-2} \geq C \sigma_k^{1-\frac{4\tau_1}{N-2}}
			\varpi_{r,N-2}^{2^*-1}, &\text{ in }\mathbb R^{N}\setminus P_{1},\\
			&-\Delta \varpi_{\rho,N-2} \geq C \sigma_k^{1-\frac{4\tau_1}{N-2}}
			\varpi_{r,N-2}^{2^*-1}, &\text{ in }\mathbb R^{N}\setminus Q_{1},\\
			&-\Delta  \varpi_{N-2}
			\geq  C \sigma_k^{1-\frac{4\tau_1}{N-2}}\varpi_{N-2}^{2^{*}-1}
			,  &\text{ in }\mathbb R^{N}\setminus (P_{1}\cup Q_1).
		\end{aligned}	
	\end{equation}
	For $\alpha\in[\frac{N}{2}-1+\tau_1, N-2)$,
	\begin{equation}\label{10}
		\begin{aligned}
			-\Delta  \varpi_{r,\alpha}
			\geq& C  \sigma_k^{4-\frac{4\alpha}{N-2}}\varpi_{r,N-2}^{2^{*}-2}\varpi_{r,\alpha}+ \frac12 \alpha(N-2-\alpha) \varpi_{r,\alpha+2} 
			, &\text{ in }\mathbb R^{N}\setminus P_{1},\\
			-\Delta  \varpi_{\rho,\alpha}
			\geq& C  \sigma_k^{4-\frac{4\alpha}{N-2}}\varpi_{\rho,N-2}^{2^{*}-2}\varpi_{\rho,\alpha}+ \frac12 \alpha(N-2-\alpha) \varpi_{\rho,\alpha+2} 
			, &\text{ in }\mathbb R^{N}\setminus Q_{1}.
		\end{aligned}	
	\end{equation}

\end{Lem}

\begin{proof}
	For $\alpha\in (0, N-2]$, $y\in\R^N\setminus\{0\}$, we have 
	\begin{equation}\label{Delta}
		-\Delta  (1 + \lvert y \rvert)^{-\alpha}  = \alpha (1 + \lvert y \rvert)^{-\alpha - 2} \left( \frac{N - 1}{\lvert y \rvert} + N - 2- \alpha  \right).
	\end{equation}
	Therefore,
	\[
	-\Delta \varpi_{r,N-2} \geq (N - 1)(N-2)
	\varpi_{r,N+1}.
	\]
	On the other hand, by \eqref{varpi2} and \eqref{varpi1}, 
	\[
	\varpi_{r,N-2}^{2^*-1}
	\leq C \varpi_{r,N+2-\frac{4\tau_1}{N-2}}\leq  C
	\sigma_k^{\frac{4\tau_1}{N-2}-1}\varpi_{r,N+1}.
	\]
	Hence, we obtain the first inequality of \eqref{10'}.
	
	Now  let $\alpha\in [\frac{N}{2}-1+\tau_1, N-2)$.
	Then by \eqref{Delta},
	\[
	-\Delta \varpi_{r,\alpha}\geq \alpha(N-2-\alpha) \varpi_{r,\alpha+2}.
	\]
	By \eqref{varpi2} and \eqref{varpi1} again and since $\frac{4(\alpha-\tau_1)}{N-2}\geq 2$,
	\[
	\varpi_{r,N-2}^{2^*-2} \varpi_{r,\alpha} \leq \sigma_k^{\frac{4\alpha}{N-2}-4} \varpi_{r,\alpha}^{2^*-1}
	\leq C \sigma_k^{\frac{4\alpha}{N-2}-4} \varpi_{r,\alpha+\frac{4(\alpha-\tau_1)}{N-2}}\leq  C
	\sigma_k^{\frac{4\alpha}{N-2}-4}\varpi_{r,\alpha+2}.
	\]
	Therefore, we get the first inequality of \eqref{10}. As our arguments only depends on the properties \eqref{varpi2} and \eqref{varpi1}, other inequalities in \eqref{10'}, \eqref{10} follow in a same way.
\end{proof}

The following inequality will be used.
\begin{Lem}\label{lema6}
	Let  $m_0> 2$. Then there is $C>0$ such that for $a, b\in\R$,
	\[
	\left| |a+b|^{m_0-2}(a+b) -|a|^{m_0-2}a -(m_0-1)|a|^{m_0-2}b \right|\leq C\begin{cases} 
		|b|^{m_0-1} &\text{if } 2< m_0 <3,\\
		|a|^{m_0-3} b^2 + |b|^{m_0-1} &\text{if } m_0\geq 3.
	\end{cases}
	\]
\end{Lem}

\noindent\textbf{Acknowledgments}:
This research was supported by the National Natural Science Foundation of China (Grant Nos. 12271539, 12371107).

\end{document}